\documentclass[a4paper,11pt]{article}
\usepackage[utf8]{inputenc}

\usepackage[title]{appendix} 

\usepackage[pdfencoding=auto, psdextra]{hyperref}
\hypersetup{
    colorlinks=true,
    }
\usepackage{enumitem} % labels in enumerate
\usepackage{authblk} % Authors and affiliations
\usepackage{caption}
\usepackage{multirow} % for tables
\usepackage{booktabs} % for tables

\usepackage{amssymb}
\usepackage{amsmath}
\usepackage{amsthm}
\usepackage{bm} % bold math symbols eg greek letters
\usepackage{mathtools}
\usepackage{stmaryrd} % For double brackets
\usepackage{setspace} % horizontal spacing arrays
\usepackage{mathabx}
\usepackage{cleveref} % pretty citing of subequations

\crefrangeformat{subequation}%
      {(#3#1#4--#5\crefstripprefix{#1}{#2}#6)}

\theoremstyle{plain}
\newtheorem{theorem}{Theorem}[section]
\newtheorem{prop}[theorem]{Proposition}
\newtheorem{lemma}[theorem]{Lemma}

\newtheorem{hyp}[theorem]{Assumption}

\newtheorem{rque}[theorem]{Remark}

\newcommand{\N}{\mathbb{N}}
\newcommand{\R}{\mathbb{R}}
\def\P{\mathbb{P}}
\newcommand{\E}{\mathbb{E}}
\newcommand{\M}{\mathcal{M}}
\newcommand{\D}{\mathbb{D}}
\newcommand{\tvnorm}{\|\cdot\|_{TV}}
\newcommand{\e}{\mathrm{e}}
\newcommand{\inftynorm}[1]{\left\lVert#1\right\rVert_{\infty}}
\newcommand{\C}{\mathcal{C}} 
\newcommand{\ind}{\mathbf{1}}
\newcommand{\setind}[1]{\mathbf{1}_{\left\{#1\right\}}}
\newcommand{\bbrackets}[2]{\llbracket #1, #2 \rrbracket}
\newcommand{\angles}[2]{\langle #1, #2 \rangle}
\newcommand{\OmegaG}{\Omega_\mathcal{G}}
\newcommand{\OmegaCvgce}{\OmegaG^*}
\newcommand{\config}{\mathbb{S}}
\newcommand{\nmax}{n_{\max}}
\newcommand{\bfs}{\mathbf{s}}
\newcommand{\bfi}{\mathbf{i}}
\newcommand{\bfn}{\mathbf{n}}
\newcommand{\bfv}{\mathbf{v}}
\newcommand{\bfB}{\mathbf{B}}
\newcommand{\B}{\mathbb{B}}
\newcommand{\I}{\mathcal{I}}
\newcommand{\A}{\mathcal{A}}
\newcommand{\tilzeta}{\tilde{\zeta}}
\newcommand{\tileta}{\tilde{\eta}}

\title{Large population limit for a multilayer \emph{SIR} model including households and workplaces}
\author[1,2]{Madeleine Kubasch}
\affil[1]{\footnotesize Centre de math\'ematiques appliqu\'ees (CMAP), Ecole Polytechnique, 91128 Palaiseau, France}
\affil[2]{\footnotesize MaIAGE, INRAE, Universit\'e Paris-Saclay, 78350 Jouy-en-Josas, France}
\date{28 May 2026}

\begin{document}

\maketitle 

\begin{abstract}
  We study a multilayer \emph{SIR} model with two levels of mixing, namely a global level which is uniformly mixing, and a local level with two layers distinguishing household and workplace contacts, respectively. 
  We establish the large population convergence of the corresponding stochastic process. 
  For this purpose, we use an individual-based  model whose state space specifies the remaining infectious period length for each infected. 
  This allows to deal with the  natural correlation of the epidemic states of individuals whose household and workplace share a common infected. 
  In a general setting where a non-exponential distribution of infectious periods may be considered, convergence to the unique deterministic solution of a measure-valued equation is obtained. 
  In the particular case of exponentially distributed infectious periods, we show that it is possible to further reduce the obtained deterministic limit, leading to a closed, finite dimensional dynamical system capturing the epidemic dynamics. 
  This model reduction subsequently is studied from a numerical point of view. 
  We illustrate that the dynamical system derived from the large population approximation is a pertinent model reduction when compared to simulations of the stochastic process or to an alternative edge-based compartmental model, both in terms of accuracy and computational cost. 
    \newline\newline
    \textbf{Code availability.} \url{https://github.com/m-kubasch/household-workplace-model}
\end{abstract}

\maketitle

%%%%%%%%%%%%%%%%%%%%% SECTION %%%%%%%%%%%%%%%%%%%%%
\section{Introduction}
\label{sec:intro}

Epidemic spread depends by essence on the way individuals interact with one another. As a consequence, models have been developed which take into account main features of real-life contacts, for instance through contact networks \cite{kissMathematicsEpidemicsNetworks2017}. In particular, some attention has been drawn to studying clustered networks, as clustering has a strong impact on epidemic spread \cite{hebert-dufresnePropagationDynamicsNetworks2010, volzEffectsHeterogeneousClustered2011} which is intimately related to the way clustering is achieved within the network \cite{houseInsightsUnifyingModern2011}. A particular form of clustering consists in the presence of entirely connected small social structures, such as households, workplaces, or schools, which exist in addition to random contacts in the general population.  
This kind of population structure is captured by models with two levels of mixing \cite[and references therein]{brittonChapter3General2019}, which are related to efficient control measures. Indeed, both COVID-19 and influenza epidemics illustrate the pertinence of teleworking and school closures \cite{mendez-britoSystematicReviewEmpirical2021, simoySociallyStructuredModel2021, lucaImpactRegularSchool2018}, and models explicitly distinguishing different contact types are well suited to simulate the impact of these measures \cite{dilauroImpactContactStructure2021}. However, precisely understanding the impact on disease propagation of the way individuals are organized in households and workplaces is not straightforward \cite{bansayeEpidemiologicalFootprintContact2024}.

This motivates the study of models with several levels of contact, which as we shall see lead to interesting mathematical issues due to their multiscale population structure. Pellis \emph{et al.} \cite{pellisThresholdParametersModel2009} have proposed a model with two levels of mixing structured in three layers of contacts: households, workplaces and the general population.
This household-workplace model has already been studied to some extent, establishing for instance the epidemic growth rate \cite{pellisEpidemicGrowthRate2011, bansayeEpidemiologicalFootprintContact2024} and several reproduction numbers \cite{ballReproductionNumbersEpidemic2016}. In particular, the $R_0$ used throughout this paper was introduced in \cite{pellisThresholdParametersModel2009}. Its definition is based on a branching approximation at the beginning of the epidemic, when only a negligible fraction of individuals is infected in an otherwise large susceptible population. The branching process describes the dynamics of the population of infected individuals, which is structured by origin of infection: in the household, workplace or general population. The spectral radius of the associated mean reproduction matrix yields $R_0$, and its unique positive eigenvector gives the proportions of infections per contact layer. In particular, computing these eigenelements is numerically efficient,  since they do not rely on simulations of the epidemic trajectories \cite[working package]{bansayeEpidemiologicalFootprintContact2024}.
\smallskip

One drawback of the household-workplace model is its complexity, both mathematically and numerically. Indeed, its epidemic trajectories are not simple to analyse, due to correlations arising as soon as an individual may belong to several small contact structures at once. Also, simulations require a significant amount of computation time, especially when considering larger population sizes. As a consequence, it is of interest to develop reduced models, which may be more prone to theoretical studies and/or numerical exploration. In particular, large population approximations of stochastic models have proven fruitful to achieve such model reductions in many contexts, among which epidemics on random graphs. 

Historically, the standard \emph{SIR} (susceptible, infected, recovered) model developed by Kermack and McKendrick itself corresponds to the large population limit of the uniformly mixing stochastic \emph{SIR} model. In the Markovian setting, the convergence of the stochastic model to its deterministic limit can be established using classical results on the convergence of finite type density-dependent Markov jump processes, \emph{e.g.} \cite{anderssonDensityDependentJump2000}. When infectious periods are not restricted to being exponentially distributed, the large population convergence of the stochastic model to the unique deterministic solution of a system of integral equations can also be obtained \cite{kurtzEpidemicModels1981, forienRecentAdvancesEpidemic2022}. For more complex contact networks however, it often is challenging to propose closed systems of equations correctly describing the epidemic dynamics. 

A well-understood case is the \emph{SIR} model on the configuration graph, for which a reduced model referred to as edge-based compartmental model (EBCM)  \cite{volzSIRDynamicsRandom2008, millerNotePaperErik2011} has been proven to be the large population limit of the underlying stochastic model \cite{decreusefondLargeGraphLimit2012, jansonLawLargeNumbers2014}. Since then, the equivalence with other reduced models has been established under appropriate assumptions \cite{houseInsightsUnifyingModern2011, wilkinsonRelationshipsMessagePassing2017, jacobsenLargeGraphLimit2018, kissNecessarySufficientConditions2023}, and the EBCM formalism has been extended to related models \cite{sherborneMeanfieldModelsNonMarkovian2018, jacobsenLargeGraphLimit2018}.  The configuration graph is a favourable setting for this analysis thanks to the absence of clustering in the large graph limit, which however also constitutes a major limitation, prohibiting for instance the existence of household-like structures.

Some attention has thus been drawn to models where each individual belongs to a random number of fully connected subgraphs (\emph{cliques}) of the same type, hence being closely related to the household-workplace model. Several reduced models have been proposed, including an EBCM \cite{volzEffectsHeterogeneousClustered2011, st-ongeHeterogeneousTransmissionGroups2023, hebert-dufresnePropagationDynamicsNetworks2010}. To our knowledge, the convergence of the underlying stochastic model to the proposed reduced model has not been established in any of these settings. Notice that these models share a major common point, which will also hold true in our setting: they focus on the epidemic at the level of structures, as they keep track of the proportions of cliques containing a certain number of susceptibles and infected, leading to high-dimensional dynamical systems for larger clique sizes. 

When considering two levels of mixing, the first model for which a large population limit has been determined is the household model, which assumes a uniformly mixing general population and that each individual belongs to exactly one household \cite{houseDeterministicEpidemicModels2008}, thus being a special case of the household-workplace model. Here, the stochastic model can again be formalized as a finite type density-dependent Markov jump process, ensuring the large population convergence to the deterministic model. If either of these two assumptions is relaxed, \emph{e.g.} considering a configuration graph at the global level \cite{dilauroImpactContactStructure2021, maEffectiveDegreeHousehold2013} or individuals belonging to several households \cite{barnardEdgeBasedCompartmentalModelling2018}, reduced models have been proposed, but without rigorous derivation from stochastic models. In particular, the case of the household-workplace model is not covered, and the only reduced models proposed so far approach the epidemic dynamics using well calibrated uniformly mixing models \cite{bansayeEpidemiologicalFootprintContact2024, delvallerafoDiseaseDynamicsMean2021}. In particular, numerical comparison of the uniformly mixing \emph{SIR} model and the household-workplace model shows that the discrepancy between the two models depends on the epidemic scenario under consideration \cite{bansayeEpidemiologicalFootprintContact2024}. Generally speaking, assuming uniform mixing leads to higher epidemic peak and final size than the household-workplace model, particularly for epidemics spreading mainly at the local level. Thus, while uniformly mixing reduced models are capable of capturing some characteristics of the household-workplace epidemic under appropriate circumstances, they do not allow in general for an accurate prediction of the epidemic dynamic over time. 
\smallskip

As a consequence, in this paper, we will study the large population limit of the multilayer \emph{SIR} model with households and workplaces. In order to do so, we will formalize the model in a finite population as an individual-based stochastic process, and establish that this sequence of processes converges in law when the size of the population grows to infinity. This allows to identify a new model reduction, and establishes that it is asymptotically exact. Besides, it paves the way for more quantitative estimates on this approximation.

Notice here that each infected individual correlates the epidemic spread in his household and workplace, being infectious for exactly the same period of time in both structures. In order to deal with this dependence, the duration of infectious periods will explicitly be taken into account in the mathematical representation of the model. This difficulty actually arises as soon as one considers the probability of an individual belonging to several cliques at once, whether they are of different types or not, and we refer to \cite{ballEpidemicsRandomIntersection2014} where a similar approach has been developed for branching approximations. 
Let us emphasize that this model formulation allows to immediately consider a wide range of infectious period length distributions instead of being restricted to the Markovian case, which is a pertinent generalisation for many epidemic models \cite{sherborneMeanfieldModelsNonMarkovian2018, forienRecentAdvancesEpidemic2022, fengEpidemiologicalModelsNonExponentially2007, lloydRealisticDistributionsInfectious2001}.

The model will be represented by a measure-valued process mixing discrete and continuous components. More precisely, we establish the convergence of the individual-based process to the unique solution of an explicit measure-valued equation. In the particular case where this distribution is exponential, it is possible to go one step further and reduce the epidemic dynamics to a closed, finite dimensional dynamical system which is similar in spirit to reductions proposed in related settings \cite{houseDeterministicEpidemicModels2008, st-ongeHeterogeneousTransmissionGroups2023}. 
\bigskip

The present paper is structured as follows. Section \ref{sec:model-presentation} introduces the individual-based model, and Section \ref{sec:results} subsequently presents the convergence results in detail. Section \ref{sec:numeric} is devoted to numerical aspects. We first illustrate that the obtained dynamical system is in good accordance with stochastic simulations, discuss its implementation and examine its computational cost in terms of computation time compared to stochastic simulations. Next, we confront our reduced model to an alternative model reduction which we obtain using the EBCM formalism. Finally, Section \ref{sec:proofs} contains the proofs of our results. 
\bigskip

Before proceeding, let us introduce some notations that will be used throughout the paper. For any integers $n \leq m$, we write $\bbrackets{n}{m} = \{n, \cdots, m\}$. For a measurable space $(E, \mathcal{E})$, let 
%$\M_P(E)$ be the set of point measures, 
$\M_F(E)$ be the set of finite measures  and $\M_1(E)$ the set of probability measures on $E$. 
%We define $\M_{P,1}(E) = \M_P(E) \cap \M_1(E)$ the set of punctual probability measures on $E$. 
For a measure $\mu$ on $E$ and a suitable function $f$ (either non-negative or belonging to $L^1(\mu)$), let $\langle \mu, f \rangle = \int_E f d\mu$. Also, for $x \in E$, $\delta_x$ designates the Dirac measure at point $x$. Further, for any metric space $E$ and any integer $m$, let $\C(E,\R^m)$ be the set of continuous functions $f: E \to \R^m$. Similarly, $\C_b(E,\R^m)$ is defined as the subset of bounded functions $f \in \C(E,\R^m)$. Finally, the space $\C^1_b(E, \R^m)$ designates the set of bounded functions $f: E \to \R^m$ such that $f$ is differentiable and its differential is continuous and bounded.

%%%%%%%%%%%%%%%%%%%%% SECTION %%%%%%%%%%%%%%%%%%%%%
\section{Presentation of the model}
\label{sec:model-presentation}

Let us begin by introducing the epidemic model of interest, in two successive steps. At first, a general model description is yielded, which corresponds to a more intuitive presentation of the model, before stating the mathematical model in detail using a measure-valued stochastic differential equation. 

\subsection{General presentation of the model}
\label{sec:general-model}

Let us start by describing the population structure of interest. Consider a population of $K$ individuals, which each are assigned one household and one workplace.

More precisely, such a population structure can be obtained as described in \cite{bansayeEpidemiologicalFootprintContact2024}. Suppose that households and workplaces are of size at least one and at most $\nmax$. Consider distributions $\pi^H$ and $\pi^W$ on $\bbrackets{1}{\nmax}$. These distributions correspond to the large population limit of household and workplace size distributions, in the sense that in an infinite population, a proportion $\pi^H_j$ of households would be of size $j$, while $\pi^W$ would play a similar role for workplaces. In such an infinite population, the average household and workplace sizes, respectively $m_H$ and $m_W$, would be given for $X \in \{H,W\}$ by 
\begin{equation*}
    m_X = \sum_{j=0}^{\nmax} j \pi^X_j.
\end{equation*}

On a probability space $(\OmegaG, \P_\mathcal{G}, \mathcal{F}_\mathcal{G})$, we construct a sequence $(\mathbf{G}^K)_{K \geq 1}$ of this random population structure as follows. For $K \geq 1$, let $k \in \bbrackets{0}{K}$ be the number of individuals who are not yet member of a household. While $k > 0$, choose a size $\tilde{n}$ according to $\pi^H$, independently from the household sizes that were chosen during previous steps. The newly uncovered household is then of size $n = \min(\tilde{n},k)$, and $n$ individuals out of the $k$ remaining ones are picked uniformly at random to assemble this new household. Consequently, it remains to update $k$ to $k - n$. The process stops as soon as $k=0$, as all individuals then belong to a household. Finally, this process is repeated independently for workplaces, using $\pi^W$ instead of $\pi^H$.

\bigskip

It remains to describe the way the disease spreads in the population. The epidemic model considered here is an extension of the standard \emph{SIR} model. At each time, each individual is either \emph{susceptible} if he has never encountered the disease and may be contaminated; \emph{infected} if the individual is currently infectious, in which case he may transmit the disease to other susceptibles; or \emph{recovered}, once the infectious period is over, in which case the individual has become immune against the disease.

The disease is transmitted among individuals as follows. Within each household, each workplace and the general population, uniform mixing is assumed. Given the current population state, the waiting times until the next infection event within each layer are exponentially distributed. For households, we consider a \emph{one-to-one} contact rate $\lambda_H$, meaning that whenever there are $s$ susceptibles and $i$ infected within a household, the next infectious contact occurs at rate $\lambda_H s i$. Similarly, another one-to-one contact rate $\lambda_W$ is associated to workplace contacts. Within the general population, a \emph{one-to-all} contact rate $\beta_G$ is considered: when there are $s$ susceptibles and $i$ infected within a population of size $K$, infectious contacts occur at rate $\frac{\beta_G}{K}si$. Indeed, each given encounter within the general population becomes less likely when the population size $K$ increases. 

Finally, infected individuals remain infectious for a period of time which is independent from $(\mathbf{G}_K)_{K \geq 1}$ and distributed according to a probability distribution $\nu$ on $\R_+$, which we assume to be absolutely continuous with regard to the Lebesgue measure. Once they recover, they are supposed to be immune against the disease from there on. In particular, if $\nu$ is an exponential distribution, this corresponds to the Markovian \emph{SIR} model. 

\subsection{The epidemic model at the level of households and workplaces}
\label{sec:model}

As we aim at investigating the large population limit of this model, we choose to enrich the population description as to obtain a closed Markov process. This corresponds to a favourable mathematical setting, as it allows us to use the associated martingale problem.
In order to do so, we will represent the population in terms of an individual-based model, also known as agent-based model. More precisely, it seems natural to focus on entire structures, \emph{i.e.} households and workplaces, which are characterized by their size and the number of susceptible and infected individuals they contain. Indeed, this point of view has already proven useful for deriving reproduction numbers for related models \cite{ballReproductionNumbersEpidemic2016}, as well as the epidemic growth rate of the household-workplace model \cite{pellisEpidemicGrowthRate2011}. However, this is not enough to obtain a closed system of Markovian dynamics. The problem is that each infected individual correlates the spread of the epidemic within his household and his workplace, leading to an intricate correlation network. In order to circumvent this difficulty, similarly to \cite{ballEpidemicsRandomIntersection2014}, we will thus further characterize each structure by the infectious periods of its infected members. Adopting this point of view is key, as it allows to handle both the progressive discovery of the graph as the epidemic spreads, and the correlations arising from infected individuals, without explicitly keeping in memory the discovered graph.
\bigskip

Let $\omega \in \OmegaG$. For a population of size $K \geq 1$, let $K_H$ be the number of households and $K_W$ the number of workplaces in $\mathbf{G}^K(\omega)$. While $K_H$ and $K_W$ depend on $\omega$, this dependency is not specified explicitly for readability. This will also apply to the forthcoming notations. Label the $K_H$ households in an arbitrary fashion $1, \dots, K_H$. Consider the set
\begin{equation*}
E = \left\{(n,s,\tau) \in \bbrackets{1}{\nmax} \times \bbrackets{0}{\nmax} \times \R^{\nmax}: s \leq n; \; \forall j > n-s, \tau_j = 0 \right\}.
\end{equation*}
Then for $k \in \llbracket 1, K_H \rrbracket$, the $k$-th household is characterized at time $t \geq 0$ by its type 
\begin{equation*}
    x^H_k(t) = (n^H_k, s^H_k(t), \tau^H_k(t)) \in E.
\end{equation*}

The first two components of $x^H_k$ correspond respectively to the size of the household (which is constant over time), and the number of susceptible members of the household at time $t$. The third component $\tau^H_k$ is a vector containing the \emph{remaining infectious periods} of the members of the household. Indeed, at time $t$, there are $n^H_k - s^H_k(t)$ infected or recovered individuals within the household. For $j \in \llbracket 1, n^H_k - s^H_k(t) \rrbracket $, if $\tau^H_{k,j}(t) > 0$, the individual is still infectious and will remain so for $\tau^H_{k,j}(t)$ units of time. Otherwise, if $\tau^H_{k,j}(t) \leq 0$, the individual has recovered, and the recovery has occurred at time $t - |\tau^H_{k,j}(t)|$. For $j > n^H_k - s^H_k(t)$, $\tau^H_{k,j}(t)$ has no interpretation, and is set to zero for convenience in computations. In other words, the infectiousness of a previously contaminated individual with remaining infectious period $\tau$ is given by $\setind{\tau > 0}$.

Similarly, label the $K_W$ workplaces in an arbitrary order $1, \dots, K_W$. For $\ell \in \llbracket 1, K_W \rrbracket$, the $\ell$-th workplace is characterized by its type $x^W_\ell(t) = (n^W_\ell, s^W_\ell(t), \tau^W_\ell(t))$, which is defined analogously to household types.

Notice that all of these quantities depend on the population size $K$, but this dependency is omitted to simplify notations.
\bigskip

By definition, these types evolve over time. 
On the one hand, for any $X \in \{H,W\}$, for any $k \in \llbracket 1, K_X \rrbracket$ and $j \in \bbrackets{1}{\nmax}$, the $j$-th component of $\tau^X_k$ %follows a basic deterministic dynamic:
decreases linearly at unitary rate if it describes the remaining infectious period of an individual having contracted the disease at some previous time, and stays constant otherwise:
\begin{equation*}
\forall j \in \llbracket 1, \nmax \rrbracket, \; \frac{d}{dt} \tau^X_{k,j}(t) = -\ind_{\{j \leq n^X_k - s^X_k(t)\}}.
\end{equation*} 
Let $(e_j)_{1 \leq j \leq \nmax}$ denote the canonical basis of $\R^{\nmax}$. Then for any $0 \leq t \leq T$, and $x = (n,s,\tau) \in E$, we may define $\Psi(x, T, t)$ as the type of a structure at time $T$ given that it was in state $x$ at time $t$, supposing that no infections occurred in the meantime:
\begin{equation*}
    \Psi(x, T, t) = \left(n, s, \tau - \sum_{j=1}^{n-s}(T-t)e_j\right).
\end{equation*}

On the other hand, infections within each level of mixing also cause the modification of the types of the household and the workplace of newly contaminated individuals. More precisely, consider a contamination occurring at time $t$. Suppose that the newly infected belongs to the $k$-th household and $\ell$-th workplace. Let $\sigma$ be the realisation of a random variable of distribution $\nu$, which is drawn independently for each new infected. Then $x^H_k$ and $x^W_\ell$ jump from $x^H_k(t-)$ and $x^W_\ell(t-)$ to $\mathfrak{j}(x^H_k(t-), \sigma)$ and $\mathfrak{j}(x^W_\ell(t-), \sigma)$ respectively, where for any $x = (n,s,\tau) \in E$,
\begin{equation*}
    \mathfrak{j}(x, \sigma) = \left(n, s-1, \tau + \sigma e_{n-s+1}\right).
\end{equation*} 
It remains to describe how one identifies the household $k$ and workplace $\ell$ the newly infected belongs to. Let $S(t-)$ be the number of susceptibles in the population previously to the infection event. If it takes place within the general population, any susceptible individual is chosen with uniform probability to be contaminated. The newly infected thus belongs to the $k$-th household with probability $s^H_k(t-)/S(t-)$, and independently to the $\ell$-th workplace with probability $s^W_\ell(t-)/S(t-)$. Similarly, if the infection occurs within a household, only the workplace of the newly infected needs to be uncovered, and corresponds to the $\ell$-th workplace with the same probability as previously. Within-workplace infections are treated analogously.
\bigskip 

We are now ready to introduce the stochastic process  
$(\zeta^K_t = (\zeta^{H|K}_t, \zeta^{W|K}_t))_{t \geq 0}$
taking values in $\M_1(E)^2$.
%$\mathfrak{M}_{P,1} = \M_{P,1}(E) \times \M_{P,1}(E)$. 
$\zeta^{H|K}_t$ and $\zeta^{W|K}_t$ correspond respectively to the normalized counting measures associated to the distributions of household and workplace types at time $t$, \emph{i.e.} for any time $t \geq 0$ and $X \in \{H,W\}$,
\begin{equation*}
    \zeta^{X|K}_t = \frac{1}{K_X} \sum_{k = 1}^{K_X} \delta_{x^X_k(t)}.
\end{equation*}

Start by noticing that, as both household and workplace sizes are bounded, for $X \in \{H,W\}$, the following inequality holds:
\begin{equation*}
\frac{K}{\nmax} \leq K_X \leq K.
\end{equation*}
Thus, studying the asymptotic $K \to \infty$ amounts to $(K_H, K_W) \to (\infty, \infty)$. 

Observe that the number of infected individuals in a household of type $(n,s,\tau)$ is given by $i(\tau) = \sum_{k = 1}^{\nmax} \setind{\tau_k > 0}$. Then for any $t \geq 0$, 
\begin{equation*}
    I_H(t) = \frac{1}{K_H} \sum_{k=1}^{K_H} i(\tau^H_k(t))
\end{equation*}
corresponds to the average number of infected individuals per household at time $t$. Similarly, one may define $S_H(t)$ as the average number of susceptibles per household at time $t$, as well as the workplace-related quantities $I_W(t)$ and $S_W(t)$. Then 
\begin{equation}
\label{eq:def-SI}
    \forall X \in \{H,W\},\; S(t) = K_X S_X(t) \text{ and } I(t) = K_X I_X(t).
\end{equation}
Further, let $N_H$ be the average household size, which is constant over time and always equal to $K/K_H$. This leads to $I(t)/K = I_X(t)/N_X$, which will be of use in computations. Notice that we will need to check that Equation \eqref{eq:def-SI} is well posed, as equalities of the type $K_H S_H(t) = K_W S_W(t)$ technically need to be proven for the stochastic process formalizing the model. Notice that for $X \in \{H,W\}$, $S_X$ and $I_X$ actually depend on the population size $K$, which is omitted in notations for readability.

Finally, let us briefly emphasize that the partition of the population in households and workplaces is entirely conveyed by 
$\zeta^K_0  \in \M_1(E)^2$, 
%$\zeta^K_0  \in \mathfrak{M}_{P,1}$, 
as it does not vary over time. In particular, the proportions of households and workplaces of each size are supposed to correspond to those observed in $\mathbf{G}^K(\omega)$. Similarly, there are some natural constraints on $\zeta^K_0$, as $\zeta^{H|K}$ and $\zeta^{W|K}$ describe the same population, once dispatched into households, and once dispatched into workplaces. For instance, the total number of members in some epidemic state (susceptible, infected or recovered) within all households is equal to the total number of members in this state within all workplaces. Hence $K_H S_H(0) = K_W S_W(0)$ and $K_H I_H(0) = K_W I_W(0)$ almost surely. Further, at time $0$, each infected or recovered needs to have the same remaining infectious period in both his household and workplace. In other words, almost surely, 
{\footnotesize
\begin{equation}
\label{eq:hyp}
    \{\tau^H_{k,j}(0) : 1 \leq k \leq K_H,\, 1 \leq j \leq n^H_k - s^H_k(0)\} = \{\tau^W_{\ell, j}(0) : 1 \leq \ell \leq K_W,\, 1 \leq j \leq n^W_\ell - s^W_\ell(0)\}.
\end{equation}
}
Forthcoming Lemma \ref{lemme:shsw} shows that these conditions are enough to ensure that the previous characterization of $S(t)$ and $I(t)$ in terms of $S_X(t)$ and $I_X(t)$ is legitimate.

\bigskip

Before giving the proper definition of $\zeta^K$, let us introduce some necessary notations. Again, $K$-dependency is not specified explicitly in order to simplify notations.
Let
\begin{equation}
    \label{eq:UG}
    U_G = (\R_+)^3 \times \bbrackets{1}{K_H} \times \bbrackets{1}{K_W} \times \R_+,
\end{equation} 
and consider the following measure on $U_G$:
\begin{equation*}
    \mu_G(du) = \mu_G(d\bm{\theta}, dk, d\ell,d\sigma) = d\bm{\theta} \otimes \mu_{\#}(dk) \otimes \mu_{\#}(d\ell) \otimes \nu(d\sigma),
\end{equation*}
where $d\bm{\theta}$ and $\mu_{\#}$ denote the Lebesgue measure on $\R^3$ and the standard counting measure, respectively. For $t \geq 0$ and $u = (\bm{\theta}, k, \ell, \sigma) \in U_G$ where $\bm{\theta} = (\theta_1, \theta_2, \theta_3)$, let us define
\begin{equation*}
    \mathcal{I}_G(t,u) = \setind{\theta_1 \leq \frac{\beta_G}{K} S(t) I(t), \, \theta_2 \leq \frac{s^H_k(t)}{S(t)}, \, \theta_3 \leq \frac{s^W_\ell(t)}{S(t)}}.
\end{equation*}
The idea is that $\I_G$ will yield the correct rate for infection events in the general population. More precisely, the constraint on $\theta_1$ corresponds to the rate of infectious contacts at that level of mixing, while the constraints on $\theta_2$ and $\theta_3$ are related to the probability that the newly infected belongs to the $k$-th household and $\ell$-th workplace. 

Similarly, let 
\begin{equation}
    \label{eq:defU}
    U = (\R_+)^2 \times \bbrackets{1}{K_H} \times \bbrackets{1}{K_W} \times \R_+,
\end{equation} 
endowed with the measure 
\begin{equation*}
    \mu(du) = \mu(d\bm{\theta}, dk, d\ell,d\sigma) = d\bm{\theta} \otimes \mu_{\#}(dk) \otimes \mu_{\#}(d\ell) \otimes \nu(d\sigma),
\end{equation*}
where with slight abuse of notation, $d\bm{\theta}$ designates the Lebesgue measure on $\R^2$. Then for $t \geq 0$ and $u = (\bm{\theta}, k, \ell, \sigma) \in U$ where $\bm{\theta} = (\theta_1, \theta_2)$, we further introduce
\begin{equation*}
    \mathcal{I}_H(t,u) = \setind{\theta_1 \leq \lambda_H s^H_k(t)i(\tau^H_k(t)), \, \theta_2 \leq \frac{s^W_\ell(t)}{S(t)}}.
\end{equation*}
This time, the constraints on $\theta_1$ and $\theta_2$ correspond respectively to the rate of infection within the $k$-th household and the probability of the newly infected belonging to the $\ell$-th workplace. Also, for any $T \geq t \geq 0$ and $u = (\bm{\theta}, k, \ell, \sigma) \in U_G \cup U$, consider the following quantity, which will allow to keep track of the change in the household population due to an infection within the $k$-th household:
\begin{equation*}
    \Delta_H(u,T,t) = \delta_{\left(\Psi(\mathfrak{j}(x^H_k(t-),\sigma),T,t)) \right)} - \delta_{\left(\Psi(x^H_k(t-),T,t)) \right)}.
\end{equation*}
Finally, define $\mathcal{I}_W(t,u)$ and $\Delta_W(u,T,t)$ analogously: 
\begin{equation*}
    \begin{aligned}
    \mathcal{I}_W(t,u) &= \setind{\theta_1 \leq \lambda_W s^W_\ell(t)i(\tau^W_\ell(t)), \, \theta_2 \leq \frac{s^H_k(t)}{S(t)}},\\ 
    \text{and} \quad \Delta_W(u,T,t) &= \delta_{\left(\Psi(\mathfrak{j}(x^W_\ell(t-),\sigma),T,t)) \right)} - \delta_{\left(\Psi(x^W_\ell(t-),T,t)) \right)}.
    \end{aligned}
\end{equation*}
We are now ready to yield the main characterization of $\zeta^K_t$, as inspired by \cite{tranModelesParticulairesStochastiques2006}.

\begin{prop} 
\label{def:zeta}
 Define on the same probability space as $\zeta^K_0$, and independently from $\zeta^K_0$, three independent Poisson point measures $Q_Y$ on $\R_+ \times U_Y$ with intensity $dt \mu_Y(du)$, for $Y \in \{H,W,G\}$. Then $\zeta^K = (\zeta^{H|K}, \zeta^{W|K})$ is defined as the unique strong solution taking values in $\D(\R_+, \M_1(E)^2)$ of the following equation. For $X \in \{H,W\}$ and $T \geq 0$, 
\begin{equation}
\label{eq:def-zeta}
 \zeta^{X|K}_T = \frac{1}{K_X}\left(\sum_{j=1}^{K_X} \delta_{\Psi(x^X_j(0), T, 0)} + \sum_{Y \in \{H,W,G\}} \int_0^T \int_{U_Y} \mathcal{I}_Y(t-,u) \Delta_X(u,T,t) Q_Y(dt,du)\right),
\end{equation}
where $U_G$ and $U_H = U_W = U$ are defined by Equations \eqref{eq:UG} and \eqref{eq:defU}, respectively.
\end{prop}

The idea behind Equation \eqref{eq:def-zeta} goes as follows. Let us focus for example on the distribution $\zeta^{H|K}_T$ of household types at time $T$. Each household's type contributes to the distribution at uniform weight $1/K_H$. If no infection event occurs between times $0$ and $T$, then the state of the $k$-th household at time $T$ is given by $\Psi(x^H_k(0),T,0)$. However, suppose now that before time $T$, at least one initially susceptible member of the $k$-th household is infected, and let $t$ be the first time at which such an event occurs. Then $x^H_k(t) = \mathfrak{j}(x^H_k(t-),\sigma)$ where $\sigma$ is distributed according to $\nu$. If no other infections affect this household up to time $T$, it will be in state $\Psi(\mathfrak{j}(x^H_k(t-),\sigma),T,t)$ instead of $\Psi(x^H_k(0),T,0) = \Psi(x^H_k(t-),T,t)$. This reasoning is reflected in $\Delta_H$, and can be iterated over the whole of $[0,T]$. Finally, the terms $\I_Y$ for $Y \in \{G,H,W\}$ assure that all infection events occur at the corresponding rates.

\begin{proof}
The proof uses classical arguments \cite{fournierMicroscopicProbabilisticDescription2004}, which will only be outlined here. Start by establishing existence of $\zeta^{K}$. Consider the sequence $(T_n)_{n \geq 0}$ of successive jump times of $\zeta^{K}$, where we define $T_0 = 0$. Then using a method similar to rejection sampling, $(T_n)_{n \geq 0}$ can be obtained as a subsequence of the jump times of a Poisson process with intensity $\nmax (\lambda_H \nmax + \lambda_W \nmax + \beta_G)K$, whose only limiting value is $+\infty$. Thus $\lim_{n \to +\infty} T_n = +\infty$ almost surely, ensuring that $\zeta^{K}$ takes values in $\D(\R_+, \M_1(E))$.     

Finally, uniqueness is obtained by an induction argument which proves that for any $n \geq 0$, $(T_n, \zeta^{K}_{T_n})$ is uniquely determined by $(\zeta^{K}_0, (Q_Y)_{ Y \in \mathcal{S}})$ where $\mathcal{S} = \{G,H,W\}$. This obviously is true for $n=0$. The induction step relies on the observation that $T_{n+1}$ is uniquely determined by $(T_n, \zeta^{K}_{T_n}, (Q_Y)_{ Y \in \mathcal{S}})$ and $\zeta^{K}_{T_{n+1}}$ by $(T_{n+1}, T_n, \zeta^{K}_{T_n}, (Q_Y)_{ Y \in \mathcal{S}})$. The induction hypothesis allows to conclude.
\end{proof}

Let us briefly show that it follows from Proposition \ref{def:zeta} that Equation \eqref{eq:def-SI} is well posed.

\begin{lemma}
    \label{lemme:shsw}
    Suppose that almost surely, $K_H S_H(0) = K_W S_W(0)$ and Equation \eqref{eq:hyp} holds. Then for any $t \geq 0$, $K_H S_H(t) = K_W S_W(t)$ and $K_H I_H(t) = K_W I_W(t)$, almost surely.
\end{lemma}

\begin{proof}
    Let $T \geq 0$, and for any $x = (n,s,\tau) \in E$, let $\bfn(x) = n$, $\bfs(x) = s$ and $\bfi(x) = i(\tau)$.
    Start by focusing on $S_X(T) = \angles{\zeta^{X|K}_T}{\bfs}$ for $X \in \{H,W\}$. It follows from Equation \eqref{eq:def-zeta} that 
    {\small
    \begin{equation*}
        K_X S_X(T) = \sum_{j=1}^{K_X} \bfs(\Psi(x^X_j(0), T, 0)) + \sum_{Y \in \{H,W,G\}} \int_0^T \int_{U_Y} \mathcal{I}_Y(t-,u) \angles{\Delta_X(u,T,t)}{\bfs} Q_Y(dt,du).
    \end{equation*}
    }
    Notice that on the one hand, for any $x \in E$ and $0 \leq t \leq T$, $\bfs(\Psi(x, T, t)) = \bfs(x)$. Hence the first term of the right-hand side equals $K_X S_X(0)$, and for any $u = (\bm{\theta}, k, \ell, \sigma)$,
    \begin{equation*}
        \angles{\Delta_H(u,T,t)}{\bfs} = \bfs(\Psi(\mathfrak{j}(x^H_k(t-),\sigma),T,t)) - \bfs(\Psi(x^H_k(t-),T,t)) = -1.
    \end{equation*}
    The analogous computation yields $\angles{\Delta_W(u,T,t)}{\bfs} = -1$. Thus $K_H S_H(T) = K_W S_W(T)$ almost surely.

    Let us now turn to $I_X(T) = \angles{\zeta^{X|K}_T}{\bfi}$. This time, for any $u = (\bm{\theta}, k, \ell, \sigma)$ and $0 \leq t \leq T$, it holds that $\angles{\Delta_X(u,T,t)}{\bfi} = \setind{\sigma > (T-t)}$. Finally, Equation \eqref{eq:hyp} ensures that 
    \begin{equation*}
        \sum_{k=1}^{K_H} \bfi(\Psi(x^H_k(0), T, 0)) = \sum_{k=1}^{K_H} \sum_{j=1}^{n^H_k - s^K_k(0)} \setind{\tau^H_{k,j}(0) > T} = \sum_{\ell=1}^{K_W} \bfi(\Psi(x^W_\ell(0), T, 0)).
    \end{equation*}
    The conclusion follows as previously from Equation \eqref{eq:def-zeta}.
\end{proof}

%%%%%%%%%%%%%%%%%%%%% SECTION %%%%%%%%%%%%%%%%%%%%%
\section{Main results}
 
In this section, we are going to present our main results on the convergence of $(\zeta^K)_{K \geq 1}$ in the Skorokhod space $\D\left(\R_+, \M_1(E)\right)^2$. For $X \in \{H,W\}$, let $\overline{X}$ be the complementary structure type, \emph{i.e.} $\overline{X} = W$ if $X = H$ and vice-versa.

As we are interested in studying the limit $K \to \infty$, an important ingredient will be the asymptotic behavior of the sequence of random graphs on which the epidemic spreads. Let $\pi^{H|K}$ and $\pi^{W|K}$ be the household and workplace size distributions observed in $\mathbf{G}^K$. The law of large numbers ensures that $(\pi^{H|K}, \pi^{W|K})_{K \geq 1}$ converges $\P_\mathcal{G}$-almost everywhere to $(\pi^H, \pi^W)$. We hence define 
\begin{equation*}
    \OmegaCvgce = \{\omega \in \OmegaG : (\pi^{H|K}(\omega), \pi^{W|K}(\omega)) \xrightarrow[K \to \infty]{}  (\pi^H, \pi^W)\},
\end{equation*}
and our main results will hold for $\omega \in \OmegaCvgce$.

Further, the following assumption on the sequence of initial conditions $(\zeta^K_0)_{K \geq 1}$ will be required from now on.
\begin{hyp}
\label{hyp:zetaK0}
    For any $X \in \{H,W\}$ and $T \geq 0$, suppose that:
    \begin{enumerate}[label=(\roman*)]
        \item 
        \begin{equation*} 
            \lim_{N \to \infty} \sup_{K \geq 1} \E\left[ \sup_{0 \leq t \leq T} \frac{1}{K_X} \sum_{k=1}^{K_X} \sum_{i=1}^{\nmax} \setind{n^X_k - s^X_k(0) \geq i, \; | \tau^X_{k,i}(0) - t| \geq N} \right] = 0.
        \end{equation*}

        \item For any $c \in \R$, for any $i \in \bbrackets{1}{\nmax}$,
        \begin{equation*}
            \lim_{\epsilon \to 0} \sup_{K \geq 1} \E\left[ \frac{1}{K_X} \sum_{k=1}^{K_X} \setind{n^X_k - s^X_k(0) \geq i, \; | (\tau^X_{k,i}(0) - T) - c | \leq \epsilon} \right] = 0.
        \end{equation*}
        
    \end{enumerate}
\end{hyp}

Briefly, the first assumption allows to control the impact of the initial condition on the queues of the distribution of remaining infectious periods at each time, while the second condition is related to aspects of absolute continuity. These conditions are for instance satisfied if for any $K \geq 1$, at time $0$, the remaining infectious periods of infected individuals are i.i.d. of law $\nu$, while those of recovered individuals are set to be equal to zero. Notice that this choice for recovered individuals does not represent a loss of generality, as it does not affect the epidemic spread and initially recovered individuals will remain recognizable at any time $T$ as the only ones whose remaining infectious periods equal $-T$. 

\subsection{Large population approximation of \texorpdfstring{$(\zeta^K)_{K \geq 1}$}{the individual-based process}}
\label{sec:results}

 For any $f \in  \C^1_b(\R_+ \times E, \R)$, let $f_t(x) = f(t,x)$ and $f^\I_t(x) = \angles{\nu}{f_t(\mathfrak{j}(x,\cdot))}$ for every $(t,x) \in \R_+ \times E$. Consider the differential operator $\A$ defined as follows. For any $x = (n,s,\tau) \in E$,  
\begin{equation*}
\A f_t(x) = \partial_t f(t,x) - \sum_{k=1}^{n-s} \partial_{\tau_k} f(t,x). 
\end{equation*}
Also, for any $x = (n,s,\tau) \in E$, let $\bfn(x) = n$, $\bfs(x) = s$ and $\bfi(x) = i(\tau)$ be the functions which to a structure in state $x$ associate the corresponding structure size, number of susceptible and number of infected members, respectively. For instance, for any $X \in \{H,W\}$ the average rate of within-structure infections at time $t$ is given by 
\begin{equation*}
    \lambda_X \angles{\zeta^{X|K}_t}{\bfs \bfi} = \frac{\lambda_X}{K_X} \sum_{k=1}^{K_X} \bfs(x^X_k(t)) \bfi(x^X_k(t)).
\end{equation*}
Notice also that as mentioned previously, the average size of structures is constant over time, hence $\angles{\zeta^{X|K}_t}{\bfn} = \angles{\zeta^{X|K}_0}{\bfn}$ for all $t \geq 0$.
Finally, let $\mathfrak{M}_1 = \M_1(E)^2$. We are now ready to state our first result, whose proof is postponed to Section \ref{sec:proof-cvgce}. 

\begin{theorem} 
\label{thm:cvgce}
Let $\omega \in \OmegaCvgce$. Suppose that $\left( \zeta^K_0 \right)_{K \geq 1}$ satisfies Assumption \ref{hyp:zetaK0} and converges in law to $\eta_0 \in \mathfrak{M}_1$ such that $\angles{\eta^X_0}{\bfs} > 0$ for any $X \in \{H,W\}$. Then $(\zeta^K)_{K \geq 1}$ converges in $\D\left(\R_+, \M_1(E)\right)^2$ to $\eta = (\eta^H, \eta^W)$ defined as the unique solution of the following system of Equations \eqref{eq:defeta}. For any $f \in \C^1_b(\R_+ \times E, \R)$, for any $T \geq 0$,

\begin{equation}
\label{eq:defeta}
\begin{aligned}
    \angles{\eta^X_T}{f_T} &= \angles{\eta^X_0}{f_0} + \int_0^T \angles{\eta^X_t}{\A f_t} dt + \lambda_X \int_0^T \angles{\eta^X_t}{\bfs \bfi(f^\I_t - f_t)} dt \\
    &+ \lambda_{\overline{X}} \int_0^T \frac{\angles{\eta^{\overline{X}}_t}{\bfs \bfi}}{\angles{\eta^{\overline{X}}_t}{\bfs}} \angles{\eta^X_t}{ \bfs (f^\I_t - f_t)} dt + \beta_G \int_0^T \frac{\angles{\eta^H_t}{\bfi}}{\angles{\eta^H_0}{\bfn}} \angles{\eta^X_t}{ \bfs (f^\I_t - f_t)} dt. 
\end{aligned}
\end{equation}
\end{theorem}

This measure-valued equation can further be related to a system of PDEs. Indeed, it is possible to establish an absolute continuity result for the marginals of $\eta^X$ conditioned on the structure's size and number of susceptible members. The associated densities can be shown to satisfy, in the sense of distributions, a system of differential equations related to non-linear nonlocal transport equations. Details are provided in \cite[Chapter 3]{kubaschApproximationStochasticModels2024}, and we refer to \cite{tranModelesParticulairesStochastiques2006} for related results on age-structured models.

\begin{rque}
Theorem \ref{thm:cvgce} focuses on the case $\angles{\eta^X_0}{\bfs} > 0$ for $X \in \{H,W\}$, as this condition ensures that the epidemic is capable of spreading at the macroscopic level. This arguably corresponds to the most important setting for epidemiology and public health. If $\angles{\eta^X_0}{\bfs} = 0$ for $X \in \{H,W\}$, only recovery events are perceivable in the large population limit, leading to 
\begin{equation*}
\begin{aligned}
    \angles{\eta^X_T}{f_T} &= \angles{\eta^X_0}{f_0} + \int_0^T \angles{\eta^X_t}{\A f_t} dt. 
\end{aligned}
\end{equation*}
In particular, the proof of Theorem \ref{thm:cvgce} could readily be adapted in this case, since infection rates converge to zero along with the proportion of susceptible individuals in the large population limit. 
\end{rque}

\bigskip

From now on, let us assume that $\nu$ is the exponential distribution of parameter $\gamma$. As we shall see, it then is possible to deduce from Theorem \ref{thm:cvgce} that the proportion of susceptible and infected individuals in the population converges to the solution of a dynamical system, when the size of the population grows large. 

Let $s(t)$ and $i(t)$ be the proportions of susceptible and infectious individuals, respectively, in the population at time $t$ according to distribution $\eta_t$.  Further introduce the set
\begin{equation*}
\config = \left\{ (n-i,i): 2 \leq n \leq \nmax, 0 \leq i \leq n - 1 \right\}.
\end{equation*}
For $(S,I) \in \config$, let $n^H_{S,I}(t)$ be the proportion of households containing $S$ susceptible and $I$ infected individuals at time $t$, according to distribution $\eta^H_t$. Define $n^W_{S,I}(t)$ analogously for workplaces. Finally, consider
\begin{equation*}
\tau_G(t) = \beta_G i(t),
\text{ and } 
\tau_X(t) = \frac{\lambda_X}{m_X} \sum_{(S,I) \in \config} SI \; n^X_{S,I}(t) \text{ for } X \in \{H,W\}.
\end{equation*}

We assume that at time $0$, the remaining infectious period of each infected individual is supposed to be distributed according to $\nu$, independently from one another. Further, given that a household ($X = H$) or workplace ($X = W$) is of size $n$, the probability that at time $0$ it contains $s$ susceptible and $i$ infected members is given by $p^X_{s,i | n}$. We assume that there exists $\varepsilon = (\varepsilon_s, \varepsilon_i) \in (0,1)^2$ such that $\varepsilon_s + \varepsilon_i \leq 1$ and for any $X \in \{H,W\}$, 
\begin{equation}
\label{eq:new-hyp-eta0}
 \sum_{n=1}^{\nmax} \sum_{\substack{0 \leq s \leq n \\ 0 \leq i \leq n - s}} s \, \pi^X_n p^X_{s,i | n} = m_X \varepsilon_s \quad \text{and} \quad \sum_{n=1}^{\nmax} \sum_{\substack{0 \leq s \leq n \\ 0 \leq i \leq n - s}} i \, \pi^X_n p^X_{s,i | n} = m_X \varepsilon_i. 
\end{equation}
In other words, at time 0, the proportion of susceptibles and infected in the population is given by $\varepsilon_s$ and $\varepsilon_i$, respectively. Finally, we initialize the remaining infectious periods of recovered individuals at zero.

This setting corresponds to the following probability distribution $\eta_{0,\epsilon} = (\eta^H_{0,\epsilon}, \eta^W_{0,\epsilon}) \in \mathfrak{M}_1$ characterized for $X \in \{H,W\}$ as follows. For any $n \in \bbrackets{1}{\nmax}$ and $s \in \bbrackets{0}{n}$:
\begin{equation}
\label{eq:etastar}
\eta^X_{0,\epsilon}(n, s, d\tau) =  \sum_{i=0}^{n-s} \pi^X_n p^X_{s,i|n} \left(\nu^{\otimes i} \otimes \delta_{0}^{\otimes\left(\nmax - i)\right)}\right)(d\tau).
\end{equation}
It then is possible to describe the epidemic dynamics by a finite, closed set of ordinary differential equations, as shown in the following result whose proof is postponed to Section \ref{sec:proof-sysdyn}.

% We assume that at time $0$, a fraction $\varepsilon$ of uniformly chosen individuals are infected amidst an otherwise susceptible population. Furthermore, at time $0$, the remaining infectious period of each infected individual is supposed to be distributed according to $\nu$, independently from one another. Let us emphasize here that actually, only this second assumption is crucial for the results to hold, while the original distribution of infected individuals does not need to be uniform (in which case forthcoming Equations \eqref{eq:etastar} and \eqref{eq:initialcdts} need to be adapted). We have chosen this particular initial condition as it has been previously considered in the literature, and refer to the Discussion for further comments.

%In practice, this setting corresponds to the following probability distribution $\eta_{0,\varepsilon} = (\eta^H_{0,\varepsilon}, \eta^W_{0,\varepsilon}) \in \mathfrak{M}_1$ characterized for $X \in \{H,W\}$ as follows. For any $n \in \bbrackets{1}{\nmax}$ and $s \in \bbrackets{0}{n}$:
%\begin{equation}
%\label{eq:etastar}
%\eta^X_{0,\varepsilon}(n, s, d\tau) =  \pi^X_n \binom{n}{s} (1 - \varepsilon)^s \varepsilon^{n-s} \left(\nu^{\otimes (n-s)} \otimes \delta_{0}^{\otimes\left(\nmax - n + s)\right)}\right)(d\tau).
%\end{equation}
%It then is possible to describe the epidemic dynamics by a finite, closed set of ordinary differential equations, as shown in the following result whose proof is postponed to Section \ref{sec:proof-sysdyn}.

\begin{theorem}
\label{thm:cvgce-sysdyn}
Let $\varepsilon = (\varepsilon_s, \varepsilon_i) \in (0,1)^2$ and suppose that $\eta$ satisfies Equation \eqref{eq:defeta} with $\eta_0 = \eta_{0,\epsilon}$. 
Then the functions $(s,\, i,\, n^X_{S,I}: X \in \{H,W\}, (S,I) \in \config)$ are characterized as being the unique solution of the following dynamical system: for any $t \geq 0$, $X \in \{H,W\}$ and $(S,I) \in \config$,

\begin{subequations}
\label{eq:sysdyn}
\begin{align}
    \frac{d}{dt} s(t) &= - (\tau_H(t) + \tau_W(t) + \tau_G(t) s(t)), \label{eq:subeq-sysdyn-si} \\
    \frac{d}{dt} i(t) &=  - \frac{d}{dt} s(t) - \gamma i(t), \label{eq:subeq-sysdyn-i} \\
    \frac{d}{dt} n^X_{S,I}(t) &= - \left(\lambda_X SI + \tau_{\overline{X}}(t) \frac{S}{s(t)} + \tau_G(t) S + \gamma I\right) n^X_{S,I}(t) \label{eq:subeq-sysdyn-structures} \\
     \; +& \gamma (I+1) n^X_{S,I+1}(t) \setind{S + I < \nmax} \notag \\
     \; +& \left(\lambda_X (S+1)(I-1) + \tau_{\overline{X}}(t) \frac{S+1}{s(t)} + \tau_G(t) (S+1)  \right) n^X_{S+1, I-1}(t) \setind{I \geq 1} \notag, 
\end{align}
\end{subequations}
with initial conditions given by 
%\begin{equation}
%\label{eq:initialcdts}
%    s(0) = 1 - \varepsilon; \;\; 
%    i(0) = \varepsilon; \;\; n^X_{S,I}(0) = \binom{S+I}{I} \pi^X_{S+I} (1 - \varepsilon)^S \varepsilon^I. 
%\end{equation}
\begin{equation}
\label{eq:initialcdts}
    s(0) = \varepsilon_s; \;\; 
    i(0) = \varepsilon_i; \;\; n^X_{S,I}(0) = \sum_{n = S+I}^{\nmax} \pi^X_n p^X_{S,I | n}. 
\end{equation}

\end{theorem}

This dynamical system may be understood as follows. Equation \eqref{eq:subeq-sysdyn-si} corresponds to the fact that the proportion of susceptibles decreases whenever a new infection occurs within the general population, or within a household or workplace. Similarly, Equation \eqref{eq:subeq-sysdyn-i} is due to newly contaminated individuals moving from the susceptible to the infected state, which they in turn leave at rate $\gamma$. 
It remains to take an interest in Equation \eqref{eq:subeq-sysdyn-structures}. The first line indicates that a structure of type $(S,I) \in \config$ changes its composition upon either the infection of one of its susceptible members which may occur in any layer of the graph, or upon the removal of one of its infected members. Simultaneously, a structure of type $(S,I+1)$ transforms into a structure of type $(S,I)$ whenever one of its infected members recovers, while a structure of type $(S+1,I-1)$ becomes of type $(S,I)$ upon infection of a susceptible member. 
In particular, this result shows that under the assumptions of Theorem \ref{thm:cvgce-sysdyn}, in the large population limit, the natural correlation between structures sharing infected members becomes negligible. This allows to obtain a stronger model reduction than in Theorem \ref{thm:cvgce}, in the sense that the model reduces to a finite-dimensional ODE-system instead of a measure-valued equation. As will become apparent in Section \ref{sec:proof-sysdyn}, this result relies on recovery times being exponentially distributed, and particularly the memory-less property. 

Before detailing the proofs of Theorems \ref{thm:cvgce} and \ref{thm:cvgce-sysdyn}, let us examine the latter from a numerical point of view.

\subsection{Numerical assessment of the limiting dynamical system}
\label{sec:numeric}

The aim of this section is first to portray that the proposed large population limit, under the form of dynamical system \crefrange{eq:subeq-sysdyn-si}{eq:subeq-sysdyn-structures}, is in good accordance with the original stochastic model for large population sizes. This secondarily leads to some practical comments on the implementation of the dynamical system. Finally, a comparison with another reduced model for epidemics with two layers of mixing will be established, namely with an edge-based compartmental model (EBCM) in the line of work of Volz \emph{et al.} \cite{volzEffectsHeterogeneousClustered2011}.

\subsubsection{Implementation of the dynamical system and illustration of Theorem \ref{thm:cvgce-sysdyn}}

\begin{figure}
    \centering
    \includegraphics[width=0.8\textwidth]{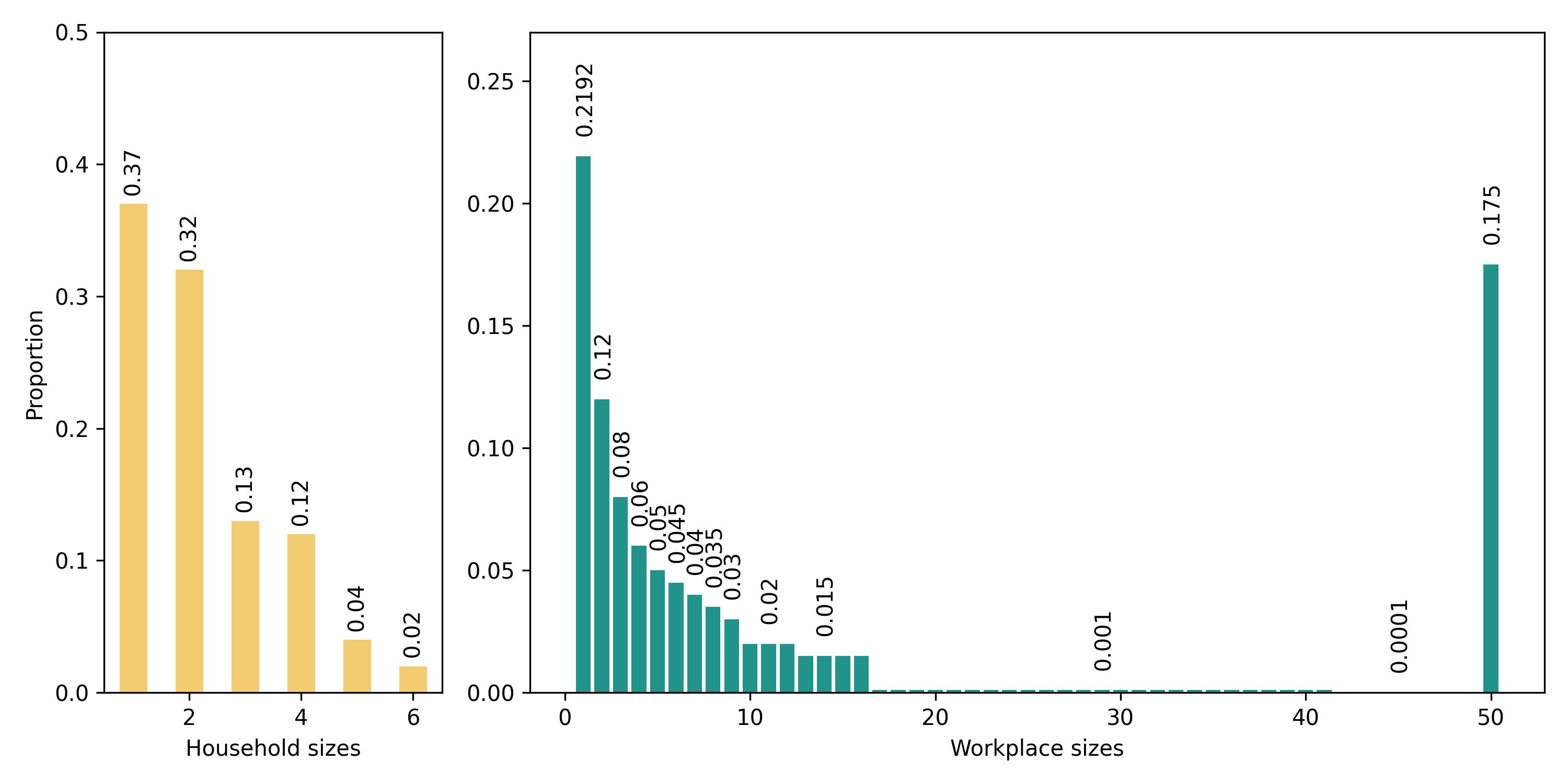}
    \caption{Household and workplace size distributions $\pi^H$ (left) and $\pi^W$ (right) used in simulations.}
    \label{fig:piX-insee}
\end{figure}

Let us start by illustrating the result of Theorem \ref{thm:cvgce-sysdyn} through numerical simulations. Using Gillespie's algorithm, we have performed fifty simulations of the epidemic within a population of $K = 10000$ individuals, where $\pi^H$ and $\pi^W$ are roughly inspired by the French household and workplace distributions as observed in 2018 by Insee \cite{bansayeEpidemiologicalFootprintContact2024}. These distributions are represented in Figure \ref{fig:piX-insee}. Two sets of epidemic parameters have been considered, leading to either $R_0 = 2.5$ or $R_0 = 1.2$. In both cases, the majority of contaminations take place at the local level. Indeed, for the first scenario with $R_0 = 2.5$, 42\% and 18\% of infections occur within households and workplaces, respectively, and these proportions both equal 40\% in the second scenario. Further, the epidemic is started by infecting either $10$ or $100$ individuals chosen uniformly at random at time $0$. For each simulation, we have followed the evolution of the proportion of susceptible and infected individuals within the population over time.

The simulation outcomes are presented in Figure \ref{fig:convergence}. For each choice of parameters, the solutions $s$ and $i$ of dynamical system \crefrange{eq:subeq-sysdyn-si}{eq:subeq-sysdyn-structures} with initial condition given by \eqref{eq:initialcdts} are plotted on the same graph.  As expected, one observes good accordance of the stochastic simulations and the deterministic functions $(s,i)$.

\begin{figure}
    \centering
    \includegraphics[width=\textwidth]{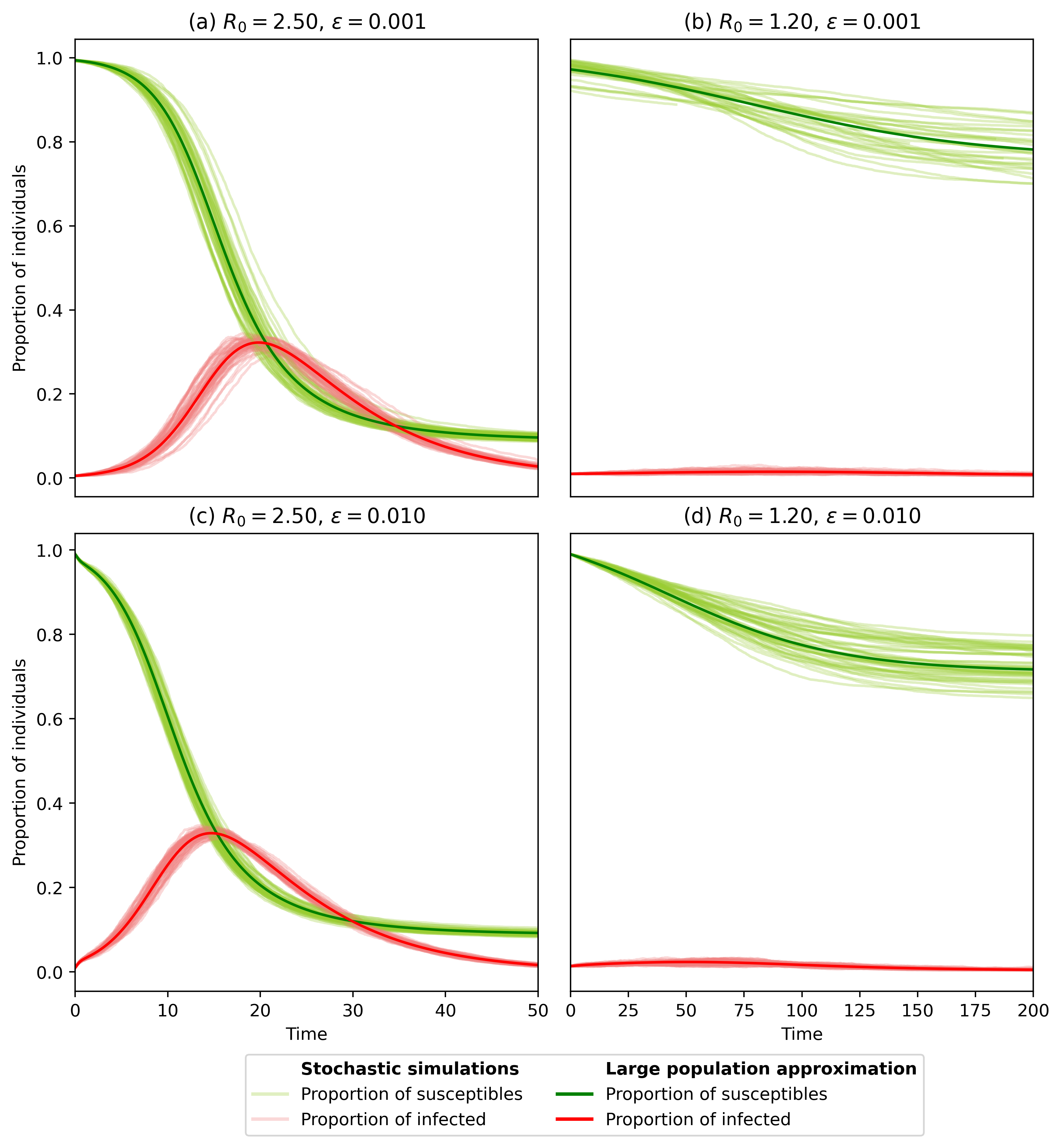}
    \caption{Comparison of the stochastic model with its large population approximation given by dynamical system \crefrange{eq:subeq-sysdyn-si}{eq:subeq-sysdyn-structures}. Household and workplace distributions are those of Figure \ref{fig:piX-insee}. Two sets of epidemic parameters are considered, namely $(\beta_G, \lambda_H, \lambda_W, \gamma) = (0.125, 1.5, 0.00115, 0.125)$ and $(\beta_G, \lambda_H, \lambda_W, \gamma) = (0.03, 0.05, 0.0015, 0.125)$ for the left and right column respectively ($R_0 = 2.5$ and $R_0 = 1.2$). The initial conditions are either $\varepsilon = 0.001$ in Panels (a) and (b), or $\varepsilon = 0.01$ in (c) and (d). For each of these scenarios, Gillespie's algorithm is used to simulate $50$ trajectories of the stochastic model defined in Proposition \ref{def:zeta} in a population of $K = 10000$ individuals (faint lines). For Panels (a) and (b), only trajectories reaching a threshold proportion of $0.005$ infected are kept, and time is shifted so that time $0$ corresponds to the moment when this threshold is reached.
    Finally, the deterministic solution $(s,i)$ of \crefrange{eq:subeq-sysdyn-si}{eq:subeq-sysdyn-structures} is represented for each scenario (thick lines). For Panels (a) and (b), the same time shifting procedure as for simulations is applied.}
    \label{fig:convergence}
\end{figure}
\bigskip

Before proceeding further, let us briefly emphasize a few aspects of the implementation of the proposed deterministic model. A potential drawback of dynamical system \crefrange{eq:subeq-sysdyn-si}{eq:subeq-sysdyn-structures} consists in its large dimension. Indeed, it holds that $\# \config = \nmax(\nmax+1)/2 - 1$. The number of equations of dynamical system \crefrange{eq:subeq-sysdyn-si}{eq:subeq-sysdyn-structures} is hence of order $O(\nmax^2)$. However, this fast-growing number of equations actually is manageable, as it is possible to implement the dynamical system in an automated way, in the sense that each equation does not need to be written one-by-one by the programmer. We refer to \ref{apx:implementation} for details.

Nevertheless, the large dimension of the dynamical system of interest raises the question whether it is numerically speaking interesting to actually use it for numerical explorations. We have compared the average time needed to either solve once dynamical system \crefrange{eq:subeq-sysdyn-si}{eq:subeq-sysdyn-structures}, or to simulate one trajectory of the stochastic model using Gillespie's algorithm, for different choices of epidemic parameters. In practice, for stochastic simulations, it is often necessary to compute several individual trajectories in order to obtain the general behavior of the epidemic. However, as it is possible to execute these simulations in parallel, comparison to one individual simulation seemed the most pertinent. The procedure and results are detailed in \ref{apx:runtime}. In summary, solving the reduced model is up to one order of magnitude faster than performing one stochastic simulation for values of $R_0 > 1$ that are not too close to the critical case $R_0 = 1$. This shows that the reduced model is pertinent for numerical exploration.

\subsubsection{Comparison to edge-based compartmental models} 
\label{sec:ebcm}

One may notice that the population structure, as described in Section \ref{sec:general-model}, can be regarded as a modification of the well-studied configuration model. Indeed, as uniform mixing is assumed within households and workplaces, each individual belongs to two totally connected subgraphs (possibly of size one), referred to as \emph{cliques}. In particular, our network of household- and workplace-contacts may be seen as a two-layer graph, where each layer corresponds to a clique configuration model (CCM) \cite{millerPercolationEpidemicsRandom2009, newmanRandomGraphsClustering2009}, a generalization of configuration models to include \emph{cliques}. In particular, an EBCM on CCM has been designed to capture the epidemic impact of individuals belong to several cliques at once \cite{volzEffectsHeterogeneousClustered2011}. It thus is possible to derive a similar EBCM for the household-workplace model, which provides a state-of-the-art heuristic approximation of the epidemic. This Section aims at comparing this heuristic approximation with dynamical system \crefrange{eq:subeq-sysdyn-si}{eq:subeq-sysdyn-structures}, in order to compare their numerical performance. Details are provided in \ref{apx:ebcm}.  

Consider the case where at time $0$, a fraction $\varepsilon_i = \varepsilon > 0$ of uniformly chosen individuals are infected amidst an otherwise susceptible population (\emph{i.e.} $\varepsilon_s = 1 - \varepsilon$). This corresponds to setting
\begin{equation}
\label{eq:unif-init}
s(0) = 1 - \varepsilon; \quad i(0) = \varepsilon; \quad p^X_{s,i|n} = \binom{n}{s} (1 - \varepsilon)^s \varepsilon^{n-s}. 
\end{equation}

For such initial conditions, edge-based compartmental models on CCM variants have been known to be in excellent accordance with simulations of the corresponding stochastic epidemic models, under the assumption of a very small initial proportion of infected. In our case, we have confronted the EBCM with dynamical system \crefrange{eq:subeq-sysdyn-si}{eq:subeq-sysdyn-structures}, as well as simulated trajectories of our stochastic model. As expected, for very small values of $\varepsilon$, the EBCM and dynamical system \crefrange{eq:subeq-sysdyn-si}{eq:subeq-sysdyn-structures} both yield the correct epidemic dynamics, whereas for larger values of $\varepsilon$, the EBCM does not fit the simulated epidemic trajectories. We refer to \ref{apx:ebcm} for details. 

Finally, proceeding like before, we obtain that the number of equations of the EBCM is of order $O(\nmax^3)$. This significantly increases the computation time, as briefly illustrated in \ref{apx:ebcm}, arguing against the applicability of this EBCM for numerical explorations. 

To conclude, in the particular case of the household-workplace model studied in this article, the EBCM is numerically close to the large population approximation described by dynamical system \crefrange{eq:subeq-sysdyn-si}{eq:subeq-sysdyn-structures}, under the condition that the initial proportion of infected is very small. However, considering both the computational cost of its higher dimension and the loss of accuracy for more general initial conditions of the EBCM, the large population approximation given by dynamical system \crefrange{eq:subeq-sysdyn-si}{eq:subeq-sysdyn-structures} significantly improves upon this state-of-the art heuristic approximation.

%%%%%%%%%%%%%%%%%%%%% SECTION %%%%%%%%%%%%%%%%%%%%%
\section{Proofs}
\label{sec:proofs}

This section is devoted to establishing Theorems \ref{thm:cvgce} and \ref{thm:cvgce-sysdyn}. The proof of Theorem \ref{thm:cvgce} establishes convergence of the measure-valued epidemic process in the appropriate measure-valued Skorokhod space, following a classical tightness-identification-uniqueness scheme. Notice that, as $\M_F(E)$ is of infinite dimension, the underlying topology plays a subtle role in establishing this result. In addition, the stochastic process under consideration combines both continuous, deterministic dynamics as well as stochastic jump events. In particular, some technical difficulty arises from the infectiousness being a discontinuous function of the remaining infectious period. This will become apparent in the proof of forthcoming Proposition \ref{prop:idlimit}.
Nevertheless, it allows us to obtain a deterministic prediction of the dynamics of the structure type distributions during the course of an epidemic. At this level, the limiting object is rich, allowing it to convey detailed information on the distribution of remaining infectious periods within structures. We can further coarse-grain the infinite-dimensional limiting object in the case where $\nu$ is an exponential distribution, as demonstrated in Theorem \ref{thm:cvgce-sysdyn} and the dynamical system \crefrange{eq:subeq-sysdyn-si}{eq:subeq-sysdyn-structures}. As mentioned previously, the existence of an asymptotically exact, closed, finite-dimensional ODE-system capturing the epidemic dynamics was not obvious from the beginning. 
Indeed, in order to obtain this result, we need to show that structuring the population of previously contaminated individuals by remaining infectious period is not necessary to handle the correlation of epidemic states of structures sharing a common infected. This result fundamentally relies on the memory-less property of the exponential distribution, as illustrated in forthcoming Propositions \ref{prop:dist-infected} and \ref{prop:nXt}.

\subsection{Proof of Theorem \ref{thm:cvgce}}
\label{sec:proof-cvgce}

Let us start with the proof of Theorem \ref{thm:cvgce}. It follows a classical scheme, establishing tightness of $(\zeta^K)_{K \geq 1}$, whose limiting values are shown to satisfy Equation \eqref{eq:defeta}. Uniqueness of the solutions of this equation given the initial condition then ensures the desired convergence result. In particular, the proof is inspired by \cite{fournierMicroscopicProbabilisticDescription2004} and \cite{tranModelesParticulairesStochastiques2006}.

\subsubsection{Uniqueness and continuity of the solution of Equation (\ref{eq:defeta}).} 

We are first going to establish a uniqueness result for the solutions of Equation \eqref{eq:defeta}. Notice that we do not need to prove existence of solutions in this section, as forthcoming Proposition \ref{prop:idlimit} constructs such solutions as limiting values of $(\zeta^K)_{K \geq 1}.$

Let us start with a technical lemma, whose proof we present for sake of completeness.

\begin{lemma}
\label{lem:density}
    Let $f \in \C_b(E, \R)$. There exists a sequence $(f_k)_{k \geq 1}$ taking values in $\C^1_b(E, \R)$ such that $f_k$ converges to $f$ and $\sup_{k \geq 1} \inftynorm{f_k} \leq \inftynorm{f}$.
\end{lemma}

\begin{proof}
    Consider a mollifier $\psi$, \emph{i.e.} $\psi \in \C^\infty(\R^{\nmax})$ is compactly supported, its mass $\int_{\R^{\nmax}} \psi(x)dx$ equals $1$, and for $k\geq 1$, the function $\psi_k: x \mapsto k^{\nmax} \psi(kx)$ satisfies $\lim_{k \to \infty} \psi_k = \delta_0$ in the sense of distributions.
    Let $f \in \C_b(E, \R)$. Define the sequence $(f_k)_{k \geq 1}$ as follows:
    \begin{equation*}
        f_k: (n,s,\tau) \in E \mapsto f(n,s,\cdot) * \psi_k (\tau).
    \end{equation*}

    Then, for any $x \in E$, by definition of $(\psi_k)_{k \geq 1}$ it holds that $\lim_{k \to \infty} f_k(x) = f(x)$. Further, as $\psi_k$ is of integral $1$ for any $k$, it is obvious that for any $k$, $\inftynorm{f_k} \leq \inftynorm{f}$. Finally, it also follows from the usual properties of convolution that for any $k \geq 1$, $f_k$ is smooth with respect to its last variable and the corresponding partial derivatives are bounded, hence $f_k \in \C^1_b(E, \R)$. 
\end{proof}

We may now turn to the main result of this paragraph. With slight abuse of notation, for an element $\eta = (\eta_1, \eta_2) \in \mathfrak{M}_1$, we define its total variation norm by $\| \eta \|_{TV} = \| \eta_1 \|_{TV} \vee \| \eta_2 \|_{TV}$.

\begin{prop}
\label{prop:uniqueness}

    Let $\eta_\star \in \mathfrak{M}_1$. Then Equation \eqref{eq:defeta} admits at most one measure-valued solution $\eta$ which belongs to $\C(\R_+, (\mathfrak{M}_1, \tvnorm))$, such that $\eta_0 = \eta_\star$.
\end{prop}

 From now on, for $X \in \{H,W\}$, define $n_X=  \langle \eta^X_0, \bfn \rangle$, $s_X(t) = \langle \eta^X_t, \mathbf{s} \rangle$ and $i_X(t) = \langle \eta^H_t, \mathbf{i} \rangle$. Let us establish the proposition. 

\begin{proof}
    First, notice that it follows immediately from Equation \eqref{eq:defeta} that $\angles{\eta^X_T}{1} = \angles{\eta^X_0}{1}$ for $X \in \{H,W\}$, thus $\eta_0 \in \mathfrak{M}_1$ implies that for any $T \geq 0$, $\eta_T \in \mathfrak{M}_1$. 

   Let us show that any solution $\eta$ of Equation \eqref{eq:defeta} belongs to $\C(\R_+,(\mathfrak{M}_1, \tvnorm))$. In order to do so, it is enough to show that $\eta^X \in \C(\R_+, (\M_1(E), \tvnorm))$ for any $X \in \{H,W\}$. We are going to detail the proof for $\eta^H$ only, as $\eta^W$ can be handled in the same way.
   
   Let $T\geq 0$ and $g \in \C^1_b(E, \R)$ such that $\inftynorm{g} \leq 1$. Consider the function defined by
   \begin{equation*}
       \forall (t,x) \in \R \times E, \quad f_t(x) = g(\Psi(x,T,t))
   \end{equation*}
    and recall that $f^\I_t(x) = \angles{\nu}{f_t(\mathfrak{j}(x,\cdot))}$.
    Then by definition, $f_T(x) = g(x)$. It follows from the assumption $g \in \C^1_b(E, \R)$ that $f \in \C^1_b(\R_+ \times E,\R)$. The advantage of this construction is that $t \mapsto \Psi(x,T,t)$ corresponds to a reversal of time, which cancels out the deterministic dynamics described by the differential operator $\A$. Indeed, letting $x = (n,s,\tau)$, a brief computation shows that 
    \begin{equation*}
        \begin{aligned}
            \partial_t f_t(x) = \sum_{k=1}^{n-s} \partial_{\tau_k} g(\Psi(x,T,t)) \text{ and }
            \partial_{\tau_k} f_t(x) = \partial_{\tau_k} g(\Psi(x,T,t)),
        \end{aligned}
    \end{equation*}
   which yields that $\A f_t(x) = 0$ for all $(t,x) \in \R \times E$. Using the fact that $\angles{\eta^H_T}{g} = \angles{\eta^H_T}{f_T}$,  it follows from Equation \eqref{eq:defeta} that 
   \begin{equation*}
   \begin{aligned}
       \angles{\eta^H_T}{g} &= \angles{\eta^H_0}{f_0} + \lambda_H \int_0^T \angles{\eta^H_t}{\bfs \bfi \left(f^\I_t - f_t\right)} dt \\
       &+ \lambda_W \int_0^T \frac{1}{s_W(t)} \angles{\eta^W_t}{\bfs\bfi} \angles{\eta^H_t}{\bfs \left(f^\I_t - f_t\right)} dt
       + \beta_G \int_0^T \frac{i_H(t)}{n_H} \angles{\eta^H_t}{\bfs \left(f^\I_t - f_t\right)} dt.
    \end{aligned}
   \end{equation*}

    Recall that $i_H(t) \leq \nmax$ and $\frac{i_H(t)}{n_H} \leq 1$ since for any $x \in E$, $\bfi(x) \leq \bfn(x)$. We may notice that the following inequalities hold, as for any $t$, $\inftynorm{f_t} \leq 1$:
    \begin{equation}
    \label{eq:ineq-eta}
    \begin{aligned}[c]
        &\angles{\eta^H_t}{\bfs \bfi \left(f^\I_t - f_t\right)} \leq 2 (\nmax)^2, \\
        &\frac{1}{s_W(t)} \angles{\eta^W_t}{\bfs\bfi} \leq \nmax, 
    \end{aligned}
    \quad \quad
    \begin{aligned}[c]
        &\angles{\eta^H_t}{\bfs \left(f^\I_t - f_t\right)} \leq 2 \nmax,\\
        &\frac{i_H(t)}{n_H} \angles{\eta^H_t}{\bfs \left(f^\I_t - f_t\right)} \leq 2 \nmax.
    \end{aligned}
    \end{equation}

    Let $C = 2\nmax\left(\lambda_H\nmax + \lambda_W \nmax + \beta_G\right)$ and let $\epsilon \in \R$. It then follows from Inequalities \eqref{eq:ineq-eta} that
    \begin{equation}
    \label{eq:aux1}
        |\angles{\eta^H_T - \eta^H_{T +\epsilon}}{g}| \leq C|T - (T+\epsilon)| = C|\epsilon|.
    \end{equation}
    Consider now $h \in \C_b(E,\R)$ such that $\inftynorm{h} \leq 1$. Lemma \ref{lem:density} ensures that there exists a sequence $(g_k)_{k \geq 1}$ taking values in $\C^1_b(E,\R)$ which converges simply to $h$ and such that $\inftynorm{g_k} \leq 1$. By dominated convergence, this implies that 
     \begin{equation*}
        |\angles{\eta^H_T - \eta^H_{T +\epsilon}}{h}| \leq C|T - (T+\epsilon)| = C|\epsilon|.
    \end{equation*}
    As $E$ is a Polish space, it follows from Proposition A.6.1 of \cite{tranModelesParticulairesStochastiques2006} that 
    \begin{equation}
    \label{eq:aux2}
        || \eta^H_T - \eta^H_{T +\epsilon} ||_{TV} = \sup_{h \in \C_b(E, \R): \inftynorm{h} \leq 1} |\angles{\eta^H_T - \eta^H_{T +\epsilon}}{h}| \leq C|\epsilon|.
    \end{equation}
    As this holds for any $\epsilon$ and any $T \geq 0$, the continuity in total variations of $\eta^H$ is established.

    It remains to establish uniqueness of the solution $\eta$ of Equation \eqref{eq:defeta} with initial condition $\eta_0 = \eta_\star$. We will once more establish uniqueness component-wise, and focus on $\eta^H$ as $\eta^W$ is treated in a similar fashion. 

    Let $\eta, \widetilde{\eta}$ be two solutions of Equation \eqref{eq:defeta} with initial condition $\eta_\star$. Let $\widetilde{i}_X(t) = \angles{\widetilde{\eta}^X_t}{\bfi}$ and define $\widetilde{s}_X(t)$ in analogous manner. As before, let $T \geq 0$. Consider again $g \in \C^1_b(E,\R)$ such that $\inftynorm{g} \leq 1$, and define $f_t$ and $f^\I_t$ as previously. Then, for $X \in \{H,W\}$
    \begin{equation*}
    \begin{aligned}
        |\angles{\eta^X_T& - \widetilde{\eta}^X_T}{g}|  \leq \lambda_X \int_0^T \left| \angles{\eta^X_t - \widetilde{\eta}^X_t}{\bfs \bfi \left(f^\I_t - f_t \right)} \right| dt \\
        &+ \lambda_{\overline{X}} \int_0^T \left|\frac{1}{s_{\overline{X}}(t)} \angles{\eta^{\overline{X}}_t}{\bfs\bfi} \angles{\eta^X_t}{\bfs(f^\I_t - f_t)} - \frac{1}{\widetilde{s}_{\overline{X}}(t)} \angles{\widetilde{\eta}^{\overline{X}}_t}{\bfs\bfi} \angles{\widetilde{\eta}^X_t}{\bfs(f^\I_t - f_t)} \right| dt \\
        &+ \frac{\beta_G}{m_H} \int_0^T \left|i_H(t) \angles{\eta^X_t}{\bfs(f^\I_t - f_t)} - \widetilde{i}_H(t) \angles{\widetilde{\eta}^X_t}{\bfs(f^\I_t - f_t)} \right| dt.
    \end{aligned}
    \end{equation*}
    
Notice that, by definition, $f_t \in \C_b(E,\R)$. It further follows from the usual criterion of continuity for parametric integrals that $f^\I_t$ is continuous on $E$ and $\inftynorm{f^\I_t} \leq \inftynorm{f_t} \leq 1$. Proceeding similarly as in Inequalities \eqref{eq:ineq-eta}, there exists $C_1 > 0$ such that, for any $t \in [0,T]$, 
    \begin{equation*}
     \lambda_X \left| \angles{\eta^X_t - \widetilde{\eta}^X_t}{\bfs \bfi \left(f^\I_t - f_t \right)} \right| \leq C_1 ||\eta^X_t - \widetilde{\eta}^X_t||_{TV}.
    \end{equation*}
    
Similarly, there exists $C_2 > 0$ such that for any $t \in [0,T]$, 
	\begin{equation*}
	\begin{aligned}
	\frac{\beta_G}{m_H} & \left|i_H(t) \angles{\eta^X_t}{\bfs(f^\I_t - f_t)} - \widetilde{i}_H(t) \angles{\widetilde{\eta}^X_t}{\bfs(f^\I_t - f_t)} \right| \\
	& \leq \frac{\beta_G}{m_H} 2 (\nmax)^2 \left(  \left|i_H(t) - \widetilde{i}_H(t) \right| + \left| \angles{\eta^X_t - \widetilde{\eta}^X_t}{\bfs(f^\I_t - f_t)} \right|  \right) \\
	& \leq C_2 \left(||\eta^H_t - \widetilde{\eta}^H_t||_{TV} + ||\eta^X_t - \widetilde{\eta}^X_t||_{TV} \right).
    	\end{aligned}
	\end{equation*}
    
    Finally, notice that Equation \eqref{eq:defeta} implies that $T \mapsto s_H(T)$ is continuously differentiable on $(0,T)$, and
    \begin{equation*}
    s_H'(T) \geq - \left(\lambda_W + \lambda_H + \frac{\beta_G}{m_H} \right)\nmax \int_0^T s_H(t) dt.
    \end{equation*}
    As $s_H(0) > 0$, Gronwall's lemma leads to: 
    \begin{equation*}
    s_H(T) \geq s_H(0) \exp\left( - \left(\lambda_W + \lambda_H + \frac{\beta_G}{m_H} \right)\nmax T\right).
    \end{equation*}
    We obtain that there exists $\varepsilon_T > 0$ such that 
    \begin{equation*}
    \forall t \in [0,T], \forall Y \in \{H,W\}, \quad \varepsilon_T \leq s_Y (t), \widetilde{s}_Y (t) \leq \nmax. 
    \end{equation*}
    Given that the application $x \mapsto x^{-1}$ is Lipschitz continuous on $[ \varepsilon_T, \nmax]$, we may proceed as above to show that there exists $C_3 > 0$ such that for any $t \in [0,T]$,  
    \begin{equation*}
    \begin{aligned}
    \lambda_{\overline{X}} & \left|\frac{1}{s_{\overline{X}}(t)} \angles{\eta^{\overline{X}}_t}{\bfs\bfi} \angles{\eta^X_t}{\bfs(f^\I_t - f_t)} - \frac{1}{\widetilde{s}_{\overline{X}}(t)} \angles{\widetilde{\eta}^{\overline{X}}_t}{\bfs\bfi} \angles{\widetilde{\eta}^X_t}{\bfs(f^\I_t - f_t)} \right| \\
    & \leq C_3 \left(||\eta^{\overline{X}}_t - \widetilde{\eta}^{\overline{X}}_t||_{TV} + ||\eta^X_t - \widetilde{\eta}^X_t||_{TV} \right).
    \end{aligned}
    \end{equation*}
    
Thus, there exists $C_X > 0$ such that 
     \begin{equation*}
        |\angles{\eta^X_T - \widetilde{\eta}^X_T}{g}| \leq C_X \int_0^T \left(||\eta^{\overline{X}}_t - \widetilde{\eta}^{\overline{X}}_t||_{TV} + ||\eta^X_t - \widetilde{\eta}^X_t||_{TV} \right) dt. 
    \end{equation*}
    We then may follow the same steps that allowed to establish Equation \eqref{eq:aux2} from Equation \eqref{eq:aux1}, and obtain that  
    \begin{equation*}
        ||\eta^X_T - \widetilde{\eta}^X_T||_{TV} \leq C_X \int_0^T \left(||\eta^{\overline{X}}_t - \widetilde{\eta}^{\overline{X}}_t||_{TV} + ||\eta^X_t - \widetilde{\eta}^X_t||_{TV} \right) dt.
    \end{equation*}
   In particular, there exists $C > 0$ such that 
    \begin{equation*}
        ||\eta^H_T - \widetilde{\eta}^H_T||_{TV} + ||\eta^W_T - \widetilde{\eta}^W_T||_{TV} \leq C \int_0^T \left(||\eta^H_t - \widetilde{\eta}^H_t||_{TV} + ||\eta^W_t - \widetilde{\eta}^W_t||_{TV} \right) dt.
    \end{equation*}
    Gronwall's lemma then assures that 
    \begin{equation*}
    \begin{aligned}
        \forall t \in [0,T], \;\; ||\eta^H_T - \widetilde{\eta}^H_T||_{TV} + ||\eta^W_T - \widetilde{\eta}^W_T||_{TV} = 0.
    \end{aligned}
    \end{equation*}
   As $T \geq 0$ is arbitrary, this concludes the proof.

\end{proof}

\subsubsection{Tightness of \texorpdfstring{$(\zeta^K)_{K \geq 1}$}{the individual-based process} in \texorpdfstring{$\D\left(\R_+, (\M_F(E), w)\right)^2$}{the appropriate Skorokhod space}}

Let us now turn to the tightness of $(\zeta^K)_{K \geq 1}$ in $\D\left(\R_+, (\M_F(E), w)\right)^2$, where $w$ designates the weak topology on $\M_F(E)$. We start by establishing the following preliminary result, whose proof relies on the chain rule.

\begin{lemma}
\label{lem:dphif}
Let $f \in \C^1_b(\R_+\times E,\R)$. Then for any $T \geq t_0 \geq 0$, for any $x \in E$, 
\begin{equation*}
f\left(T, \Psi(x, T, t_0)\right) = f\left(t_0, x \right) + \int_{t_0}^T \A f\left(t,\Psi(x, t, t_0)\right) dt.
\end{equation*}
\end{lemma}

\begin{proof}
For $x = (n,s,\tau) \in E$ and $t_0 \in \R_+$, define 
\begin{equation*}
    g_{t_0,x}: [t_0, +\infty) \to \R, \; T \mapsto f\left(T, \Psi(x, T, t_0)\right).
\end{equation*}
Let us start by noticing that for any $(t_0, x) \in \R_+ \times E$, $g_{t_0,x} \in \C^1(\R_+)$. Indeed, $g_{t_0,x} = f_{n,s} \circ h_{t_0,x}$, where 
\begin{align*}
& f_{n,s}: \R_+ \times \R^{\nmax} \to \R, \;\; (u,\mathbf{v}) \mapsto f\left(u, (n, s,\mathbf{v})\right) \\
\text{and } & h_{t_0, x}: \R_+ \to \R^{1+\nmax}, \;\;  t \mapsto \left(t, \tau - \sum_{k=1}^{n-s} (t-t_0)e_{k}\right).
\end{align*}

The chain rule and a quick computation of the differentiable of $h_{t_0,x}$ yields that for every $t \geq t_0$, 
\begin{equation*}
\frac{d}{dt} g_{t_0,x}(t) = \partial_1 f_{n,s}(h_{t_0,x}(t)) - \sum_{k=1}^{n-s} \partial_{k+1}f_{n,s}(h_{t_0,x}(t)). 
\end{equation*}
Notice that on the one hand, $\partial_1 f_{n,s}(u,\mathbf{v}) = \partial_t f(u,(n,s,\mathbf{v}))$ and on the other, $k \geq 2$, $\partial_k f_{n,s}(u,\mathbf{v}) = \partial_{\tau_{k-1}} f(u,(n,s,\mathbf{v}))$. As a consequence, we have shown that for any $(t_0, x) \in \R_+ \times E$, for every $t \geq t_0$, $g_{t_0,x}$ is differentiable at $t$ and satisfies 
\begin{equation*}
\frac{d}{dt} g_{t_0,x} (t) = \A f\left(t, \Psi(x, t, t_0)\right).
\end{equation*}
This concludes the proof.
\end{proof}

Throughout the section, we use the notation $\mathcal{S} = \{H,W,G\}$. Also, for any $f \in \C^1_b(\R_+ \times E, \R)$, let $f_t(x) = f(t,x)$ for any $(t,x) \in \R_+ \times E$. Finally, we define for any continuous bounded function $g: \R_+ \times E \to \R$, for any $t \geq 0$ and $u = (\bm{\theta}, k, \ell, \sigma) \in \bigcup_{Y \in \mathcal{S}} U_Y$: 
\begin{equation*}
 \begin{aligned}
    g^H_{t,u} &= g(t, \mathfrak{j}(x^H_k(t),\sigma)) - g(t, x^H_k(t))
    \quad \text{and} \quad g^W_{t,u} &= g(t, \mathfrak{j}(x^W_\ell(t),\sigma)) - g(t, x^W_\ell(t)).
\end{aligned}
\end{equation*}

\begin{prop}
\label{prop:zeta(lip)}
Consider $\zeta^K$ as introduced in Proposition \ref{def:zeta}. For any $f \in \C^1_b(\R_+ \times E, \R)$, $T \geq 0$ and $X \in \{H,W\}$,
{\small
\begin{equation*}
\begin{aligned}
    \angles{\zeta^{X|K}_T}{f_T} = \angles{\zeta^{X|K}_0\!}{f_0} + \int_0^T \! \angles{\zeta^{X|K}_t}{\A f_t}dt + \frac{1}{K_X} \sum_{Y \in \mathcal{S}} \int_0^T \int_{U_Y} \mathcal{I}_Y(t-,u) f^X_{t-,u} Q_Y(dt,du).
\end{aligned}
\end{equation*}
}
\end{prop}

\begin{proof}
    
    Let $f \in \C^1_b(\R_+ \times E, \R)$ and $X \in \{H,W\}$. Recall that, by definition, for any bounded function $g: E \to \R$, for any $T \geq t \geq 0$ and $u = (\bm{\theta}, k, \ell, \sigma) \in \bigcup_{Y \in \mathcal{S}} U_Y$:
    \begin{equation*}
            \angles{\Delta_H(u,T,t)}{g} = g(\Psi(\mathfrak{j}(x^H_k(t-),\sigma),T,t)) - g(\Psi(x^H_k(t-),T,t)), 
    \end{equation*}
    and $\angles{\Delta_W(u,T,t)}{g}$ is defined analogously, by replacing $H$ by $W$ and $k$ by $\ell$. 
    
    From Equation \eqref{eq:def-zeta}, it follows that 
    {\small
    \begin{equation*}
            \angles{\zeta^{X|K}_T\!}{f_T} \! = \! \frac{1}{K_X} \!\! \left( \sum_{j=1}^{K_X} f_T\left(\Psi(x^X_j(0),T,0) \right) \! + \!\! \sum_{Y \in \mathcal{S}} \int_0^T \!\!\! \int_{U_Y} \!\!\! \mathcal{I}_Y(t-,u)\angles{\Delta_X(u,T,t)}{f_T} Q_Y(dt,du)\!\! \right)\!\!.
    \end{equation*}
    } 
    Using the result from Lemma \ref{lem:dphif}, this becomes:
    \begin{equation*}
        \begin{aligned}
            \angles{\zeta^{X|K}_T}{f_T} &= \frac{1}{K_X} \sum_{j=1}^{K_X} \left(f_0(x^X_j(0)) + \int_0^T \A f_t(\Psi(x^X_j(0),t,0)) dt \right) \\
            &+ \frac{1}{K_X} \sum_{Y \in \mathcal{S}} \int_0^T \int_{U_Y} \mathcal{I}_Y(t-,u)\left(\int_t^T \angles{\Delta_X(u,z,t)}{\A f_{z}} dz \right) Q_Y(dt,du) \\
            &+ \frac{1}{K_X} \sum_{Y \in \mathcal{S}} \int_0^T \int_{U_Y} \mathcal{I}_Y(t-,u)f^X_{t-,u} Q_Y(dt,du). \\
        \end{aligned}
    \end{equation*}
    It follows from the definition of $\C^1_b(\R_+ \times E,\R)$ that both $f$ and $\A f$ are bounded, hence we may apply Fubini's theorem to obtain that 
    {\small
     \begin{equation*}
        \begin{aligned}
            &\angles{\zeta^{X|K}_T}{f_T} = \frac{1}{K_X} \sum_{j=1}^{K_X} f_0(x^X_j(0)) + \frac{1}{K_X} \sum_{Y \in \mathcal{S}} \int_0^T \int_{U_Y} \mathcal{I}_Y(t-,u)f^X_{t-,u} Q_Y(dt,du)\\
            &+ \frac{1}{K_X} \int_0^T \! \left(\sum_{j=1}^{K_X} \A f_{z}(\Psi(x^X_j(0),z,0)) + \sum_{Y \in \mathcal{S}} \int_0^{z} \!\!\! \int_{U_Y} \! \mathcal{I}_Y(t-,u)\angles{\Delta_X(u,z,t)}{\A f_{z}} Q_Y(dt,du) \right) \!\! dz. \\
        \end{aligned}
    \end{equation*}
    }
    The first sum on the right-hand side equals $\angles{\zeta^{X|K}_0}{f_0}$. From the second line, one recognizes in the integrand the definition of $\angles{\zeta^{X|K}_{z}}{\A f_{z}}$ from Equation \eqref{eq:def-zeta}. This yields the desired result.
\end{proof}

For $Y \in \mathcal{S}$, let 
\begin{equation*}
    \widetilde{Q}_Y(dt, du) = Q_Y(dt, du) - dt \mu_Y(du) 
\end{equation*} 
be the compensated martingale-measure associated to $Q_Y$. 

It follows that, for $f \in \C^1_b(\R_+ \times E, \R)$ and $X \in \{H,W\}$, we have the following semimartingale decomposition: 
\begin{equation}
   \label{eq:semimartingale}
    \angles{\zeta^{X|K}_T}{f_T} = M^{X|K}_T(f) + V^{X|K}_T(f),
\end{equation}
where we define
%\begin{equation*}
%    M^{X|K}_T(f) = \frac{1}{K_X} \sum_{Y \in \mathcal{S}} \int_0^T \int_{U_Y} \mathcal{I}_Y(t-,u)f^X_{t-,u} \widetilde{Q}_Y(dt,du)
%\end{equation*}
\begin{equation*}
    M^{X|K}_T(f) = \sum_{Y \in \mathcal{S}} M^{X|K}_{Y,T}(f) \quad{with} \quad M^{X|K}_{Y,T}(f) = \int_0^T \int_{U_Y} \frac{1}{K_X} \mathcal{I}_Y(t-, u) f^X_{t,u} \widetilde{Q}_Y(dt,du),
\end{equation*}
and
\begin{equation*}
    V^{X|K}_T(f) = \angles{\zeta^{X|K}_0}{f_0} + \int_0^T \angles{\zeta^{X|K}_t}{\A f_t} dt + \frac{1}{K_X} \sum_{Y \in \mathcal{S}} \int_0^T \int_{U_Y} \mathcal{I}_Y(t,u) f^X_{t,u} \mu_Y(du)dt.
\end{equation*}

\begin{prop}
\label{prop:crochet}
    Let $f \in \C^1_b(\R_+ \times E,\R)$ and $X \in \{H,W\}$. For all $Y \in \mathcal{S}$, $(M^{X|K}_{Y,T}(f))_{T\geq 0}$ is a square integrable martingale, and the following inequalities are satisfied:
     \begin{equation}
    \label{eq:aux-crochet}
    \begin{aligned}
        \forall Y \in \{H, W\}, \quad \E\left[\langle M^{X|K}_Y(f) \rangle_T\right] & \leq \frac{4}{K} \lambda_Y (\nmax)^3 \inftynorm{f}^2 T \\
        \text{and} \quad  \E\left[\langle M^{X|K}_G(f) \rangle_T\right] & \leq \frac{4}{K} \beta_G (\nmax)^2 \inftynorm{f}^2 T.
    \end{aligned}
    \end{equation}
    Thus $(M^{X|K}_T(f))_{T \geq 0}$  
    is itself a square integrable martingale. Using the same notations as in Theorem \ref{thm:cvgce}, its quadratic variation is given by 
    \begin{equation*}
        \langle M^{X|K}(f) \rangle_T = \frac{1}{K_X} \int_0^T \angles{\zeta^{X|K}_t}{\mathcal{H}^X_t((f^2_t)^\I - 2 f^\I_t f_t + f^2_t)}dt,
    \end{equation*}
    where for any $t \geq 0$ and $x \in E$, 
    \begin{equation*}
        \mathcal{H}^X_t(x) = \beta_G \frac{I_H(t)}{N_H} \bfs(x) + \lambda_X \bfs(x) \bfi(x) + \lambda_{\overline{X}} \frac{\angles{\zeta^{\overline{X}|K}_t}{\bfs \bfi}}{S_{\overline{X}}(t)} \bfs(x).
    \end{equation*}
\end{prop}

\begin{proof}
    Let $f \in \C^1_b(E,\R)$ and $X \in \{H,W\}$, and consider $M^{X|K}_T(f)$.
%    Consider $M^{X|K}_T(f)$, which can be written as
%    \begin{equation*}
%        M^{X|K}_T(f) = M^{X|K}_{X,T}(f) + M^{X|K}_{\overline{X},T}(f) + M^{X|K}_{G,T}(f),
%    \end{equation*}
%    where, for $Y \in \mathcal{S}$ and $T \geq 0$,
%    \begin{equation*}
%        M^{X|K}_{Y,T}(f) = \int_0^T \int_{U_Y} \frac{1}{K_X} \mathcal{I}_Y(t-, u) f^X_{t,u} \widetilde{Q}_Y(dt,du). 
%    \end{equation*}
    Suppose that for any $Y \in \mathcal{S}$, 
    \begin{equation*}
       \E\left[ \int_0^T \int_{U_Y} \left(\frac{1}{K_X} \mathcal{I}_Y(t, u) f^X_{t,u}\right)^2 \mu_Y(du) dt\right] < \infty , 
    \end{equation*}
    then for all $Y \in \mathcal{S}$, $(M^{X|K}_{Y,T}(f))_{T\geq 0}$ is a square integrable martingale \cite{meleardStochasticModelsStructured2015}, implying that $(M^{X|K}_T(f))_{T\geq0}$ also is a square integrable martingale. 
    As $Q^{H|K}$, $Q^{W|K}$ and $Q^{G|K}$ are independent, it follows that the quadratic variation of $M^{X|K}_T(f)$ satisfies %{\sc (Prop 3.16 poly p. robert)}
    \begin{equation*}
        \langle M^{X|K}(f) \rangle_T = \sum_{Y \in \mathcal{S}}  \langle M^{X|K}_Y(f) \rangle_T.
    \end{equation*}
    It thus is enough to study $(M^{X|K}_{Y,T}(f))_{T \geq 0}$ for all $Y \in \mathcal{S}$. In the following, we will detail the necessary computations in the case $X = H$, the case $X = W$ being similar.
    
    Consider the case $Y = H$. Start by noticing that $\sum_{\ell = 1}^{K_W} s^W_\ell(t) = K_W S_W(t)$, and that for any $k \in \bbrackets{1}{K_H}$ and $t \in [0,T]$, $s^H_k(t)$ and $i^H_k(t)$ are less then $\nmax$, almost surely. Hence, replacing $S(t)$ by $K_W S_W(t)$ in $\I_H$, 
    \begin{equation*}
    \begin{aligned}
    \label{eq:aux-crochet-bis}
        \E&\left[\langle M^{H|K}_H(f) \rangle_T\right] = \E\left[ \int_0^T \int_{U_H} \left(\frac{1}{K_H} \mathcal{I}_H(t,u)f^H_{t,u}\right)^2 \mu_H(du) dt \right] \\
        & = \E\left[\int_{0}^T \frac{1}{K_H^2} \sum_{k=1}^{K_H} \lambda_H s^H_k(t) i(\tau^H_k(t)) \angles{\nu}{\left(f_t(\mathfrak{j}(x^H_k(t),\cdot)) - f_t(x^H_k(t))\right)^2} dt\right] \\
        &\leq \frac{1}{K_H} \lambda_H (\nmax)^2 4\inftynorm{f}^2 T.
    \end{aligned}
    \end{equation*}
    Since further $K_H \geq K/\nmax$ and $\inftynorm{f}^2 < \infty$, we obtain that 
    \begin{equation*}
    %\label{eq:aux-crochet}
        \E\left[\langle M^{H|K}_H(f) \rangle_T\right] \leq \frac{4}{K} \lambda_H (\nmax)^3 \inftynorm{f}^2 T < \infty.
    \end{equation*}
    Thus $(M^{H|K}_{H,t}(f))_{t \geq 0}$ is a square integrable martingale whose quadratic variation is given by
    \begin{equation*}
    \begin{aligned}
        \langle M^{H|K}_H(f) \rangle_T & = \int_0^T \int_{U_H} \left(\frac{1}{K_H} \mathcal{I}_H(t,u)f^H_{t,u}\right)^2 \mu_H(du) dt \\
        &= \frac{\lambda_H}{K_H} \int_0^T \!\!\! \angles{\zeta^{H|K}_t}{\bfs \bfi ((f^2_t)^\I - 2 f^\I_t f_t + f^2_t)} dt,
    \end{aligned}
    \end{equation*}
    using the computations from Equation \eqref{eq:aux-crochet-bis}.

    Similarly, $(M^{H|K}_{W,T}(f))_{T \geq 0}$ is a square integrable martingale of quadratic variation
    \begin{equation*}
        \langle M^{H|K}_W(f) \rangle_T =  \lambda_W \int_0^T  \frac{\angles{\zeta^{W|K}_t}{\bfs \bfi}}{S_W(t)} \angles{\zeta^{H|K}_t}{\bfs((f^2_t)^\I - 2 f^\I_t f_t + f^2_t)} dt.
    \end{equation*}

    Further, for the case $Y = G$, let us use the equalities $S(t) = K_H S_H(t)$ and $I(t) = K_H I_H(t)$. As $I(t)/K = I_H(t)/N_H \leq 1$  almost surely, we obtain that 
    \begin{equation*}
        \E\left[\langle M^{H|K}_G(f) \rangle_T\right] = \E\left[ \int_0^T \int_{U_G} \left(\frac{1}{K_H} \mathcal{I}_G(t, u)f^H_{t,u} \right)^2 \mu_G(du) dt\right] \leq \frac{4}{K} \beta_G (\nmax)^2 \inftynorm{f}^2 T.
    \end{equation*}
    As before, $(M^{H|K}_{G,T}(f))_{T \geq 0}$ thus is a square integrable martingale of quadratic variation given by 
    \begin{equation*}
        \langle M^{H|K}_G(f) \rangle_T = \frac{1}{K_H} \beta_G \int_0^T \frac{I_H(t)}{N_H} \angles{\zeta^{H|K}_t}{\bfs((f^2_t)^\I - 2 f^\I_t f_t + f^2_t)}dt.
    \end{equation*}
    This yields the desired result for $(M^{H|K}_T(f))_{T \geq 0}$. Proceeding in the same way for $(M^{W|K}_T(f))_{T \geq 0}$ concludes the proof.
\end{proof}

We are now ready to focus on the tightness of $(\zeta^K)_{K \geq 1}$ in $\D(\R_+, (\M_F(E), w))$. Recall that a sequence of $(\mu_n)_{n \geq 1}$ taking values in $\M_F(E)$  converges in the vague topology to $\mu$ if for any $f \in \C_0(E, \R)$, 
\begin{equation}
\label{eq:topology}
\lim_{n \to \infty} \angles{\mu_n}{f}  = \angles{\mu}{f}. 
\end{equation}
If further Equation \eqref{eq:topology} is true for any $f \in \C_b(E, \R)$, the convergence holds in the weak topology. We thus start by showing tightness of $(\zeta^K)_{K \geq 1}$ in $\D(\R_+, (\M_F(E), v))$, and subsequently extend it to the case where $\M_F(E)$ is endowed by the weak topology $w$ by controlling the mass of the process. 

\begin{prop}
\label{prop:tightness}
    Under the assumptions of Theorem \ref{thm:cvgce}, the sequence $(\zeta^K)_{K \geq 1}$ is tight in $\D\left(\R_+, (\M_F(E), v) \right)^2$.
\end{prop}

The proof relies on the fact that in order to establish tightness of $(\zeta^K)_{K \geq 1}$, it is enough to show that for any $X \in \{H,W\}$, $\left(\angles{\zeta^{X|K}}{f} \right)_{K\geq 1}$ is tight for a large enough set of test functions $f$ \cite{roelly-coppolettaCriterionConvergenceMeasure1986}. This in turn is ensured using the Aldous \cite{aldousStoppingTimesTightness1978} and Rebolledo \cite{joffeWeakConvergenceSequences1986} criteria, whose application is straightforward thanks to the upper bounds established in the previous proof. 

\begin{proof}
    Once more, we will proceed component-wise and show that $\left(\zeta^{H|K}\right)_{K\geq 1}$ and $\left(\zeta^{W|K}\right)_{K\geq 1}$ are both tight in $\D\left(\R_+, (\M_F(E),v)\right)$. 
    
    Let us focus on $\left(\zeta^{H|K}\right)_{K\geq 1}$. According to Theorem 2.1 of \cite{roelly-coppolettaCriterionConvergenceMeasure1986}, it is sufficient to show that for any function $f$ belonging to a dense subset of 
    \begin{equation*}
        \C_0(E, \R) = \left\{ f: E \to \R \text{ continuous s.t.} \lim_{\inftynorm{x} \to \infty} |f(x)| = 0 \right\},
    \end{equation*}
    the sequence $\left(\angles{\zeta^{H|K}}{f} \right)_{K\geq 1}$ is tight in $\D(\R_+, \R)$. 
    Notice that by density of the set of smooth compactly supported functions in $C_0(\R^{\nmax})$ endowed with the uniform norm, it follows that $\C_0(E, \R) \cap \C^1_b(E, \R)$ is also dense in $\C_0(E, \R)$ endowed with the uniform norm. Thus, let us consider $f \in \C_0(E, \R) \cap \C^1_b(E, \R)$. 

    According to the Aldous \cite{aldousStoppingTimesTightness1978} and Rebolledo \cite{joffeWeakConvergenceSequences1986} criteria, in order to prove the tightness of $(\angles{\zeta^{H|K}}{f})_{K\geq 1}$, it is enough to show that:
    \begin{enumerate}[label= (\roman*)]
        \item For any $t$ belonging to a dense subset $\mathcal{T}$ of $\R^+$, $\left(\langle M^{H|K}(f) \rangle_t \right)_{K \geq 0}$ and $(V^{H|K}_t(f))_{K \geq 0}$ are tight in $\R$.
        \item For any $T \geq 0$, for any $\epsilon , \alpha > 0$,  there exist $\delta > 0$ and $K_0 \in \N$ such that for any two sequences of stopping times $(S_K)_{K \geq 1}$ and $(T_K)_{K \geq 1}$ satisfying $S_K \leq T_K \leq T$ for all integers $K$, 
        \begin{equation*}
        \begin{aligned}
            \sup_{K \geq K_0} \P\left(|\langle M^{H|K}(f) \rangle_{S_K} - \langle M^{H|K}(f) \rangle_{T_K}| \geq \alpha, T_K \leq S_K + \delta \right) &\leq \epsilon \\
            \text{and } \sup_{K \geq K_0} \P\left(|V^{H|K}_{S_K}(f) - V^{H|K}_{T_K}(f)| \geq \alpha, T_K \leq S_K + \delta \right) &\leq \epsilon.
        \end{aligned}
        \end{equation*}
    \end{enumerate}

    Notice that, in order to establish (i), it is enough to show that for any $t \geq 0$, 
    \begin{equation*}
        \sup_{K \geq 1} \E\left[ |\langle M^{H|K}(f) \rangle_t| \right] < \infty \text{ and } \sup_{K \geq 1} \E\left[ \left|V^{H|K}_t(f)\right| \right] < \infty.
    \end{equation*}
    Recalling that $C = 2 \nmax \left(\lambda_H \nmax + \lambda_W \nmax + \beta_G \right)$, it follows from Equation \eqref{eq:aux-crochet} that
    \begin{equation*}
        \E[|\langle M^{H|K}(f) \rangle_t|] \leq \frac{1}{K} 2 \nmax C \inftynorm{f}^2 t.
    \end{equation*}
    Similar computations yield that 
    \begin{equation*}
        \E[|V^{H|K}_t(f)|] \leq \inftynorm{f} + \inftynorm{\A f} + C \inftynorm{f} t.
    \end{equation*}
    As $f \in \C^1_b(E,\R)$, this implies that (i) holds.

    It remains to check (ii). Let $\epsilon, \alpha > 0$, and consider two sequences of stopping times $(S_K)_{K \geq 1}$ and $(T_K)_{K \geq 1}$ satisfying $S_K \leq T_K \leq T$ for all integers $K$. 
    As previously, using Equation \eqref{eq:aux-crochet}, we obtain the following upper bound:
    {\small
    \begin{equation*}
    \begin{aligned}
         \E\left[|\langle M^{H|K}(f) \rangle_{S_K} \! - \langle M^{H|K}(f) \rangle_{T_K}|\Big| T_K \leq S_K +\!\delta \right] & \leq \E\left[\!\int_{S_K}^{T_K} \!\!\! dt \Big| T_K \leq S_K +\! \delta\right] \frac{2}{K} \nmax C \inftynorm{f}^2 \\
         &\leq  \frac{\delta}{K} 2 \nmax C \inftynorm{f}^2.
    \end{aligned}
    \end{equation*}
    }
    Hence, using conditional Markov's inequality, 
    \begin{equation}
    \label{eq:aux3}
        \P\left(|\langle M^{H|K}(f) \rangle_{S_K} - \langle M^{H|K}(f) \rangle_{T_K}| \geq \alpha, T_K \leq S_K + \delta \right) \leq \frac{\delta}{\alpha K} 2 \nmax C \inftynorm{f}^2.
    \end{equation}

    Proceeding similarly, we also obtain that 
    \begin{equation}
    \label{eq:aux4}
        \P\left(|V^{H|K}_{S_K}(f) - V^{H|K}_{T_K}(f)| \geq \alpha, T_K \leq S_K + \delta \right) \leq \frac{\delta}{\alpha}  \left(\inftynorm{\A f} + C \inftynorm{f}\right).
    \end{equation}

    Equations \eqref{eq:aux3} and \eqref{eq:aux4} imply the existence of $\delta$ and $K_0$ such that (ii) is satisfied. Naturally, $\zeta^{W|K}$ can be handled analogously. This concludes the proof.
\end{proof}

Finally, this result on the tightness of $(\zeta^K)_{K \geq 1}$ in $\D(\R_+, (\M_F(E), v))^2$ lets us establish the main result of this subsection. 

\begin{prop}
\label{prop:tightness-weak}
    Under the assumptions of Theorem \ref{thm:cvgce}, the sequence $(\zeta^K)_{K \geq 1}$ is tight in $\D(\R_+, (\M_F(E), w))^2$. In addition, any adherence value belongs to $\C([0,T], (\M_F(E), w))^2$.
\end{prop}

\begin{proof}
    Let $X \in \{H,W\}$. Tightness in $\D(\R_+, (\M_F(E), w))$ of $(\zeta^{X|K})_{K \geq 1}$ will be shown using Theorem 1.1.8 from \cite{tranBalladeForetsAleatoires2014}, which we state in our setting for the sake of completeness. Let $\Phi : z \in \R \mapsto 6z^2 - 15z^4 + 10z^3$ and for $N \geq 1$, define smooth approximations of $x \in  E \mapsto \setind{\inftynorm{\tau(x)} \geq N}$ by:
    \begin{equation*}
        \forall x \in E, \forall N \geq 1, \phi_N(x) = \Phi(0 \vee (\inftynorm{\tau(x)} - (N-1)) \wedge 1 ). 
    \end{equation*}
    Then in order to ensure the tightness of $(\zeta^{X|K})_{K \geq 1}$ in $\D(\R_+, (\M_F(E), w))$, it is sufficient to show that for any $T \geq 0$, the following conditions hold: 
    \begin{enumerate}[label= (\roman*)]
        \item \label{hyp1} There exists a family of functions $F$ which is dense in $\C_0(E, \R)$ and stable under addition, such that for any $f \in F \cup \{x \in E \mapsto 1\}$, the sequence $(\angles{\zeta^{X|K}}{f})_{K \geq 1}$ is tight in $\D(\R_+, \R)$. 
        \item \label{hyp2}
        \begin{equation*}
            \lim_{N \to \infty} \limsup_{K \to \infty} \E[\sup_{t \leq T} \; \angles{\zeta^{X|K}_t}{\phi_N}] = 0.
        \end{equation*}
        \item \label{hyp3} Any limiting value of $(\zeta^{X|K})_{K \geq 1}$, if it exists, belongs to $\C([0,T], (\M_F(E),w))$.
    \end{enumerate}
    The proof hence consists in checking those assumptions. Let $T \geq 0$, and consider $X = H$, as the case $X = W$ can be treated similarly. We may see that \ref{hyp1} is satisfied, as we have shown in the proof of Proposition \ref{prop:tightness} that for any $f \in \C_0(E, \R) \cap \C^1_b(E, \R)$, $(\angles{\zeta^{H|K}}{f})_{K \geq 1}$ is tight, and further for any $K \geq 1$, for any $T \geq 0$, $\angles{\zeta^{H|K}_T}{1} = 1$ almost surely.

    Let us now turn our attention to \ref{hyp2}. Start by noticing that for any $N \geq 1$ and $x \in E$, 
    \begin{equation*}
        \phi_N(x) \leq \setind{\inftynorm{\tau(x)} \geq {N-1}} \leq f_{N-1}(x) \coloneqq \sum_{i = 1}^{\nmax} f_{N-1, i}(x)
    \end{equation*}
    where $f_{N-1,i}(x) = \setind{\bfn(x) - \bfs(x) \geq i, \; |\tau_i(x)| \geq N - 1}$.

    Let $t \in [0,T]$. For any $N - 1 \geq t$, $x \in E$, $z \in [0,t]$ and $\sigma \geq 0$, it holds by definition that 
    {\small
    \begin{equation*}
        f_{N-1,i}(\Psi(\mathfrak{j}(x, \sigma), t, z)) - f_{N-1,i}(\Psi(x, t, z)) = \setind{\bfn(x) - \bfs(x) = i-1, \; |\sigma - (t-z)| > N-1} \leq \setind{\sigma > N-1}.
    \end{equation*}
    }
    Hence, using Proposition \ref{def:zeta} and the above upper bounds, it follows that almost surely,
    {\small
    \begin{equation*}
            \angles{\zeta^{H|K}_t}{\phi_N} \leq \frac{1}{K_H} \sum_{k=1}^{K_H} f_{N-1}(\Psi(x^H_k(0), t, 0))
            + \frac{\nmax}{K_H} \sum_{Y \in \mathcal{S}} \int_0^t \! \int_{U_Y} \!\I_Y(z-,u) \setind{\sigma > N} Q^K_Y(dz, du).
    \end{equation*}
    }
    Defining as previously $C = 2 \nmax \left(\lambda_H \nmax + \lambda_W \nmax + \beta_G \right)$, this leads to the following upper bound:
    \begin{equation*}
        \E[\sup_{t \leq T} \angles{\zeta^{H|K}_t}{\phi_N}] \leq \E[\sup_{t \leq T} \frac{1}{K_H} \sum_{k=1}^{K_H} f_{N-1}(\Psi(x^H_k(0), t, 0))] + \frac{CT}{2} \nmax \nu((N-1, +\infty)).
    \end{equation*}
    As a consequence, Assumption \ref{hyp:zetaK0} (i) ensures that \ref{hyp2} is satisfied. Notice that this assumption could actually be a little bit relaxed here, as it would be enough if the supremum over $K$ were replaced by the limit superior over $K \to \infty$. 

    In order to check that condition \ref{hyp3} holds, we will follow the arguments presented in \cite{jourdainLevyFlightsEvolutionary2012}. Suppose that $\eta^H$ is a limiting value of $(\zeta^{H|K})_{K \geq 1}$. By definition, 
    \begin{equation*}
        \sup_{t \in [0,T]} \sup_{f \in L^\infty, \inftynorm{f} \leq 1} |\angles{\zeta^{H|K}_t}{f} - \angles{\zeta^{H|K}_{t-}}{f}| \leq \frac{1}{K_H}.
    \end{equation*}
    As the application $\mu \mapsto \sup_{t \in [0,T]} |\angles{\mu_t}{f} - \angles{\mu_{t-}}{f}|$ is continuous on $\D([0,T], (\M_F(E), v))$ for any $f$ in a measure-determining countable set, it follows that $\eta^H \in \C([0,T], (\M_F(E), v))$. 

    Let us now introduce $\phi_{N,M} = \phi_N(1 - \phi_M)$, which serves as a smooth and compactly supported approximation of $x \in E \mapsto \setind{N \leq \inftynorm{\tau(x)} \leq M}$, for $N \leq M$. As, on the one hand, $\mu \mapsto \sup_{t \in [0,T]} \angles{\mu_t}{\phi_{N,M}}$ is continuous on $\D([0,T], (\M_F(E),v))$, and on the other hand, for any $K \geq 1$, $\sup_{t \in [0,T]} \angles{\zeta^{X|K}_t}{\phi_{N,M}} \leq 1$, it follows that:
    \begin{equation*}
        \E[\sup_{t \in [0,T]} \angles{\eta^H_t}{\phi_{N,M}}] = \lim_{K \to \infty} \E[\sup_{t \in [0,T]} \angles{\zeta^{X|K}_t}{\phi_{N,M}}] \leq \limsup_{K \to \infty} \E[\sup_{t \in [0,T]} \angles{\zeta^{X|K}_t}{\phi_{N}}].
    \end{equation*}
    Letting $M$ go to infinity in the left hand side, dominated convergence ensures that 
    \begin{equation*}
        \E[\sup_{t \in [0,T]} \angles{\eta^H_t}{\phi_{N}}] \leq \limsup_{K \to \infty} \E[\sup_{t \in [0,T]} \angles{\zeta^{X|K}_t}{\phi_{N}}] \xrightarrow[N \to \infty]{} 0,
    \end{equation*}
    where the convergence of the right hand side is achieved as in the proof of \ref{hyp2}. In particular, it thus is possible to extract a subsequence from $(\sup_{t \in [0,T]} \angles{\eta^H_t}{\phi_{N}})_N$ which converges almost surely to zero. This implies that $(\eta^H_t)_{t \in [0,T]}$ is almost surely tight: for any $\epsilon$, there exists $N$ such that almost surely, 
    \begin{equation*}
        1 - \sup_{t \in [0,T]} \angles{\eta^H_t}{\setind{\inftynorm{\tau(\cdot)} \leq N}} \leq \sup_{t \in [0,T]} \angles{\eta^H_t}{\phi_{N}} < \epsilon.
    \end{equation*}  

    Let $g \in \C_b(E)$, and let $g_N = g(1 - \phi_N)$. It then holds that for $h$ small so that $t+h \in [0,T]$,
    \begin{equation*}
        |\angles{\eta^H_{t+h}}{g} - \angles{\eta^H_t}{g}| \leq |\angles{\eta^H_{t+h}}{g - g_N}| + |\angles{\eta^H_{t+h}}{g_N} - \angles{\eta^H_t}{g_N}| + |\angles{\eta^H_t}{g - g^N}|.
    \end{equation*}
    Let $\epsilon > 0$. As $|g - g_N| \leq \inftynorm{g} \phi_N$, there exists $N_0$ such that $\sup_{t \in [0,T]} \angles{\eta^H_t}{g-g_{N_0}} < \epsilon/3$. Further, as $\eta^H \in \C([0,T],(\M_F(E),v))$, for $h$ small enough, $|\angles{\eta^H_{t+h}}{g_{N_0}} - \angles{\eta^H_t}{g_{N_0}}| < \epsilon/3$. This allows to conclude that $\eta^H \in \C([0,T],(\M_F(E),w))$, establishing \ref{hyp3} and finally tightness of $\zeta^{H|K}$ in $\D(\R_+, (\M_F(E),w))$.
\end{proof}

\subsubsection{Identification of the limiting values of \texorpdfstring{$(\zeta^K)_{K \geq 1}$}{the individual-based process}}

The tightness of $(\zeta^K)_{K \geq 1}$ in the space $\D(\R_+, \M_F(E), w))^2$ ensures that from any subsequence of $(\zeta^K)_{K \geq 1}$, one may extract a subsubsequence which converges in this space. It remains to characterize the the limits of these subsubsequences. 

\begin{prop}
\label{prop:idlimit}
    Under the assumptions of Theorem \ref{thm:cvgce}, all limiting values of $(\zeta^K)_{K \geq 1}$ in $\D\left(\R_+, (\M_F(E), w) \right)^2$ are continuous with regard to the total variation norm, and solutions of Equation \eqref{eq:defeta}.
\end{prop}

This result builds on the intuition that in Equation \eqref{eq:semimartingale}, the martingale part vanishes, whereas the bounded variation part leads to the deterministic limit characterized by Equation \eqref{eq:defeta}. Before proceeding to the proof of this proposition, let us emphasize that there is some technical difficulty due to infectiousness being a discontinuous function of an individual's remaining infectious period. Indeed, tightness of $(\zeta^K)_{K \geq 1}$ in $\D(\R_+, (\M_F(E), w))^2$ allows us to extract a subsequence $(\zeta^{\varphi(K)})_{K \geq 1}$ which converges in law in this space to some limiting value $\eta$, and our aim is to show that $\eta$ satisfies Equation \eqref{eq:defeta}. However, convergence in law in $\D(\R_+, (\M_F(E), w))$ is not enough to ensure that $\angles{\zeta^{X|\varphi(K)}_t}{\bfi f}$ converges in law to $\angles{\eta^X_t}{\bfi f}$ for $f \in \C^1_b(E)$, as $\bfi$ is discontinuous on $E$. This leads to forthcoming Proposition \ref{prop:cvgce-stepfct}.

\begin{proof}
    Consider a subsequence $(\zeta^{\varphi(K)})_{K \geq 1}$ of $(\zeta^K)_{K \geq 1}$ which converges in law in the space $\D\left(\R_+, (\M_F(E), w) \right)^2$, and let $\eta$ be its limit. 
    
    Notice that it follows from Proposition \ref{prop:tightness-weak} that $\eta \in \C(\R_+, (\M_F(E),w))^2$ almost surely. Hence, following Proposition A.6.1 of \cite{tranModelesParticulairesStochastiques2006}, for any $X \in \{H,W\}$,
    \begin{equation*}
        ||\eta^{X}_T - \eta^{X}_{T-} ||_{TV} = \sup_{f \in \C_b(E,\R): \inftynorm{f} \leq 1} \left| \angles{\eta^{X}_T}{f} - \angles{\eta^{X}_{T-}}{f} \right| = 0 \;\; \text{almost surely}.
    \end{equation*}

    It remains to show that $\eta$ satisfies Equation \eqref{eq:defeta}. Let $T \geq 0$ and $f \in \C^1_b(\R_+ \times E,\R)$, and consider the application $\psi^H_T$ defined by 
    \begin{equation}
    \label{eq:def-psi(zeta)}
        \begin{aligned}
            \psi^H_T(\eta) & = \angles{\eta^H_T}{f_T} - \angles{\eta^H_0}{f_0} - \int_0^T \angles{\eta^H_t}{\A f_t} dt 
            - \int_0^T \angles{\eta^H_t}{\lambda_H \bfs \bfi (f^\I_t - f_t)} dt \\
            &- \int_0^T \frac{1}{\angles{\eta^W_t}{\bfs}}\angles{\eta^W_t}{\lambda_W \bfs \bfi} \angles{\eta^H_t}{\bfs (f^\I_t - f_t)} dt
            - \beta_G \int_0^T \frac{\angles{\eta^H_t}{\bfi}}{\angles{\eta^H_0}{\bfn}} \angles{\eta^H_t}{\bfs (f^\I_t - f_t)} dt.
        \end{aligned}
    \end{equation}
    Start by noticing that $\psi^H_T(\zeta^{\varphi(K)}) = M^{H|\varphi(K)}_T(f)$, as $K_H S_H(t) = K_W S_W(t)$. Using Jensen's inequality, it follows from Equation \eqref{eq:aux-crochet} that
    \begin{equation*}
        \E[|\psi^H_T(\zeta^{\varphi(K)})|]^2 \leq \E[|\psi^H_T(\zeta^{\varphi(K)})|^2] = \E[\langle M^{H|K}(f) \rangle_t] \leq \frac{1}{K} 2 \nmax C \inftynorm{f}^2 T \xrightarrow[K \to \infty]{} 0.
    \end{equation*}

    Suppose that $(\psi^H_T(\zeta^{\varphi(K)}))_{K \geq 1}$ converges in law to $\psi^H_T(\zeta)$. According to Theorem 3.5 of \cite{billingsleyConvergenceProbabilityMeasures1999}, it then is enough to proof that $(\psi^H_T(\zeta^{\varphi(K)}))_{K \geq 1}$ is uniformly integrable to obtain that its expectation converges to the expectation of $\psi^H_T(\zeta)$. In our case, uniform integrability is easily assured as the sequence $(\psi^H_T(\zeta^{\varphi(K)}))_{K \geq 1}$ is bounded. Indeed, using the fact that for all $T \geq 0$, $\zeta^{\varphi(K)}_T \in \mathfrak{M}_1$, we obtain from Equation \eqref{eq:def-psi(zeta)} that 
     \begin{equation*}
         \left|\psi^H_T(\zeta^{\varphi(K)})\right| \leq \left((2+C)\inftynorm{f} + \inftynorm{\A f} \right)T.
     \end{equation*}
     We may now conclude that 
     \begin{equation*}
         \E[|\psi^H_T(\eta)|] = \lim_{K \to \infty} \E[|\psi^H_T(\zeta^{\varphi(K)})|] = 0,
     \end{equation*}
     which yields the desired result.

    It thus suffices to show that  $(\psi^H_T(\zeta^{\varphi(K)}))_{K \geq 1}$ converges in law to $\psi^H_T(\eta)$. According to Skorokhod's representation theorem, there exists a probability space $\Omega$ on which one may define $(\tilde{\zeta}^K)_{K \geq 1}$ and $\tilde{\eta}$ equal in law to $(\zeta^{\varphi(K)})_{K \geq 1}$ and $\eta$, respectively, such that $(\tilde{\zeta}^K)_{K \geq 1}$ converges almost surely in $\D(\R_+, (\M_F(E),w))^2$ to $\tilde{\eta}$ on $\Omega$. In particular, it holds that 
    \begin{equation*}
        \forall T \geq 0, \forall X \in \{H,W\}, \forall g \in \C_b(E),\quad \angles{\tilde{\zeta}^{X|K}_T}{g} \xrightarrow[K \to \infty]{} \angles{\tilde{\eta}^X_T}{g} \;\; \text{almost surely}.
    \end{equation*}
    It follows immediately that for any $t \in [0,T]$ and $X \in \{H,W\}$, almost surely,
    \begin{equation}
    \label{eq:aux-psiA}
        (\angles{\tilzeta^{X|K}_t}{f_t}, \angles{\tilzeta^{X|K}_t}{\A f_t}) \xrightarrow[K \to \infty]{} (\angles{\tileta^X_t}{f_t}, \angles{\tileta^X_t}{\A f_t}).
    \end{equation}
    Since $|\angles{\tilzeta^{X|K}_t}{\A f_t}| \leq 2\inftynorm{Df}$, where $Df$ designates the differential of $f$, dominated convergence ensures that
    \begin{equation}
    \label{eq:aux-psiB}
        \int_0^T \angles{\tilzeta^{X|K}_t}{\A f_t} dt \xrightarrow[K \to \infty]{} \int_0^T \angles{\tileta^X_t}{\A f_t} dt \;\; \text{almost surely}.
    \end{equation}

    In order to establish the desired convergence of the last three terms of $\psi^H_T(\tilzeta)$, we will make use of the following proposition, whose proof is postponed. In this context, a $d$-dimensional rectangle is a set defined as the product of $d$ intervals of $\R \cup \{-\infty, +\infty\}$. 

    \begin{prop}
    \label{prop:cvgce-stepfct}
        For any $n \in \bbrackets{1}{\nmax}$, for any $s \in \bbrackets{0}{n}$, consider $m = m(n,s) < \infty$, a set $(A^{n,s}_k)_{k \leq m}$ of pairwise disjoint $(n-s)$-dimensional rectangles and a set $(\varphi^{n,s}_k)_{k \leq m}$ of functions belonging to $C^1_b(\R^{n-s})$. For any $\tau \in \R^{\nmax}$, let $\tau_{1,n-s} = (\tau_1, \dots, \tau_{n-s})$. Define the function $\phi : E \to \R$ by
        \begin{equation*}
            \forall x = (n,s,\tau) \in E, \quad \phi(x) = \sum_{k=1}^{m(n,s)} \ind_{A^{n,s}_k}(\tau_{1,n-s}) \varphi^{n-s}_k(\tau_{1,n-s}).
        \end{equation*}
        Then for any $X \in \{H,W\}$ and $T \geq 0$, it holds that 
        \begin{equation*}
            \angles{\tilde{\zeta}^{X|K}_T}{\phi} \xrightarrow[K \to \infty]{} \angles{\tilde{\eta}^X_T}{\phi} \;\; \text{in } L^1.
        \end{equation*}
        
    \end{prop}

    Let us focus on the second-to-last term, representing infection events occurring within workplaces, as the other two can be treated similarly. 
    
    The application $\phi(x) = \lambda_W \bfs(x) \bfi(x)$ is of the form described in Proposition \ref{prop:cvgce-stepfct}, hence for any $t \in [0,T]$, $\angles{\tilzeta^{W|K}_t}{\lambda_W \bfs \bfi}$ converges in $L^1$ to $\angles{\tileta^{W}_t}{\lambda_W \bfs \bfi}$ as $K$ tends to infinity. Also, notice that as $f \in \C^1_b(\R_+ \times E, \R)$, it follows that for any $t \in [0,T]$, $f^\I_t \in \C^1_b(E)$. Thus Proposition \ref{prop:cvgce-stepfct} ensures that for any $t \in [0,T]$, $\angles{\tilzeta^{H|K}_t}{\bfs(f^\I_t - f_t)}$ converges in $L^1$ to $\angles{\tileta^{H}_t}{\bfs(f^\I_t - f_t)}$ as $K$ tends to infinity. 

    In particular, the following convergence holds in probability:
    \begin{equation*}
        X^K_t \coloneqq \angles{\tilzeta^{W|K}_t}{\lambda_W \bfs \bfi} \angles{\tilzeta^{H|K}_t}{\bfs(f^\I_t - f_t)} \xrightarrow[K \to \infty]{} X_t \coloneqq \angles{\tileta^{W}_t}{\lambda_W \bfs \bfi} \angles{\tileta^{H}_t}{\bfs(f^\I_t - f_t)}.
    \end{equation*}
    Letting $c = 2 \lambda_W \nmax^2 \inftynorm{f}$ and $D = \{(x,y) : |x| \leq cy^2\}$, then $(X^K_t, \angles{\tilzeta^{W|K}_t}{\bfs})$ and $(X_t, \angles{\tileta^{W}_t}{s})$ belong almost surely to $D$, for any $K \geq 1$. As $\angles{\tilzeta^{W|K}_t}{\bfs}$ converges almost surely to $\angles{\tileta^W_t}{\bfs}$, and the application $(x,y) \mapsto (x/y)\setind{y \neq 0}$ is continuous on $D$, we deduce the following convergence in probability:
    \begin{equation*}
        Y^K_t \coloneqq \frac{\angles{\tilzeta^{W|K}_t}{\lambda_W \bfs \bfi}}{\angles{\tilzeta^{W|K}_t}{\bfs}} \angles{\tilzeta^{H|K}_t}{\bfs(f^\I_t - f_t)} \xrightarrow[K \to \infty]{} Y_t \coloneqq \frac{\angles{\tileta^{W}_t}{\lambda_W \bfs \bfi}}{\angles{\tileta^W_t}{\bfs}} \angles{\tileta^{H}_t}{\bfs(f^\I_t - f_t)}.
    \end{equation*}

    In addition, for any $K \geq 1$ and any $t \in [0,T]$, $|Y^K_t| \leq 2 \lambda_W \nmax \inftynorm{f}$. Thus using twice dominated convergence, we first obtain that the above convergence of $(Y^K_t)_{K \geq 1}$ to $Y_t$ also holds in $L^1$, and subsequently the following convergence holds in $L^1$:
    \begin{equation}
    \label{eq:aux-psiC}
        \int_0^T \frac{\angles{\tilzeta^{W|K}_t}{\lambda_W \bfs \bfi}}{\angles{\tilzeta^{W|K}_t}{\bfs}} \angles{\tilzeta^{H|K}_t}{\bfs(f^\I_t - f_t)} dt \xrightarrow[K \to \infty]{} \int_0^T \frac{\angles{\tileta^{W}_t}{\lambda_W \bfs \bfi}}{\angles{\tileta^W_t}{\bfs}} \angles{\tileta^{H}_t}{\bfs(f^\I_t - f_t)} dt.
    \end{equation}

    Reasoning in a similar manner, one also obtains:
    \begin{equation}
    \label{eq:aux-psiD}
        \begin{aligned}
            \int_0^T \angles{\tilzeta^H_t}{\lambda_H \bfs \bfi (f^\I_t - f_t)} dt & \xrightarrow[K \to \infty]{} \int_0^T \angles{\tileta^H_t}{\lambda_H \bfs \bfi (f^\I_t - f_t)} dt \;\; \text{in } L^1, \\
            \int_0^T \frac{\angles{\tilzeta^H_t}{\bfi}}{\angles{\tilzeta^H_0}{\bfn}} \angles{\tilzeta^H_t}{\bfs (f^\I_t - f_t)} dt & \xrightarrow[K \to \infty]{} \int_0^T \frac{\angles{\tileta^H_t}{\bfi}}{\angles{\tileta^H_0}{\bfn}} \angles{\tileta^H_t}{\bfs (f^\I_t - f_t)} dt \;\; \text{in } L^1.
        \end{aligned}
    \end{equation}
    
    Thus, Equations (\ref{eq:aux-psiA}-\ref{eq:aux-psiD}) imply that all the terms on the right hand side of the definition of $\psi^H_T(\tilzeta^K)$ as stated in Equation \eqref{eq:def-psi(zeta)} converge in probability, and thus their linear combination converges in probability to the linear combination of their limits. In other words, $\psi^H_T(\tilzeta^K)$ converges in probability to $\psi^H_T(\tileta)$, which ensures as desired that $\psi^H_T(\zeta^K)$ converges in law to $\psi^H_T(\eta)$. This concludes the proof.
    
\end{proof}

In order to conclude, we only need to show that Proposition \ref{prop:cvgce-stepfct} holds. This proposition establishes convergence of averages of a discontinuous test function $\phi$ with respect to $\widetilde{\zeta}^{X|K}_T$ to its average with respect to $\widetilde{\eta}^X_T$. More precisely, for $x = (n,s,\tau) \in E$, $\phi(x)$ is a linear combination of sufficiently smooth functions on $(n-s)$-dimensional rectangles. We thus start by showing that for any $(n-s)$-dimensional rectangle $B$, the set $(n-s) \times B$ is a continuity set of $\widetilde{\eta}^X_T$ by making use of the absolute continuity of $\nu$. We conclude by approximating $\phi$ using step functions on these continuity sets. 

\begin{proof}[Proof of Proposition \ref{prop:cvgce-stepfct}]
    
    \textbf{Step 1.} Recall that a $d$-dimensional rectangle is a set $A$ defined as the product of $d$ intervals of $\R \cup \{-\infty, +\infty\}$. If all $d$ intervals are included in $\R$, the rectangle will further be said finite.

    Let $X \in \{H,W\}$, $n \in \bbrackets{1}{\nmax}$ and $s \in \bbrackets{0}{n}$. We start by showing that for any $T \geq 0$ and any finite $(n-s)$-dimensional rectangle $B$,
    \begin{equation}
    \label{eq:cvgc-rect}
        \angles{\tilzeta^{X|K}_T}{\ind_{\{(n,s)\}\times B}} \xrightarrow[K \to \infty]{}  \angles{\tileta^{X}_T}{\ind_{\{(n,s)\}\times B}} \;\; \text{almost surely}.
    \end{equation}
    As $\tileta^X \in \C(\R_+, (\M_F(E),w))$, it follows that for any $T$, $\tilzeta^{X|K}_T$ converges almost surely to $\tileta^{X}_T$ in $(\M_F(E),w)$. Thus, in order to establish the desired result, it is sufficient to show that $\bfB = \{(n,s)\}\times B$ is a $\tileta^X_T$-continuity set, in which case the Portmanteau theorem allows to conclude. 

    For any set $A$, let $\partial A$ be the boundary of $A$. Then $\partial \bfB = \{(n,s)\} \times \partial B$. As $B$ is a $(n-s)$-dimensional rectangle, there exist $a_i < b_i \in \R$ for $i \in \bbrackets{1}{n-s}$ such that $B$ can be written as the product  of intervals (potentially open, closed or half-open) delimited by $a_i < b_i$, for $i \in \bbrackets{1}{n-s}$. Thus
    \begin{equation*}
        \partial B = \bigcup_{i=1}^{n-s} \bigcup_{c \in \{a_i,b_i\}} \left( \prod_{j=1}^{i-1} [a_j, b_j] \times \{c\} \times \prod_{k=i+1}^{n-s} [a_k, b_k] \right).
    \end{equation*}

    Consider any $i \in \bbrackets{1}{n-s}$ and $c \in \R$. We are going to prove that 
    \begin{equation}
    \label{eq:rect-bound}
        \angles{\tileta^{X}_T}{\ind_{\{(n,s)\} \times \left( \prod_{j=1}^{i-1} [a_j, b_j] \times \{c\} \times \prod_{k=i+1}^{n-s} [a_k, b_k] \right) }} = 0,
    \end{equation}
    which will be enough to conclude. In order to achieve this, let us introduce a mollifier $\psi \in \C^\infty(\R)$ in the same sense as in the proof of Lemma \ref{lem:density}, with compact support in $[-1,1]$. For $\varepsilon > 0$, define the function $\varphi_\varepsilon: x \mapsto \varepsilon^{-1} \psi(x/\varepsilon)$, whose support lies in $[-\varepsilon, \varepsilon]$ and which converges to $\delta_0$ in the sense of distributions, when $\varepsilon$ goes to zero. For any $x = (n,s,\tau) \in E$, let 
    \begin{equation*}
        \phi_{\varepsilon}(x) = \setind{\bfn(x)=n, \bfs(x)=s} \Big(\prod_{\substack{j=1\\j \neq i}}^{n-s} \ind_{[a_j, b_j]} * \varphi_\varepsilon(\tau_j) \Big) \ind_{{c}} * \varphi_\varepsilon(\tau_i).
    \end{equation*}
    As $\phi_\varepsilon \in \C^1_b(E)$ and $\tilzeta^{X|K}_T$ converges almost surely to $\tileta^X_T$ in $(\M_F(E), w)$, dominated convergence implies that 
    \begin{equation*}
        \E[\angles{\tileta^X_T}{\phi_\varepsilon}] = \lim_{K \to \infty} \E[\angles{\tilzeta^{X|K}_T}{\phi_\varepsilon}].
    \end{equation*}
    Notice that
    \begin{equation*}
        \phi_{\varepsilon}(x) \leq \setind{\bfn(x) - \bfs(x) > i} \ind_{{c}} * \varphi_\varepsilon(\tau_i(x)).
    \end{equation*}
    Hence proceeding as in the proof of Proposition \ref{prop:tightness-weak}, it follows that 
    \begin{equation*}
    \begin{aligned}
        \E[\angles{\tilzeta^{X|K}_T}{\phi_\varepsilon}] &\leq \E\left[ \frac{1}{K_X} \sum_{k=1}^{K_X} \setind{n^X_k - s^X_k(0) \geq i, \; | (\tau^X_{k,i}(0) - T) - c | \leq \epsilon} \right] \\
        &+ \frac{C}{2} \int_0^T \nu([c + (T-t) - \varepsilon, c + (T-t) + \varepsilon]) dt.
    \end{aligned}
    \end{equation*}
    
     Absolute continuity of $\nu$ with regard to the Lebesgue measure and Assumption \ref{hyp:zetaK0} ensure that the right hand side is dominated by a function $c(\varepsilon)$ which does not depend on $K$, and which goes to zero with $\varepsilon$. Thus $\E[\angles{\tileta^X_T}{\phi_\varepsilon}] \leq c(\varepsilon)$. In particular, one may construct a sequence $(\varepsilon_n)_{n \geq 1}$ which converges to $0$ and satisfies $\sum_{n \geq 0} c(\varepsilon_n) < \infty$. Then on the one hand, the Borel-Cantelli lemma ensures that 
     $\angles{\tileta^X_T}{\phi_\varepsilon}$ converges almost surely to $0$ as $n$ tends to infinity. On the other hand, by dominated convergence, $\angles{\tileta^X_T}{\phi_\varepsilon}$ converges almost surely to the left-hand side of Equation \eqref{eq:rect-bound} as $\varepsilon$ tends to zero, hence Equation \eqref{eq:rect-bound} is proven to be true. 

     As a consequence, we conclude that $\angles{\tileta^X_T}{\partial \bfB} = 0$, and thus Equation \eqref{eq:cvgc-rect} holds.

     \textbf{Step 2.} Consider now a function $\phi$ as described in the proposition. For any $N$, let us introduce a partition of $\R^{n-s}$ whose elements consist in (partially open) hypercubes of side length $2^{-N}$. For every $k \leq m(n,s)$, a partition $(B^{N}_{k,j})_{j \geq 1}$ of $A^{n,s}_k$ is obtained by taking the intersection of $A^{n,s}_k$ with those hypercubes. As $A^{n,s}_k$ is a rectangle itself, the family $(B^{N}_{k,j})_{j \geq 1}$ consists of rectangles of side length at most $2^{-N}$. For every $j$, consider a point $z^N_{k,j}$ belonging to $B^{N}_{k,j}$. Finally, define the set $J_N(k) = \{j \geq 0: \sup\{\inftynorm{x} : x \in B^N_{k,j}\} \leq N \}$, which contains only a finite number of elements. Then we can define the following approximation of $\phi$: 
     \begin{equation*}
         \forall x \in E, \;\; \phi_N(x) = \sum_{\substack{ 1 \leq n \leq \nmax\\0 \leq s \leq n}} \setind{\bfn(x)=n,\, \bfs(x) = s} \sum_{k=1}^{m(n,s)} \sum_{j \in J_N(k)} \varphi^{n-s}_k(z^N_{k,j}) \ind_{B^N_{k,j}}(\tau_{1,n-s}(x)).
     \end{equation*}

    Using our result from the first step, for every $(n,s)$ such that $1 \leq n \leq \nmax$ and $0 \leq s \leq n$, for every $k \leq m(n,s)$ and $j \leq J_N(k)$, we obtain that 
    \begin{equation}
    \label{eq:aux-cvgce-ps}
        \lim_{K \to \infty} \angles{\tilzeta^{X|K}_T}{\phi_N} = \angles{\tileta^{X}_T}{\phi_N} \;\; \text{almost surely}.
    \end{equation}

    Notice that for any $x \in E$ such that $\inftynorm{\tau(x)} > N$, $\phi_N(x) = 0$. Hence for any $x = (n,s,\tau)$, one obtains the following inequality: 
    {\small
    \begin{equation}
        | \phi_N(x) - \phi(x)| \leq |\phi(x)| \setind{\inftynorm{\tau} > N} + \sum_{k=1}^{m(n,s)} \!\! \sum_{j \in J_N(k)} |\varphi^{n-s}_k(z^N_{k,j}) - \varphi^{n-s}_k(\tau_{1,n-s} )| \ind_{B^N_{k,j}}(\tau_{1,n-s}).
    \end{equation}
    }
    
    Notice that there exists at most one $(k,j) \in \bbrackets{1}{m(n,s)} \times J_N(k)$ such that $\tau_{1,n-s} \in B^N_{k,j}$. As $\varphi^{n,s}_k \in \C^1_b(\R^{n-s})$, the mean value inequality further implies that 
    \begin{equation*}
        \forall k \leq m(n,s), \forall j \in J_N(k), \forall z \in B^{N}_{k,j}, \quad |\varphi^{n-s}_k(z^N_{k,j}) - \varphi^{n-s}_k(z) | \leq \inftynorm{D\varphi^{n,s}_k} d_N,
    \end{equation*}
    where $d_N$ denotes the maximum of the diameters of $d$-dimensional hypercubes of side length $2^{-N}$, for $d \leq \nmax$. Letting $M = \max_{n,s,k} \inftynorm{D\varphi^{n,s}_k}$, it follows that :
    \begin{equation*}
    \label{eq:aux-dom}
        \forall x \in E, \quad |\phi_N(x) - \phi(x)| \leq \inftynorm{\phi} \setind{\inftynorm{\tau} > N} + M d_N 
    \end{equation*}
    Hence $\phi_N$ converges point-wise to $\phi$. Thus, by dominated convergence, 
    \begin{equation}
    \label{eq:aux-cvgce-tileta}
        \angles{\tileta^X_T}{\phi_N} \xrightarrow[N \to \infty]{} \angles{\tileta^X_T}{\phi}.
    \end{equation}
    
    Furthermore, it follows from Equation \eqref{eq:aux-dom} that for any $K \geq 1$,
    \begin{equation*}
        \E[|\angles{\tilzeta^{X|K}_T}{\phi_N} - \angles{\tilzeta^{X|K}_T}{\phi}|] \leq \inftynorm{\phi} \sup_{K \geq 1} \E[\angles{\zeta^{X|K}_T}{\setind{\inftynorm{\tau(\cdot)} > N}}] + M d_N.
    \end{equation*}
    Reasoning as in the proof of Proposition \ref{prop:tightness-weak}, and using Assumption \ref{hyp:zetaK0}, we obtain that 
    \begin{equation*}
        \lim_{N \to \infty} \sup_{K \geq 1} \E[\angles{\zeta^{X|K}_T}{\setind{\inftynorm{\tau(\cdot)} > N}}] = 0,
    \end{equation*}
    and as a consequence,
    \begin{equation}
    \label{eq:aux-L1cvgce}
        \lim_{N \to \infty} \sup_{K \geq 1} \E[|\angles{\tilzeta^{X|K}_T}{\phi_N} - \angles{\tilzeta^{X|K}_T}{\phi}|] = 0.
    \end{equation}
    Noticing that 
    {\small
    \begin{equation*}
        \E[|\angles{\tilzeta^{X|K}_T}{\phi} - \angles{\tileta^{X}_T}{\phi}|] \leq \E[|\angles{\tilzeta^{X|K}_T}{\phi - \phi_N}|] + \E[|\angles{\tilzeta^{X|K}_T}{\phi_N} - \angles{\tileta^{X}_T}{\phi_N}|] + |\angles{\tileta^{X}_T}{\phi_N - \phi}|
    \end{equation*}
    }
    together with Equations \eqref{eq:aux-L1cvgce}, \eqref{eq:aux-cvgce-ps} and \eqref{eq:aux-cvgce-tileta} finally yields that 
    \begin{equation*}
        \lim_{K \to \infty} \E[|\angles{\tilzeta^{X|K}_T}{\phi} - \angles{\tileta^{X}_T}{\phi}|] = 0.
    \end{equation*}
    This concludes the proof.
\end{proof}

\subsubsection{Proof of Theorem \ref{thm:cvgce}} 

The previous results are sufficient to establish Theorem \ref{thm:cvgce}. Indeed, it follows from Propositions \ref{prop:tightness-weak} and \ref{prop:idlimit} that from every subsequence of $(\zeta^K)_{K \geq 1}$, one may extract a subsubsequence converging in $\D(\R_+, (\M_F(E),w))^2$ to a solution of Equation \eqref{eq:defeta} which is continuous with respect to the total variation norm. As by assumption, $\zeta^K_0$ converges in law to $\eta_0 \in \mathfrak{M}_1$, Proposition \ref{prop:uniqueness} implies that all of these subsubsequences converge to the same limit $\eta$, which is the unique solution of Equation \eqref{eq:defeta} with initial condition $\eta_0$. As $\eta_0 \in \mathfrak{M}_1$, Proposition \ref{prop:uniqueness} further ensures that $\eta \in \D(\R_+, \mathfrak{M}_1)$. This establishes the convergence of $(\zeta^K)_{K \geq 1}$ to $\eta$ in $\D(\R_+, \M_1(E))^2$. 

\subsection{Proof of Theorem \ref{thm:cvgce-sysdyn}}
\label{sec:proof-sysdyn}

This section is devoted to extracting dynamical system \crefrange{eq:subeq-sysdyn-si}{eq:subeq-sysdyn-structures} from the measure-valued integral equation \eqref{eq:defeta}, under the assumption that $\nu$ is the exponential distribution of parameter $\gamma$.

\subsubsection{Preliminary study of the dynamical system}

Before establishing Theorem \ref{thm:cvgce-sysdyn} itself, let us start by showing that dynamical system \crefrange{eq:subeq-sysdyn-si}{eq:subeq-sysdyn-structures} endowed with initial condition \eqref{eq:initialcdts} admits at most a unique solution. Existence will follow from the proofs of the forthcoming subsections, since they construct a solution to the dynamical system.

For this section, let us rewrite dynamical system \crefrange{eq:subeq-sysdyn-si}{eq:subeq-sysdyn-structures} as follows, in order to emphasize the associated Cauchy problem. Recall that the dynamical system is of dimension $d = 2 + 2\#\config = \nmax(\nmax+1)$.

Let $y \in \C^1(\R_+, \R^d)$ and $f: \R^d \to \R^d$ be defined such that dynamical system \crefrange{eq:subeq-sysdyn-si}{eq:subeq-sysdyn-structures} amounts to  
\begin{equation}
    \label{eq:cauchypb}
    y'(t) = f(y(t)) \;\; \forall t \geq 0.    
\end{equation} 
The components of $y$ (and resp. $f$) will be called $s$, $i$ and $n^X_{S,I}$ (resp. $f_s$, $f_i$ and $f_{X,S,I}$) for $X \in \{H,W\}$ and $(S,I) \in \config$, in order to simplify their identification with the unknowns of the corresponding dynamical system. More precisely, consider the applications
\begin{equation*}
        \tau_X(y) = -\frac{\lambda_X}{m_X} \sum_{(S,I) \in \config} SI \; n^X_{S,I} \text{ for } X \in \{H,W\}, 
        \text{ and } \tau_G(y) = \beta_G i.
\end{equation*}
Then $f: \R^d \to \R^d$ is defined as follows, for any $y=(s,i,n^X_{S,I}: X \in \{H,W\}, (S,I) \in \config) \in \R^d$:
\begin{equation*}
        f_s(y) = -(\tau_H(y) + \tau_W(y) + \tau_G(y)s) \; \text{and} \; f_i(y) = -f_s(y) - \gamma i, \\
\end{equation*}
while for all $X \in \{H,W\}$ and $(S,I) \in \config$
\begin{equation*}
    \begin{aligned}
        f_{X,S,I}(y) =\, &-\left[\left(\lambda_X I + \frac{\tau_{\overline{X}}(y)}{s} + \tau_G(y)\right) S  - \gamma I\right] n^X_{S,I}
         + \gamma(I+1)n^X_{S,I+1} \setind{S + I < \nmax} \\
        \,& + \left(\lambda_X(I-1) + \frac{\tau_{\overline{X}}(y)}{s} + \tau_G(y) \right)(S+1)n^X_{S+1,I-1} \setind{I \geq 1}.
    \end{aligned}
\end{equation*}

Also, notice that there are some natural constraints that we expect the solution of dynamical system \crefrange{eq:subeq-sysdyn-si}{eq:subeq-sysdyn-structures} to satisfy. Clearly, $s$, $i$ and $n^X_{S,I}$ should belong to $[0,1]$. Also, as the population is partitioned into susceptible, infected and recovered individuals, it follows that $s+i \leq 1$. Similarly, as all individuals belong to exactly one household and one workplace, and as $n^X_{S,I}$ corresponds to the proportion of structures of type $X$ which contain $S$ susceptible and $I$ infected individuals, we expect that for $X \in \{H,W\}$, 
\begin{equation}
\label{eq:aux-ineq}    
    \sum_{(S,I) \in \config} n^X_{S,I} \leq 1, 
    \text{ and } \sum_{(S,I) \in \config} S n^X_{S,I}  \leq m_X s. 
\end{equation}
We thus define the following set $V \subset \R^d$, which formalizes these constraints:
{\small
\begin{equation*}
    V = \left\{y \in [0,1]^d: s + i \leq 1, \sum_{(S,I) \in \config} n^X_{S,I} \leq 1 \text{ and } m_X s - \sum_{(S,I) \in \config} S n^X_{S,I} \geq 0 \;\; \forall X \in \{H,W\} \right\}.
\end{equation*}
}

\begin{prop}
\label{prop:sysdyn-sols}
    Let $y^* \in V$. Then the following assertions hold:    
    \begin{enumerate}[label= (\roman*)]
        \item Suppose that there exists a solution $y$ of the Cauchy problem \eqref{eq:cauchypb} with initial condition $y(0) = y^*$. Then $y(t) \in V$ for any $t \geq 0$ for which $y$ is well defined. 
        \item For any $T \geq 0$, this problem admits at most a unique solution $y$ on $[0,T]$.
        \item In particular, for any $\varepsilon > 0$, the dynamical system \crefrange{eq:subeq-sysdyn-si}{eq:subeq-sysdyn-structures} endowed with initial condition \eqref{eq:initialcdts} admits at most a unique solution.
    \end{enumerate}
\end{prop}

The proof of this proposition is available in \ref{apx:proofs}. It relies on establishing the Lipschitz continuity of $f$ on $V$, from which uniqueness is deduced using Gronwall's lemma.

\subsubsection{Some properties of the limiting measure \texorpdfstring{$\eta$}{}}

The results of this section focus on the limiting measure $\eta$ given by Theorem \ref{thm:cvgce}, and will be useful for establishing Theorem \ref{thm:cvgce-sysdyn}. 

Let us introduce the following notations. For $f \in \C^1_b(\R_+ \times E, \R)$ and $T \geq t \geq 0$, define 
\begin{equation}
\label{eq:fTt-fI}
    f_{T,t}: x \in E \mapsto f(T, \Psi(x, T,t)) \text{ and } f_{T,t}^\I: x \in E \mapsto \angles{\nu}{f_{T,t}(\mathfrak{j}(x,\cdot))}.
\end{equation}
We further define, for $X \in \{H,W\}$ and $t \geq 0$, the following quantity which relates to the infectious pressure exerted on susceptibles outside of their structure of type $X$:
\begin{equation*}
    \Lambda_X(t) =  \frac{\lambda_{\overline{X}}}{s_{\overline{X}}(t)} \angles{\eta^{\overline{X}}_t}{\bfs \bfi} + \beta_G \frac{i_H(t)}{n_H}.
\end{equation*}
We can now state a result which is similar in spirit to Proposition \ref{prop:zeta(lip)}. Notice that it holds under the same Assumptions as Theorem \ref{thm:cvgce}, and is not restricted to the Markovian case.

\begin{prop}
    \label{prop:eta(l1loc)}
    Let $\eta$ be the unique solution in $\C(\R_+, \mathfrak{M}_1)$ of Equation \eqref{eq:defeta}. Then for any $T \geq 0$ and $t \in [0,T]$, for any measurable bounded function $f : \R_+ \times E \to \R$, it holds for $X \in \{H,W\}$ that 
        \begin{equation}
        \label{eq:eta(l1loc)}
            \angles{\eta^X_t}{f_{T,t}} = \angles{\eta^X_0}{f_{T,0}}  + \lambda_X \int_0^t \angles{\eta^X_u}{\bfs\bfi(f^\I_{T,u} - f_{T,u})}dt + \int_0^t \Lambda_X(u) \angles{\eta^X_u}{\bfs(f^\I_{T,u} - f_{T,u})}dt.
    \end{equation}   
    In particular, the application $t \mapsto \angles{\eta^X_t}{f_{T,t}}$ is continuous.
\end{prop}

\begin{proof}
    Start by noticing that for any $f \in \C^1_b(\R_+ \times E, \R)$, for $X \in \{H,W\}$,
    \begin{equation}
    \label{eq:eta(c1b)}
            \angles{\eta^X_t}{f_t} = \angles{\eta^X_0}{f_{t,0}}  + \lambda_X \int_0^t \angles{\eta^X_u}{\bfs\bfi(f^\I_{t,u} - f_{t,u})}dt + \int_0^t \Lambda_X(u) \angles{\eta^X_u}{\bfs(f^\I_{t,u} - f_{t,u})}du.
    \end{equation}
    Indeed, the proof of Equation \eqref{eq:eta(c1b)} follows the exact same lines as the proof of Proposition \ref{prop:zeta(lip)}, showing that for any $f \in \C^1_b(\R_+ \times E)$, Equation \eqref{eq:eta(c1b)} leads to Equation \eqref{eq:defeta} using Lemma \ref{lem:dphif}. 

    Consider now a measurable bounded function $f : \R_+ \times E \to \R$. Proceeding as in the proof of Lemma \ref{lem:density}, we consider a mollifier $\psi$ on $\R^{1 + \nmax}$ and let $\psi_k(t,\tau) = k^{\nmax + 1} \psi(kt,k\tau)$ for any $(t, \tau) \in \R \times \R^{\nmax}$. Letting $f_k(t,(n,s,\tau)) = f(\cdot, (n,s, \cdot)) * \psi_k(t, \tau)$, we obtain by convolution a sequence of smooth functions $(f_k)_{k \geq 1}$ which converges point-wise to $f$. 
    
    Then on the one hand, for any $t \in [0,T]$ and $x \in E$, as $k$ tends to infinity, $(f_k)_{T,t}(x)$ converges to $f_{T,t}(x)$, and further $(f_k)^\I_{T,t}(x)$ converges to $f^\I_{T,t}(x)$ by dominated convergence. Define $g_k \in \C^1_b(E,\R)$ by $g_k(t,x) = (f_k)_{T,t}(x)$, it then holds that $(g_k)_{t,u} = (f_k)_{T,u}$ and $(g_k)^\I_{t,u} = (f_k)^\I_{T,u}$. Thus applying Equation \eqref{eq:eta(c1b)} to $g_k$ and using dominated convergence as $k$ goes to infinity yields the desired result. 

    Finally, the continuity of $t \mapsto \angles{\eta^X_t}{f_{T,t}}$ on $[0,T]$ is a consequence of Equation \eqref{eq:eta(l1loc)}, as the integrands of the right-hand-side are bounded. 
\end{proof}

Proposition \ref{prop:eta(l1loc)} allows us to establish the following result under the assumption that $\nu$ is the exponential distribution. In particular, it implies that within each structure, at any time, the remaining infectious periods of currently infectious individuals are independent and identically distributed, of common law $\nu$. 

\begin{prop}
\label{prop:dist-infected}
    Assume that $\nu$ is the exponential law of parameter $\gamma$, and consider $\eta$ as defined in Theorem \ref{thm:cvgce} with initial condition $\eta_0 = \eta_{0,\epsilon}$. Let $(X_n)_{n \geq 0}$ be a sequence of independent identically distributed random variables of common law $\nu$. Let $T \geq 0$, $n \in \bbrackets{1}{\nmax}$ and $s \in \bbrackets{0}{n-1}$. For any $m \in \bbrackets{0}{n-s}$, any functions $f \in \mathcal{B}_b(\R_+ \times \R^m)$, $g_1, \dots, g_{n-s-m} \in \mathcal{B}_b(\R_+ \times \R)$ and any $j_1 < \dots < j_m $ and $k_1 <  \dots < k_{n-s-m}$ such that $\{k_1, \dots k_{n-s-m}\} \cup \{j_1, \dots j_m\} = \bbrackets{1}{n-s}$, define
    {\small
    \begin{equation*}
         F(t,x) = \setind{\substack{\bfn(x) = n, \, \bfs(x)=s, \\ \tau_{j_\ell} (x) > 0 \; \forall 1 \leq \ell \leq m}} \big(f(t, \tau_{j_1}(x),\dots, \tau_{j_m}(x)) -  \E[f(t, X_1, \dots, X_m)] \big) \!\! \prod_{\ell=1}^{n-s-m} \!\!\! g_\ell(t, \tau_{k_\ell}(x)).
    \end{equation*}
    }
    Then 
    \begin{equation}
    \label{eq:prop-exp}
           \forall t \in [0,T], \quad \angles{\eta^X_t}{F_{T,t}} = 0.
    \end{equation}
\end{prop}

\begin{proof}
    Let $n \in \bbrackets{1}{\nmax}$. Consider any $m$, $f$, $g$, $\{j_1, \dots, j_m\}$ and $\{k_1, \dots k_m\}$ satisfying the constraints given in the proposition, and define $F$ as above. Throughout the proof, we let $E_m(t) = \E[f(t, X_1, \dots, X_m)]$.

    Notice that $F$ satisfies the assumptions of Proposition \ref{prop:eta(l1loc)}. Letting $C = (\lambda_X \nmax + \lambda_{\overline{X}} \nmax + \beta_G) \nmax$, it follows from Equation \eqref{eq:eta(l1loc)} that:
    \begin{equation}
    \label{eq:gronwallF}
        |\angles{\eta^X_t}{F_{T,t}}| \leq |\angles{\eta^X_0}{F_{T,0}}| + C \int_0^t \left(|\angles{\eta^X_u}{F_{T,u}}| + |\angles{\eta^X_u}{F^\I_{T,u}}|\right)du.
    \end{equation}
    
    Let $(Y_k)_{k \geq 1}$ be a sequence of independent identically distributed random variables of common law $\nu$ which is independent from $(X_k)_{k \geq 1}$.
    %, and write $p^X_{n,s} = \pi^X_n \binom{n}{s} (1 - \varepsilon)^s \varepsilon^{n-s}$. 
    Then on the one hand, Equation \eqref{eq:etastar} ensures:
    \begin{equation*}
    \begin{aligned}
        \angles{\eta^X_0}{F_{T,0}} = \sum_{i=m}^{n-s} \pi^X_n p^X_{s,i|n} \Bigg( & \Big(\E[f(T, Y_{j_1} - T, \dots, Y_{j_m} - T)\setind{Y_{j_1} - T > 0, \dots, Y_{j_m} - T > 0}] \\
        & - \E[f(T, X_1, \dots, X_m)]\E[\setind{Y_{j_1} - T > 0, \dots, Y_{j_m} - T > 0}] \Big) \prod_{\ell=1}^{n-s-m}\E[ g_\ell (T, Y_{k_\ell} - T)] \Bigg).
    \end{aligned}
    \end{equation*}
    Using the tower property of conditional expectations and the memoryless property of the exponential distribution thus lead to $\angles{\eta^X_0}{F_{T,0}} = 0$. 

    On the other hand, let us compute $F^\I_{T,u}(x)$ for any $x \in E$.  Distinguishing the cases ${n-s} \in \mathbb{J}_m = \{j_1, \dots, j_m\}$ and ${n-s} \notin \mathbb{J}_m$ leads to:
    \begin{equation*}
    \begin{aligned}
         F^\I_{T,u}(x) = \setind{\substack{\bfn(x) = n, \, \bfs(x)=s+1, \\ \tau_{j} (x) > T-u \; \forall j \in \mathbb{J}_m \setminus \{n-s\}}} \big(& \setind{n-s \in \mathbb{J}_m} e^{-\gamma (T-u)} a_{T,u}(x) \\ 
         &+ \setind{n-s \notin \mathbb{J}_m} \E[g_{n-s}(X_{n-s} - (T-u))] b_{T,u}(x) \big) 
    \end{aligned}
    \end{equation*}
    where
    \begin{equation*}
    \begin{aligned}
        a(t, x)& = (\E[f(t,\tau_{j_1}(x),\dots, \tau_{j_{m-1}}(x), X_{n-s})] -  E_m(t) \big)
        \prod_{j \in \bbrackets {1}{n-s} \setminus \mathbb{J}_m} g_j(t, \tau_{k_j}(x)),\\
        b(t, x)&= \big(f(t,\tau_{j_1}(x),\dots, \tau_{j_m}(x)) -   E_m(t) \big) 
        \prod_{j \in \bbrackets {1}{n-s-1} \setminus \mathbb{J}_m} g_j(t, \tau_{k_j}(x)).
    \end{aligned}
    \end{equation*}

   We are now ready to proceed by induction on $s \in \bbrackets{0}{n-1}$.
   First, consider the case $s = n-1$. Then either $m = 0$ and the result is immediate, or $m = 1$ and $\mathbb{J}_1 = \{1\}$. Hence necessarily $n-s \in \mathbb{J}_1$ and $a(x) = 0$ as $\E[f(X_0 - (T-t))\setind{X_0 > T-t}] = \E[f(X_1)]$. Thus $F^\I_{T,t}(x) = 0$ for any $x \in E$, and Equation \eqref{eq:gronwallF} reduces to 
    \begin{equation*}
        |\angles{\eta^X_T}{F}| = |\angles{\eta^X_T}{F_{T,T}}| \leq C \int_0^T |\angles{\eta^X_t}{F_{T,t}}| dt.
    \end{equation*}
   Recalling that $t \mapsto \angles{\eta^X_t}{F_{T,t}}$ is continuous on $[0,T]$ according to Proposition \ref{prop:eta(l1loc)}, Gronwall's inequality ensures that $|\angles{\eta^X_T}{F_{T,t}}| = 0$ for every $t \in [0,T]$. 
   
   Suppose now that $s < n - 1$, and that the result holds for $s+1$. It then follows from the induction hypothesis that for any $u \in [0,T]$,
   \begin{equation*}
   \begin{aligned}
       \angles{\eta^X_u}{\setind{\substack{\bfn(x) = n, \, \bfs(x)=s+1, \\ \tau_{j} (x) > T-u \; \forall j \in \mathbb{J}_m \setminus \{n-s\}}} a_{T,u}(\cdot)} =  
       \angles{\eta^X_u}{\setind{\substack{\bfn(x) = n, \, \bfs(x)=s+1, \\ \tau_{j} (x) > T-u \; \forall j \in \mathbb{J}_m \setminus \{n-s\}}} b_{T,u}(\cdot)}  = 0.
    \end{aligned}
   \end{equation*}
   Hence $\angles{\eta^X_t}{F^\I_{T,u}} = 0$ for any $u \in [0,T]$, and Gronwall's inequality allows to conclude as previously.
\end{proof}

\subsubsection{Proof of Theorem \ref{thm:cvgce-sysdyn}}

From now on, we assume that $\eta_0 = \eta_{0,\epsilon}$ as given by \eqref{eq:etastar}, and that $\nu$ is the exponential distribution with parameter $\gamma$.
Throughout this section, let $h$ be the Heaviside step function, \emph{i.e.} $h(z) = \setind{z > 0}$ for any real number $z$.

Let us establish the following proposition, which serves as a starting point of the proof of Theorem \ref{thm:cvgce-sysdyn}.

\begin{prop}
    \label{prop:sXt-iXt}
    Under the assumptions of Theorem \ref{thm:cvgce-sysdyn}, it holds that for any $X \in \{H,W\}$, $(s_X(t))_{t \geq 0}$ and $(i_X(t))_{t \geq 0}$ satisfy
    \begin{equation}
        \label{eq:sXt-iXt}
        \begin{aligned}
            \frac{d}{dt} s_X(t) &= - n_X\left( \frac{\lambda_X}{n_X} \angles{\eta^X_t}{\bfs\bfi} + \frac{\lambda_{\overline{X}}}{n_{\overline{X}}} \angles{\eta^{\overline{X}}_t}{\bfs \bfi} + \beta_G \frac{i_H(t)}{n_H} \frac{s_X(t)}{n_X} \right), \\
            \frac{d}{dt} i_X(t) &= - \frac{d}{dt} s_X(t) + \gamma i_X(t).
        \end{aligned}
    \end{equation}
    Further,
    \begin{equation}
        \label{eq:sX0-iX0}
        s_X(0) = (1 - \varepsilon)n_X\;\; \text{and} \;\; i_X(0) =\varepsilon n_X.
    \end{equation}
\end{prop}

\begin{proof}
    Notice that $\bfs \in \C^1_b(E,\R)$ is such that $\A \bfs = 0$ and $\bfs^\I(x) - \bfs(x) = -1$ for all $x \in E$. It thus follows immediately from Equation \eqref{eq:defeta} that, for any $X \in \{H,W\}$,
    \begin{equation*}
        s_X(T) = s_X(0) - \int_0^T \left(\lambda_X \angles{\eta^X_t}{\bfs\bfi} + \lambda_{\overline{X}} \frac{s_X(t)}{s_{\overline{X}}(t)} \angles{\eta^{\overline{X}}_t}{\bfs \bfi} + \beta_G \frac{i_H(t)}{n_H} s_X(t) \right)dt.
    \end{equation*}
    Further, since $\eta \in \C(\R_+, (\mathfrak{M}_1, \tvnorm))$, it follows that for any $t \geq 0$, $(\zeta^K_t)_{K \geq 1}$ converges in law to $\eta_t$. As $\bfs$ is continuous and bounded on $E$, this implies that $\angles{\zeta^{X|K}_t}{\bfs} \to \angles{\eta^X_t}{\bfs}$ when $K$ tends to infinity. The analogous result holds for $\bfn$. Hence Lemma \ref{lemme:shsw} ensures that, for any $t \geq 0$, $s_X(t)/s_{\overline{X}}(t) = n_X/n_{\overline{X}}$. In other words,
    \begin{equation*}
        s_X(T) = s_X(0) - n_X\int_0^T \left( \frac{\lambda_X}{n_X} \angles{\eta^X_t}{\bfs\bfi} + \frac{\lambda_{\overline{X}}}{n_{\overline{X}}} \angles{\eta^{\overline{X}}_t}{\bfs \bfi} + \beta_G \frac{i_H(t)}{n_H} \frac{s_X(t)}{n_X}  \right) dt.
    \end{equation*}
    As $\eta \in \C(\R_+, (\mathfrak{M}_1, \tvnorm))$, and $\bfs$ and $\bfi$ are bounded measurable functions, it follows that the integrand is continuous with regard to $t$. Thus, the first line of Equation \eqref{eq:sXt-iXt} comes from the fundamental theorem of calculus.

    Recall that $\eta_0 = \eta_{0,\epsilon}$ as provided by Equation \eqref{eq:etastar}. In particular, by definition of $\eta_{0,\epsilon}$, we have $s_X(0) = n_X \varepsilon_s = m_X \varepsilon_s$ and $i_X(0) = n_X \varepsilon_s = m_X \varepsilon_s$. 
%    In particular, we now have $n_X= m_X$, hence
%   \begin{equation}
%   \label{eq:sX0}
%       s_X(0) = \angles{\eta^X_{0,\epsilon}}{\bfs} = \sum_{n=1}^{\nmax} \pi^X_n \sum_{s=0}^{n}s \binom{n}{s} (1-\varepsilon)^s \varepsilon^{n-s} = (1 - \varepsilon) \sum_{n=1}^{\nmax} n \pi^X_n = (1 - \varepsilon) n^X.
%   \end{equation}
   This yields %the first part of 
   Equation \eqref{eq:sX0-iX0}.
    
    It remains to take an interest in $i_X(t)$. As $\bfi$ does not belong to $\C^1_b(E, \R)$ we cannot proceed in the same way. Remember that $\bfi(x) = \sum_{j=1}^{\nmax} h(\tau_j)$ for $x = (n,s,\tau) \in E$. Thus, we may apply Proposition \ref{prop:eta(l1loc)} and obtain that 
    \begin{equation}
    \label{eq:aux8}
        \begin{aligned}
            \angles{\eta^X_T}{\bfi} & = \angles{\eta^X_0}{\bfi_{T,0}} + \lambda_X \int_0^T \angles{\eta^X_t}{\bfs\bfi(\bfi^\I_{T,t} - \bfi_{T,t})}dt \\
            & + \lambda_{\overline{X}} \int_0^T \frac{1}{s_{\overline{X}}(t)} \angles{\eta^{\overline{X}}_t}{\bfs\bfi} \angles{\eta^X_t}{\bfs(\bfi^\I_{T,t} - \bfi_{T,t})}dt + \beta_G \int_0^T \frac{i_H(t)}{n_H} \angles{\eta^X_t}{\bfs(\bfi^\I_{T,t} - \bfi_{T,t})}dt.
        \end{aligned}
    \end{equation}
   Further, for any $\sigma \geq 0$, $T \geq t \geq 0$ and $x \in E$,
    \begin{equation*}
        i(\Psi(\mathfrak{j}(x,\sigma),T,t)) - i(\Psi(x,T,t)) = \setind{\sigma > (T-t)},
    \end{equation*}
    hence
    \begin{equation*}
        \bfi^\I_{T,t}(x) - \bfi_{T,t}(x) = \nu([T-t,\infty)) = \e^{-\gamma(T-t)}.
    \end{equation*}
    Injecting this into Equation \eqref{eq:aux8} and using as before that $s_X(t)/s_{\overline{X}}(t) = n_X/n_{\overline{X}}$ yields
    \begin{equation}
    \label{eq:aux8bis}
        i_X(T) = \angles{\eta^X_0}{\bfi_{T,0}} + n_X\e^{-\gamma T}\int_0^T \e^{\gamma t}\left( \frac{\lambda_X}{n_X} \angles{\eta^X_t}{\bfs\bfi} + \frac{\lambda_{\overline{X}}}{n_{\overline{X}}} \angles{\eta^{\overline{X}}_t}{\bfs \bfi} + \beta_G \frac{i_H(t)}{n_H} \frac{s_X(t)}{n_X}  \right) dt.
    \end{equation}
    As $\eta_0 = \eta_{0,\epsilon}$, we may compute the first term of the right-hand side of this equation and obtain that  
   \begin{equation}
   \label{eq:iX0}
           \angles{\eta^X_0}{\bfi_{T,0}} = \angles{\eta^X_{0,\epsilon}}{\bfi_{T,0}}
           = \e^{-\gamma T} 
		\sum_{n=1}^{\nmax} \sum_{\substack{0 \leq s \leq n \\ 0 \leq i \leq n - s}} i \, \pi^X_n p^X_{s,i | n} = n_X \varepsilon_i
           % \sum_{n=1}^{\nmax} \pi^X_n \sum_{s=0}^n \binom{n}{s} (1 - \varepsilon)^s \varepsilon^{n-s} (n-s) 
           =  \e^{-\gamma T} \varepsilon_i n_X.
   \end{equation}
   Using the continuity of the integrand in Equation \eqref{eq:aux8bis}, we may now differentiate it with regard to $T$: 
   {\small
    \begin{equation*}
           \frac{d}{dT} i_X(T) = n_X \!\! \left( \!\frac{\lambda_X}{n_X} \angles{\eta^X_T}{\bfs\bfi} + \frac{\lambda_{\overline{X}}}{n_{\overline{X}}} \angles{\eta^{\overline{X}}_T}{\bfs\bfi} + \beta_G \frac{i^H_T}{n_H} \frac{s_X(T)}{n_X} \!\! \right)\!\! - \gamma i_X(T)
            = -\frac{d}{dT} s_X(T) - \gamma i_X(T).
   \end{equation*}
   }
   We thus have recovered the second line of Equation \eqref{eq:sXt-iXt}. 
   %Finally, the second half of Equation \eqref{eq:sX0-iX0} is obtained by a computation analogous to \eqref{eq:sX0}. 
   This concludes the proof. 
\end{proof}

 For $(S,I) \in \config$, define the function $f^{S,I}: E \to \{0,1\}$ by 
    \begin{equation*}
        f^{S,I}(x) = \setind{\bfs(x) = S,\, \bfi(x) = I}.
    \end{equation*}
   For $t \geq 0$, let $n^X_{S,I}(t) = \angles{\eta^X_t}{f^{S,I}}$, which defines a continuous function on $\R_+$ as $f^{S,I}$ is bounded and measurable and $\eta \in \C(\R_+, (\mathfrak{M}_1, \tvnorm))$. In words, this corresponds to the proportion of structures of type $X$ which contain exactly $S$ susceptible and $I$ infected individuals. Notice that
    \begin{equation*}
        \{x \in E:   \bfs(x)\bfi(x) > 0\} = \{x \in E:   \bfs(x)\bfi(x) > 0, \; (\bfs(x),\bfi(x)) \in \config\}.  
    \end{equation*}
    We may thus rewrite the first line of Equation \eqref{eq:sXt-iXt} as follows:
    {\small
     \begin{equation*}
            \frac{d}{dt} s_X(t) =  - n_X\left( \frac{\lambda_X}{n_X} \sum_{(S,I) \in \config} SI n^X_{S,I}(t) + \frac{\lambda_{\overline{X}}}{n_{\overline{X}}} \sum_{(S,I) \in \config} SI n^{\overline{X}}_{(S,I)}(t) + \beta_G \frac{i_H(t)}{n_H} \frac{s_X(t)}{n_X}  \right).
    \end{equation*}
    }
    Similarly, it holds that
    \begin{equation*}
        \Lambda_X(t) =  \frac{\lambda_{\overline{X}}}{s_{\overline{X}}(t)} \angles{\eta^{\overline{X}}_t}{\bfs \bfi} + \beta_G \frac{i_H(t)}{n_H} =  \frac{\lambda_{\overline{X}}}{s_{\overline{X}}(t)} \sum_{(S,I) \in \config} SI n^{\overline{X}}_{(S,I)}(t) + \beta_G \frac{i_H(t)}{n_H},
    \end{equation*}
    which may also be written in terms of the notations of Equation \crefrange{eq:subeq-sysdyn-si}{eq:subeq-sysdyn-structures}:
    \begin{equation*}
        \Lambda_X(t) =  \left(\frac{s_{\overline{X}}(t)}{n_X}\right)^{-1} \tau_{\overline{X}}(t) + \beta_G \frac{i_H(t)}{n_H}.
    \end{equation*}

This motivates a closer study of the functions $n^X_{S,I}$ for $(S,I) \in \config$.

\begin{prop}
    \label{prop:nXt}
    Let $X \in \{H,W\}$ and $(S,I) \in \config$. Under the assumptions of Theorem \ref{thm:cvgce-sysdyn}, it holds that 
    \begin{equation}
   \label{eq:aux-sysdyn-nXSI}
    \begin{aligned}
       \frac{d}{dt} n^X_{S,I}(t) &= \gamma \left((I+1) n^X_{S,I+1}(t)\setind{S + I < \nmax} - I n^X_{S,I}(t) \right) \\
       &+ \lambda_X \left((S+1)(I-1)n^X_{S+1,I-1}(t)\setind{I \geq 1} - SI n^X_{S,I}(t)\right) \\
       &+ \Lambda_X(t) \left((S+1)n^X_{S+1,I-1}(t)\setind{I \geq 1} - S n^X_{S,I}(t)\right).
    \end{aligned}   
   \end{equation}
   Further
    \begin{equation}
    \label{eq:nX0}
    %n^X_{S,I}(0) = \binom{S+I}{I} \pi^X_{S+I} (1 - \varepsilon)^S \varepsilon^I.
    n^X_{S,I}(0) = \sum_{n=S + I}^{\nmax} \pi^X_n p^{X}_{S,I | n}
    \end{equation}
\end{prop}

\begin{proof}
    Let us start by establishing the initial condition of Equation \eqref{eq:nX0}. Let $(S,I) \in \config$. We will make use of another expression of $f^{S,I}$, which will actually be of use throughout the proof.  For two integers $j \leq n$, let $\B(n,j)$ be the set of unordered subsets of $j$ elements chosen in $\bbrackets{1}{n}$. It then holds that for any $x \in E$,
    \begin{equation*}
        f^{S,I}(x) = \sum_{n=S+I}^{\nmax} \setind{\substack{\bfn(x)=n\\\bfs(x)=S}} \sum_{\bfv \in \B(n-S,I)} H^{n,S,\bfv}(\tau(x)),
    \end{equation*} 
    where, for $\bfv \in \B(n-S,I)$ and any $\tau \in \R^{\nmax}$,
    \begin{equation*}
        H^{n,S,\bfv}(\tau) = \prod_{j \in \bfv} h(\tau_j) \prod_{j \in \bbrackets{1}{n-S} \setminus \bfv} (1 - h(\tau_j)).
    \end{equation*}
    The idea behind this expression is that a structure of type $(n,s,\tau)$ contains $S$ susceptible and $I$ infected members if and only if $s = S$, and further exactly $I$ out of the $n-S$ first components of $\tau$ are positive. In other words, there exists at most one element $\bfv \in \B(n-S,I)$ for which the term in the sum is not equal to zero, in which case the set $\bfv$ corresponds to the indexes of infectious members, while $\bbrackets{1}{n-S} \setminus \bfv$ is the set of recovered members.

    %Let $\varepsilon > 0$. 
    %Throughout the following, for $X \in \{H,W\}$, $n \in \bbrackets{1}{\nmax}$ and $s \in \bbrackets{0}{n}$, let $p^X_{n,s} = \pi^X_n \binom{n}{s}(1 - \varepsilon)^s \varepsilon^{n-s}$. 
    Let $(Y_\ell)_{\ell \geq 1}$ be a sequence of independent identically distributed random variables of common law $\nu$. By definition of $\eta_{0,\epsilon}$, it holds that
    \begin{equation*}
            n^X_{S,I}(0) %&= \sum_{n=S+I}^{\nmax} p^X_{n,S} \sum_{\bfv \in \B(n-S,I)} \int_{\R^{n-S}} H^{n,S,\bfv}(\tau_1, \dots, \tau_{n-S}) \nu(d\tau_1)\dots\nu(d\tau_{n-S}) \\
            %=  \sum_{n=S+I}^{\nmax} p^X_{n,S} \sum_{\bfv \in \B(n-S,I)} \angles{\nu}{h}^{\#\bfv} \angles{\nu}{1 - h}^{(n-S)-\#\bfv}.
            = \sum_{n = S+I}^{\nmax} \pi^X_n \sum_{0 \leq i \leq n-S} p^X_{S,i | n} \sum_{\bfv \in \B(n-S,I)} H^{n, S, \bfv}(Y_1, \dots, Y_i, 0, \dots 0).
    \end{equation*}
    %Whenever $(n-S)-\#\bfv > 0$, the term vanishes as $\angles{\nu}{1 - h} = 0$. Hence only the case $\bfv=\bbrackets{1}{n-S}$ remains, which in turn corresponds to $n=S+I$ and leads to Equation \eqref{eq:nX0}.
	In particular, $H^{n, S, \bfv}(Y_1, \dots, Y_i, 0, \dots 0) = 0$ except for $\bfv = \bbrackets{1}{I}$, in which case it equals one. This leads to Equation \eqref{eq:nX0}.
	
    Next, let us apply Proposition \ref{prop:eta(l1loc)} to $f^{S,I}$. A brief computation, based on distinguishing the cases where $n-s \in \bfv$ or $n-s \notin \bfv$ for any $\bfv \in \B(n-S, I)$, yields
    \begin{equation*}
        \left(f^{S,I}\right)^\I_{T,t} = \setind{I \geq 1} \e^{-\gamma(T-t)}f^{S+1,I-1}_{T,t} + \setind{S+1 < \nmax}(1 - \e^{-\gamma(T-t)})f^{S+1,I}_{T,t}.
    \end{equation*}
    As a consequence, Proposition \ref{prop:eta(l1loc)} ensures that
    \begin{equation}
    \label{eq:nxsiT}
    \begin{aligned}
        n^X_{S,I}(T) &= \angles{\eta^X_0}{f^{S,I}_{T,0}} - \int_0^T \left(\lambda_X \angles{\eta^X_T}{\bfs \bfi f^{S,I}_{T,t}} + \Lambda_X(t) \angles{\eta^X_t}{\bfs f^{S,I}_{T,t}} \right) dt \\
        & + \setind{I \geq 1} \int_0^T \e^{-\gamma(T-t)} \left(\lambda_X \angles{\eta^X_T}{\bfs \bfi f^{S+1,I-1}_{T,t}} + \Lambda_X(t) \angles{\eta^X_t}{\bfs f^{S+1,I-1}_{T,t}} \right) dt \\
        & + \setind{S+I < \nmax} \int_0^T (1 -\e^{-\gamma(T-t)}) \left(\lambda_X \angles{\eta^X_T}{\bfs \bfi f^{S+1,I}_{T,t}} + \Lambda_X(t) \angles{\eta^X_t}{\bfs f^{S+1,I}_{T,t}} \right) dt.
    \end{aligned}
    \end{equation}
    We will thus focus on differentiating with respect to $T$ the different expressions composing the right-hand-side. 
    
    Let us start with the term $\angles{\eta^X_0}{f^{S,I}_{T,0}}$. By definition of $f^{S,I}_{T,0}$ and $\eta^X_{0, \epsilon}$, it holds that
    \begin{equation*}
    	%\angles{\eta^X_0}{f^{S,I}_{T,0}} = \sum_{n=S+I}^{\nmax} p^X_{n,S} \binom{I}{n-S} \e^{-\gamma IT} (1 - \e^{-\gamma T})^{n-S-I}.
	\angles{\eta^X_0}{f^{S,I}_{T,0}} = \sum_{n=S+I}^{\nmax} \pi^X_n \sum_{i = I}^{n-S} p^{X}_{S,i | n} \binom{i}{I} \e^{-\gamma IT} (1 - \e^{-\gamma T})^{i - I}.
    \end{equation*}
    
    Hence, using the Equality $(i-I)\binom{i}{I} = (I+1) \binom{i}{I+1}$, we obtain that 
    {\small
    \begin{equation*}
    	\frac{d}{dT} \angles{\eta^X_0}{f^{S,I}_{T,0}} = -\gamma I \angles{\eta^X_0}{f^{S,I}_{T,0}} + \gamma (I+1) \sum_{n=S+I+1}^{\nmax} \pi^X_{n} \sum_{i = I+1}^{n-S} p^{X}_{S,i | n} \binom{i}{I+1} \e^{-\gamma(I+1)T}(1 - \e^{-\gamma T})^{i-I-1}
    \end{equation*}
    }
    and we thus recognize that 
    \begin{equation}
    \label{eq:dtinit}
    	\frac{d}{dT} \angles{\eta^X_0}{f^{S,I}_{T,0}} = -\gamma I \angles{\eta^X_0}{f^{S,I}_{T,0}} + \gamma (I+1) \angles{\eta^X_0}{f^{S,I+1}_{T,0}}.
    \end{equation}
    
    Let us now focus on the remaining terms of the right-hand side of Equation \eqref{eq:nxsiT}. This motivates a closer study of $f^{S,I}_{T,t}$ for any $(S,I) \in \config$. By definition, for $x \in E$,
    \begin{equation*}
        f^{S,I}_{T,t}(x) = \sum_{n = S+I}^{\nmax} \setind{\substack{\bfn(x)=n\\\bfs(x)=S}} \! \sum_{\bfv \in \B(n-S,I)} \prod_{j \in \bfv} h(\tau_j(x) - (T-t)) \!\!\!\!\!\!\! \prod_{j \in \bbrackets{1}{n-S} \setminus \bfv} \!\!\!\!\!\!\! (1 - h(\tau_j(x) - (T-t))).
    \end{equation*} 
    In particular, $f^{S,I}_{T,t}(x) > 0$ requires that at time $t$, in a structure of type $x$, there are $J \geq I$ infected individuals, out of which exactly $I$ must have a remaining infectious period exceeding $(T-t)$. Hence, consider $(S,I)$ and $T$ to be fixed and define, for any $J \in \bbrackets{I}{\nmax - S}$ and $x \in E$,
    \begin{equation}
    \label{eq:auxgj}
        g_J(t,x) = \sum_{n = S+ J}^{\nmax} \setind{\substack{\bfn(x)=n\\\bfs(x)=S}} \sum_{\bfv_0 \in \B(n-S,J)} \sum_{\substack{\bfv \subseteq \bfv_0\\\#\bfv = I}} G^{n,\bfv_0, \bfv}(t, \tau(x)),
    \end{equation}
    where, for $t \geq 0$ and $\tau \in \R^{\nmax}$,
    \begin{equation*}
        G^{n,\bfv_0, \bfv}(t, \tau) = \prod_{j \in \bfv}h(\tau_j - (T-t)) \prod_{j \in \bfv_0 \setminus \bfv} \setind{0 < \tau_j \leq T-t} \prod_{j \in \bbrackets{1}{n-S}\setminus \bfv_0} (1 - h(\tau_j)).
    \end{equation*}
    Dependence of $g_J$ and $G^{n,\bfv_0, \bfv}$ on $(S,I)$ and $T$ is omitted in these notations for readability. It then holds that 
    \begin{equation*}
        \forall x \in E, \quad f^{S,I}_{T,t}(x) = \sum_{J=I}^{\nmax - S} g_J(t,x).
    \end{equation*}
    Let $J \in \bbrackets{I}{\nmax - S}$, $\bfv_0 = \{j_1, \dots j_J\} \in \B(n-s, J)$ and $\bfv \subseteq \bfv_0$ such that $\# \bfv = I$. Consider Equation \eqref{eq:auxgj}, and apply Proposition \ref{prop:dist-infected} to the functions $f : \R \times \R^J \to \R$ and $g_k : \R \times \R \to \R$ defined by 
    \begin{equation*}
    \begin{aligned}
        f(t, \tau_{j_1}, \dots, \tau_{j_J}) &= \prod_{j \in \bfv} h(\tau_j - (T-t)) \prod_{j \in \bfv_0 \setminus \bfv} \setind{0 \leq \tau_j \leq T-t}, \\
        g_k(t, \tau) &= 1 - h(\tau) \quad \forall k \in \bbrackets{1}{n-S-J}.
    \end{aligned}
    \end{equation*}
    This leads to the following equality. For all $t \in [0,T]$:
    \begin{equation*}
    \begin{aligned}
        \angles{\eta^X_t}{g_J(t,\cdot)} &= \!\! \sum_{n=S+J}^{\nmax} \sum_{\substack{\bfv_0 \in \B(n-S,J) \\ \bfv \subseteq \bfv_0 : \# \bfv = I}} 
         \!\!\! \e^{-\gamma(T-t)I}(1-\e^{-\gamma(T-t)})^{J-I} \angles{\eta^X_t}{\setind{\substack{\bfn(\cdot) = n\\\bfs(\cdot) = S}} H^{n,S,\bfv_0}(\tau(\cdot))} \\
        &= \binom{J}{I} \e^{-\gamma(T-t)I}(1-\e^{-\gamma(T-t)})^{J-I} n^X_{S,J}(t).
    \end{aligned}
    \end{equation*}
   As a consequence, we obtain in particular that
    \begin{equation*}
    	\angles{\eta^X_t}{\bfs \bfi f^{S,I}_{T,t}} = \sum_{J = I}^{\nmax - S} \binom{J}{I} \e^{-\gamma(T-t)I}(1-\e^{-\gamma(T-t)})^{J-I} SJ n^X_{S,J}(t).
    \end{equation*}
    Differentiating with respect to $T$ yields 
    \begin{equation*}
        \partial_T \angles{\eta^X_t}{\bfs \bfi f^{S,I}_{T,t}} = -\gamma I \angles{\eta^X_t} {\bfs \bfi f^{S,I}_{T,t}} + \setind{S+I < \nmax} \gamma (I+1) \angles{\eta^X_t} {\bfs \bfi f^{S,I+1}_{T,t}}.
    \end{equation*}
    In particular, $(T,t) \mapsto \partial_T \angles{\eta^X_t}{\bfs \bfi f^{S,I}_{T,t}}$ is continuous, as for any $(S,J) \in \config$,  $n^X_{S,J}$ is continuous thanks to the continuity of $t \mapsto \eta^X_t$ with regard to the total variation norm. Let $g : \R \to \R_+$ be  such that $t \mapsto g(t) \angles{\eta^X_t}{\bfs \bfi f^{S,I}_{T,t}}$ is continuous on $\R_+$ for any $(S,I) \in \config$. It then holds that
    \begin{equation}
    \label{eq:dtsif}
    	\begin{aligned}
    		\frac{d}{dT}\int_0^T g(t) \angles{\eta^X_t}{\bfs \bfi f^{S,I}_{T,t}} dt &= g(T) SI n^X_{S,I}(T) - \gamma I \int_0^T g(t) \angles{\eta^X_t}{\bfs \bfi f^{S,I}_{T,t}} dt \\
		& + \setind{S+I < \nmax} \gamma(I+1) \int_0^T g(t) \angles{\eta^X_t}{\bfs \bfi f^{S,I+1}_{T,t}} dt.
	\end{aligned}
    \end{equation}
    Similarly, for $g : \R \to \R_+$ such that $t \mapsto g(t) \angles{\eta^X_t}{\bfs f^{S,I}_{T,t}}$ is continuous on $\R_+$ for any $(S,I) \in \config$,
     \begin{equation}
     \label{eq:dtsf}
     	\begin{aligned}
    		\frac{d}{dT}\int_0^T g(t) \angles{\eta^X_t}{\bfs f^{S,I}_{T,t}} dt &= g(T) S n^X_{S,I}(T) - \gamma I \int_0^T g(t) \angles{\eta^X_t}{\bfs f^{S,I}_{T,t}} dt \\
		& + \setind{S+I < \nmax} \gamma(I+1) \int_0^T g(t) \angles{\eta^X_t}{\bfs f^{S,I+1}_{T,t}} dt.
	\end{aligned}
    \end{equation}
    
    In particular, we may apply Equations \eqref{eq:dtsif} and \eqref{eq:dtsf} to $g(t) = 1$, as well as $g(t) = \Lambda_X(t)$ and $g(t) = \e^{-\gamma t} \Lambda_X(t)$. Indeed, $t \mapsto \Lambda_X(t) S n^X_{S,I}(t)$ is continuous and well defined for any $t \geq 0$, thanks to the inequality $\angles{\eta^{\overline{X}}_t}{\bfs \bfi} S n^X_{S,I}(t) \leq (\nmax n_X/n_{\overline{X}}) s_{\overline{X}}(t)^2$ which ensures that despite the division by $s_{\overline{X}}$ in the definition of $\Lambda_X$, there are no singularities.
    
    In conclusion, Equation \eqref{eq:dtinit}, together with Equations \eqref{eq:dtsif} and \eqref{eq:dtsf}, allows to differentiate the right-hand-side of Equation \eqref{eq:nxsiT}. In particular, regrouping the terms factorised by $-\gamma I$ and $\gamma(I+1)$ allows to distinguish $-\gamma I n^X_{S,I}(T)$ and $\gamma(I+1) n^X_{S,I+1}(T)$, using Equation \eqref{eq:nxsiT}. This computation finally leads to Equation \eqref{eq:aux-sysdyn-nXSI}.  
\end{proof}

We may finally focus on the main result of this section, namely Theorem \ref{thm:cvgce-sysdyn}.

\begin{proof}[Proof of Theorem \ref{thm:cvgce-sysdyn}.]

    Before concluding, we need to emphasize that it would have been possible to chose $X = W$ when replacing $S$ and $I$ by $K_X S_X$ and $K_X I_X$ for $\mathcal{I}_G$ in Proposition \ref{prop:crochet}. All of the subsequent results still hold, simply replacing the household-related quantities in the definition of the rate for mean-field infections by their workplace-related counterparts. 
    
    As a consequence, Propositions \ref{prop:sXt-iXt} and \ref{prop:nXt} show that for any $X \in \{H,W\}$, 
    \begin{equation*}
    y_X = \left( \frac{s_X}{m_X}, \frac{i_X}{m_X}, n^X_{S,I}: (S,I) \in \config, n^{\overline{X}}_{(S,I)}: (S,I) \in \config \right)
    \end{equation*} 
    satisfies the Cauchy problem \eqref{eq:cauchypb} with initial condition \eqref{eq:initialcdts}. However, Proposition \ref{prop:sysdyn-sols} ensures uniqueness of the solutions to this Cauchy problem. It hence is sensible to define, for $t \geq 0$,
    \begin{equation*}
        s(t) = \frac{s_H(t)}{m_H} = \frac{s_W(t)}{m_W} \quad \text{and} \quad i(t) = \frac{i_H(t)}{m_H} = \frac{i_W(t)}{m_W}.
    \end{equation*}
    This leads to dynamical system \crefrange{eq:subeq-sysdyn-si}{eq:subeq-sysdyn-structures} with initial conditions \eqref{eq:initialcdts}, and concludes the proof. 
\end{proof}

%%%%%%%%%%%%%%%%%%%%% SECTION %%%%%%%%%%%%%%%%%%%%%
\section*{Discussion}
\label{sec:discussion}
\addcontentsline{toc}{section}{\nameref{sec:discussion}}

This paper has focused on proposing a new reduction for an \emph{SIR} model with two levels of mixing, which explicitly includes households and workplaces. This reduced model was obtained as its large population limit, and the associated convergence of the stochastic model was established. 
\smallskip

Let us discuss several possible model extensions. First, more general contact network models may be of interest. For instance, one may consider a local level of mixing containing an arbitrary, yet finite, number of layers. As long as within each layer, each node is part of exactly one clique, and as long as cliques within each layer are constituted independently from one another as in the case for households and workplaces, the adaptation of the aforementioned results is expected to be straightforward. In particular, for exponentially distributed infectious period lengths, the dimension of the corresponding dynamical system should still be of order $O(\nmax^2)$, implying that the model should remain tractable. In addition, we have assumed here that clique sizes are bounded given their interpretation as small contact structures such as households or workplaces, which is motivated by applications arising in epidemiology and public health. However, we believe that an extension to cliques of unbounded size is achievable, at the cost of technical detail. 

Similarly, one may seek to relax the assumption that contacts in the general population are uniformly mixing, and consider for example a configuration model instead.  In this case, it seems sensible to assume that EBCM-like equations will appear, as they arise in the large population limit of the epidemic spreading on a configuration model \cite{decreusefondLargeGraphLimit2012, millerNotePaperErik2011, volzSIRDynamicsRandom2008}.

% Furthermore, we have compared the reduced model obtained in this work with the corresponding EBCM in the line of \cite{volzEffectsHeterogeneousClustered2011}. In the case of our household-workplace model with two levels of mixing, the EBCM seems the less appropriate choice, as it is of significantly higher dimension $O(\nmax^3)$ and only approximates the macroscopic epidemic well if the initial proportion of infected is very small. However, this may change if a more general contact structure within layers is considered, such as a configuration model for the global level. In this case, it seems sensible to assume that EBCM-like equations will appear, as they arise in the large population limit of the epidemic spreading on a configuration model \cite{decreusefondLargeGraphLimit2012, millerNotePaperErik2011, volzSIRDynamicsRandom2008}.\smallskip

Another type of model extension consists in modifying the way the epidemic spreads on top of the contact network. For instance, it would be possible to let infection rates depend deterministically on the infected individual's remaining infectious period length, or more classically age of infection \cite{kermackContributionMathematicalTheory1927, brauerAgeInfectionEpidemic2016}, which can be included analogously to the former. In addition, one may wish to allow for reinfections, as in \emph{SIS} or \emph{SIRS} models. In that case, it appears necessary to recall some information on the already discovered contact network. This is due to the fact that the infection of a previously recovered individuals leads back to a household and workplace which the epidemic has already visited, whose composition is expected to be different from a not yet infected contact structure. 
\smallskip

\begin{figure}
    \centering
    \includegraphics[width=0.8\textwidth]{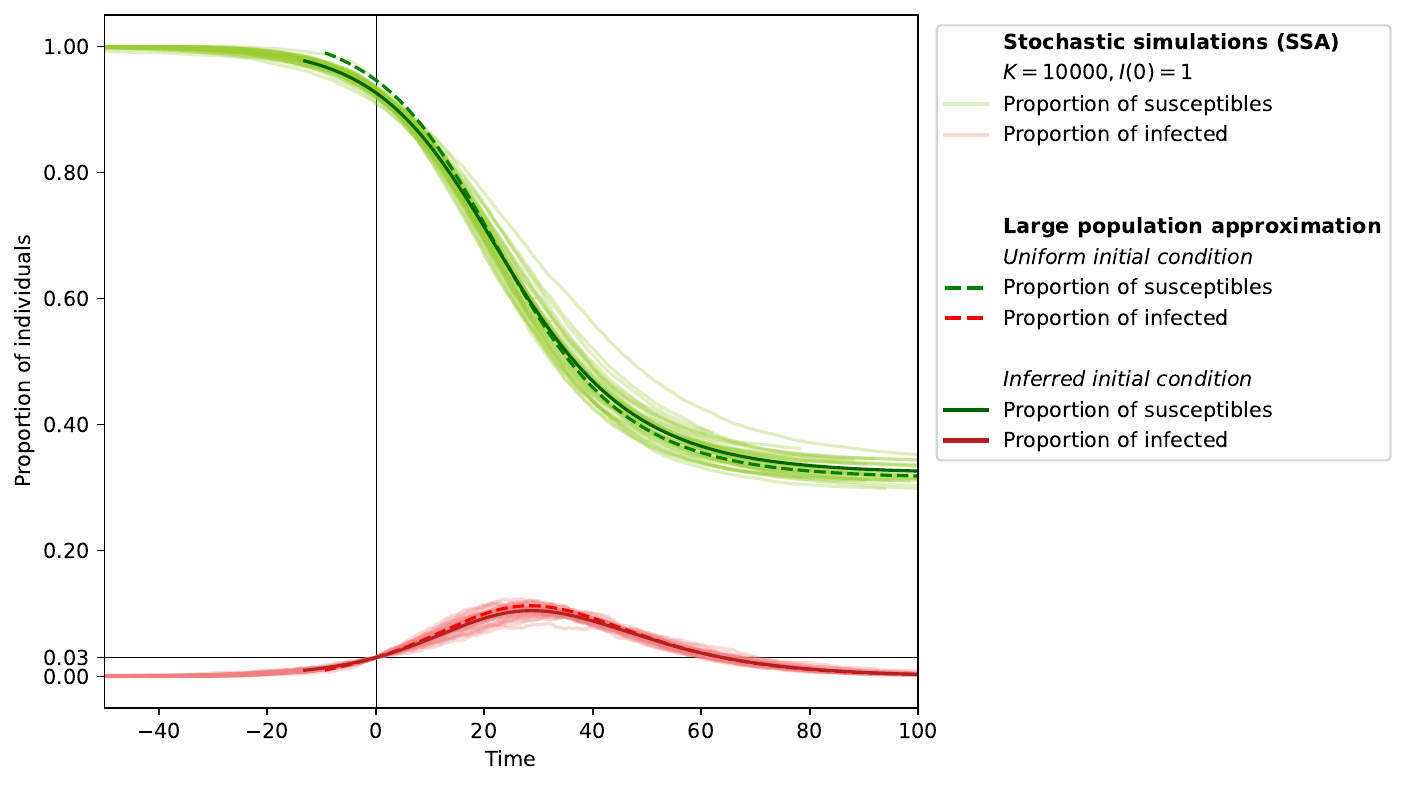}
    \caption{Comparison of the stochastic model starting from a single infected, with its large population approximation for two different choices of initial condition. The first initial condition corresponds to a uniform choice of $1\%$ of infected individuals at time $0$. The second initial condition is obtained by simulating a large number ($>2000$) of stochastic epidemic trajectories, starting from a single infected until the proportion of infected reaches one percent. The initial condition corresponds to the average of the observed simulation outcomes. 
    Regarding the stochastic model, for this figure, 100 epidemics starting from a single infected were simulated. Similarly to Figure \ref{fig:convergence}, only those reaching a threshold of 3\% of infected are represented, and a time shift is applied to ease comparison between model outputs.
    Structure size distributions are those of Figure \ref{fig:piX-insee}. Epidemic parameters: $(\beta_G, \lambda_H, \lambda_W, \gamma) = (0.085, 0.1, 0.001, 0.125)$, $R_0 = 1.7$.}
    \label{fig:y0}
\end{figure}

Finally, let us emphasize that by essence, the large population limit obtained here assumes that the number of infected individuals is of the same order as the population size. In a realistic scenario, however, an epidemic is initiated by very few infected individuals. In the case of a large epidemic outbreak, the number of infected subsequently grows until it no longer is negligible when compared to the population size, from which point on the large population limit captures the dynamics of the outbreak (see for instance \cite{bansayeSharpApproximationHitting2024} and references therein). This raises the question: which initial condition is pertinent for the large population approximation? 

For uniformly mixing population, this is rather straightforward, for two main reasons. On the one hand, at each time, infected individuals are interchangeable in terms of infectious pressure exerted on susceptibles. On the other hand, the presence of recovered individuals at time $t$ can be neglected in the study of the epidemic dynamics over the time interval $[t, \infty)$, simply by restricting the study to all other individuals, which still constitute a uniformly mixing population. As a consequence, in such a setting, it makes sense to suppose that at time zero, there are only infected and susceptible individuals, and that infected individuals are chosen uniformly at random in the population, with independent and identically distributed infectious period lengths. 

In the household-workplace model however, one actually needs to know how infected and recovered individuals are distributed among households and workplaces, meaning that neither can recovered be ignored, nor is there any reason to believe that infected individuals are distributed uniformly at random in the population. Indeed, Figure \ref{fig:y0} illustrates that when compared to stochastic simulations starting from a single infected, the large population approximation with initial condition corresponding to a uniform choice of infected individuals fails to reproduce the epidemic dynamics, as infected individuals are not uniformly distributed in the population when the epidemic becomes macroscopic. The epidemic dynamics are instead correctly captured when using an initial condition which is inferred from stochastic simulations. These capture the repartition of infected and recovered individuals among households and workplaces when the population starts from a single infected (\emph{i.e.} a negligible infected population) and subsequently becomes macroscopic. As a consequence, it appears crucial to get a better understanding of this realistic macroscopic initial condition, which arises from an epidemic started by a negligible number of infected. This may be achieved using a branching process approximation of the epidemic, which is designed to approach the initial, stochastic phase of the epidemic, and hence would represent a reduced model which complements the large population limit obtained in the present work.

%%%%%%%%%%%%%%%%%%%%% ANNEXE %%%%%%%%%%%%%%%%%%%%%%

\begin{appendix}

\section{Implementation of the large population limit}
\subsection{Automatic implementation of the dynamical system}
\label{apx:implementation}

It is possible to implement dynamical system \crefrange{eq:subeq-sysdyn-si}{eq:subeq-sysdyn-structures} in an automated way, in the sense that equations do not need to be written individually. 
The key lies in the fact that the set $\config$ can be constructed automatically, with an intrinsic organization of the states $(S,I)$ it contains. For example, one may arrange them by growing number $n$ of susceptible and infected members of the structure, and for each $n$, by growing number $i$ of infected, leading to 
\begin{equation*}
\config = \{(2,0), (1,1), \dots, (\nmax,0), (\nmax-1,1),\dots,(1,\nmax-1)\}.    
\end{equation*}
This in turn allows to make an explicit correspondence between any state $(S,I) \in \config$ and \emph{e.g.} some position in a vector containing all functions of our dynamical system of interest. A similar idea was already employed in \cite{pellisEpidemicGrowthRate2011}, for another purpose. With the previous structure of $\config$, one may for instance notice that for any $n \in \bbrackets{2}{\nmax}$ and $i \in \bbrackets{0}{n-1}$, the state $(n-i,i)$ is the $c(n-i,i)$-th state enumerated in $\config$, where $c(n-i,i) = (n-1)n/2 + i$.
As a consequence, the general expression of Equation \eqref{eq:subeq-sysdyn-structures} may be used to handle all the dynamics of the functions $n^X_{S,I}$, for $(S,I) \in \config$ and $X \in \{H,W\}$. 

Also, notice that in practice, household sizes tend not to be as big as workplace sizes. It thus makes sense to distinguish explicitly a maximal size for each type of structure. This allows to avoid implementing unnecessary equations corresponding \emph{e.g.} to household sizes that are not actually observed, and which thus artificially increase the dimension of the system.

\subsection{Computational performance}
\label{apx:runtime}

The aim of this section is to numerically assess the computational cost associated to solving the large dimensional dynamic system \crefrange{eq:subeq-sysdyn-si}{eq:subeq-sysdyn-structures} in comparison to stochastic simulations using Gillespie's algorithm, also referred to as SSA (stochastic simulation algorithm). In order to do so, the average execution times of one stochastic simulation (SSA) and of one resolution of the associated dynamic system using the ODE solver odeint from the scipy.integrate library are compared.

Let us start by describing the general procedure. Each of the two scripts (stochastic simulation or reduced model) is executed one hundred times, all runs being independent from one another. For each run and each script, the computation time of the script of interest is measured, as well as the computation time of a reference function (summing all integers up to one billion with a simple for-loop). The ratios of the runtimes of both the script of interest and the reference function are computed. Comparison of the computation times for the stochastic and the reduced model is then based on the comparison of the averages of those normalised runtimes. 

It remains to take an interest in the choice of the model parameters, namely the structure size distributions, the epidemic parameters \emph{i.e.} the contact rates $\beta_G$, $\lambda_H$, $\lambda_W$ and the removal rate $\gamma$, as well as the initial proportion of infected $\varepsilon$ and the time interval $[0,T]$ on which the epidemic is simulated. For the stochastic model, the population size $K$ will be fixed to ten thousand individuals. 
For all scenarios considered here, the structure size distributions will be those of Figure \ref{fig:piX-insee}, and the initial proportion of individuals will be set to $\varepsilon = 0.005$ in initial condition \eqref{eq:unif-init}. Different values of the epidemic parameters will be considered, as to obtain scenarios that differ both in terms of $R_0$ and in terms of the proportions of infections occurring within the general population, within households or within workplaces, respectively referred to as $p_G$, $p_H$ and $p_W$. The removal rate $\gamma$ will be fixed at $0.125$, and only the contact rates will effectively vary. In total, ten different scenarios will be used, characterized by their values of $R_0 \in \{1.2, 1.4, 1.7, 2.0, 2.5\}$ and $(p_G, p_H,p_W) \in \{(0.2, 0.4, 0.4), (0.4, 0.4, 0.2)\}$. 

Finally, parameter $T$ will be chosen as follows. For each set of epidemic parameters detailed above, the reduced model is used to compute the time $T_*$ at which the epidemic falls below one percent of infected individuals in the population, after the epidemic peak. $T$ then is determined by rounding down $T_*$ to the closest multiple of five. 

\begin{figure}[ht]
    \centering
    \includegraphics[width=0.85\textwidth]{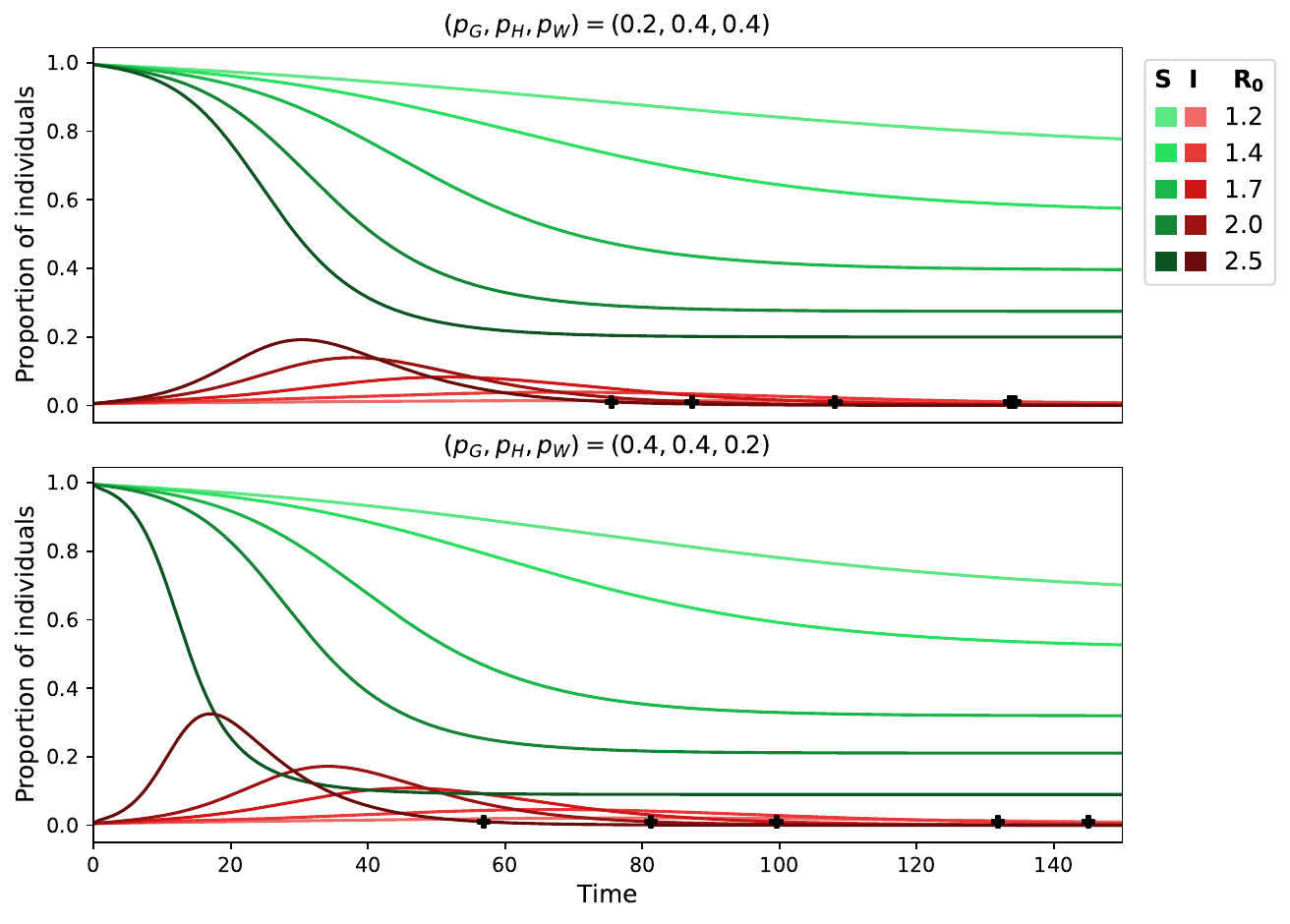}
    \caption{Proportion of susceptible (S) and infected (I) in the population, for each scenario detailed in Table \ref{tab:scenarios}, as given by dynamical system \crefrange{eq:subeq-sysdyn-si}{eq:subeq-sysdyn-structures}. Scenarios are separated by values of $(p_G, p_H, p_W)$ of infections per layer, namely $(0.2, 0.4, 0.4)$ and $(0.4, 0.4, 0.2)$ for the top and bottom panels, respectively. The corresponding values of $R_0$ are indicated by the color shades, as shown in the legend. The black crosses indicate for each curve that the proportion of infected falls below the threshold of one percent.}
    \label{fig:comp-scenarios}
\end{figure}

Figure \ref{fig:comp-scenarios} uses the reduced model to plot the trajectories of the proportion of susceptible and infected individuals in the population, for each scenario. The corresponding parameters are summarized in Table \ref{tab:scenarios}. Notice that in particular, this includes the parameters of Figure \ref{fig:convergence}.

\begin{table}[tb]
    \caption{Considered values of the contact rates and final times, grouped by value of $R_0$ and proportions of infections per layer $(p_G, p_H, p_W)$ characterizing the scenarios.}
    \label{tab:scenarios}
    \centering
    \begin{tabular}{c c c c c c c}
    \toprule
         & $R_0 = 1.2$ & $R_0 = 1.4$ & $R_0 = 1.7$ & $R_0 = 2.0$ & $R_0 = 2.5$ & $(p_G, p_H, p_W)$\\
        \midrule
        $\beta_G$ & 0.03 & 0.035 & 0.045 & 0.05 & 0.06 & \multirow{4}{6em}{(0.2, 0.4, 0.4)} \\
        $\lambda_H$ & 0.05 & 0.07 & 0.09 & 0.15 & 0.2 & \\
        $\lambda_W$ & 0.0015 & 0.0016 & 0.0018 & 0.002 & 0.0022 & \\
        $T$ & 130 & 130 & 105 & 85 & 75 & \\
        \midrule
        $\beta_G$ & 0.06 & 0.07 & 0.085 & 0.1 & 0.125 & \multirow{4}{6em}{(0.4, 0.4, 0.2)} \\
        $\lambda_H$ & 0.06 & 0.07 & 0.1 & 0.15 & 1.5 & \\
        $\lambda_W$ & 0.00075 & 0.0008 & 0.001 & 0.0011 & 0.00115 & \\
        $T$ & 145 & 130 & 95 & 80 & 55 & \\
    \bottomrule
    \end{tabular}
\end{table}

Let us now turn to the results. For each scenario of Table \ref{tab:scenarios}, measurement of average normalised computation times was repeated three times. The results are shown in Figure \ref{fig:runtime-ratios}, which indicates for each scenario the ratio of the average normalised runtime for one resolution of dynamical system \crefrange{eq:subeq-sysdyn-si}{eq:subeq-sysdyn-structures} over the average normalised runtime of one stochastic simulation. Let us first take an interest in the datasets labeled $(p_G, p_H, p_W) = (0.2, 0.4, 0.4)$ and $(p_G, p_H, p_W) = (0.4, 0.4, 0.2)$. One may notice first that for each scenario, the results of all three repeats are close to one another, indicating that the results are reproducible. Further, for both possible values of $(p_G, p_H, p_W)$, the results indicate a shared general trend. Indeed, for values of $R_0$ close to the critical case $R_0 = 1$, the ratio exceeds one, and diminishes subsequently, falling below one between $R_0 = 1.4$ and $R_0 = 1.7$ and attaining values of order $10^{-1}$. This behavior suggests that solving dynamic system \crefrange{eq:subeq-sysdyn-si}{eq:subeq-sysdyn-structures} is advantageous in terms of computation time for intermediate or high values of $R_0$, being up to one order of magnitude faster than one stochastic simulation. As the time interval $[0,T]$ on which the epidemic is studied originally depends on the scenario and is significantly shorter for larger values of $R_0$, one may wonder whether this difference influences the results. As a consequence, we have repeated the same procedure for all of the scenarios characterised by $(p_G, p_H, p_W) = (0.2, 0.4, 0.4)$, with fixed $T = 75$. Figure \ref{fig:runtime-ratios} shows that the associated results are very similar to those obtained previously, pleading against this hypothesis.

\begin{figure}[ht]
    \centering
    \includegraphics[width=0.8\textwidth]{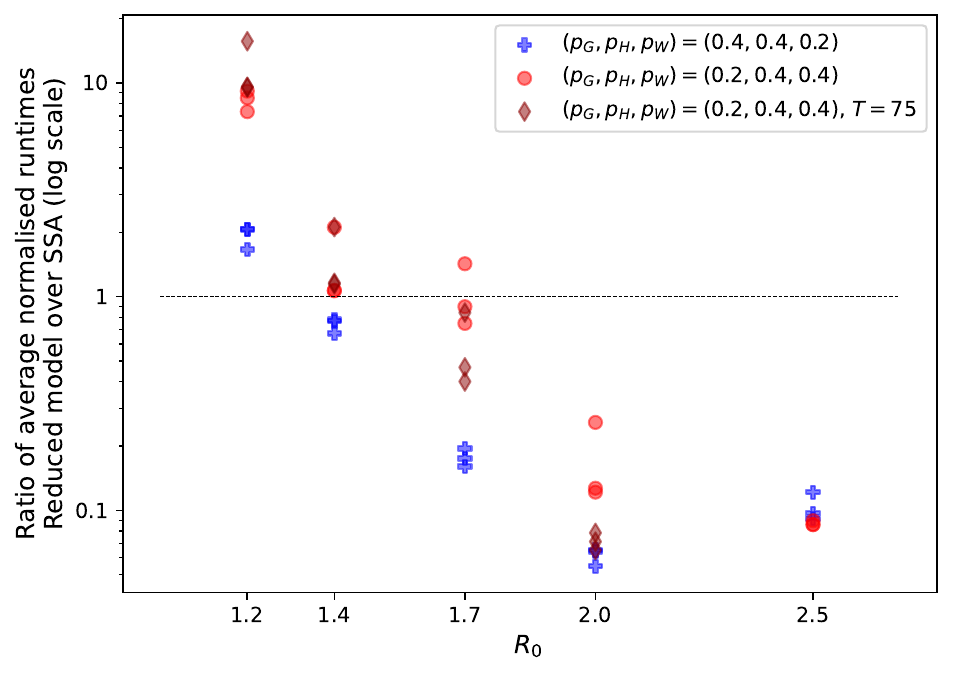}
    \caption{Ratio of the average normalised computation time for solving once dynamical system \crefrange{eq:subeq-sysdyn-si}{eq:subeq-sysdyn-structures} over the average normalised computation time for one stochastic simulation (SSA). This ratio was computed three times for each scenario of Table \ref{tab:scenarios}. The results are presented as a function of $R_0$, while colors indicate the value of $(p_G, p_H, p_W)$. Unless stated otherwise, the parameter $T$ from Table \ref{tab:scenarios} was used. The dotted line indicates the threshold of one.}
    \label{fig:runtime-ratios}
\end{figure}

\noindent Of course, this comparison could be pushed further. For instance, the most basic version of the SSA algorithm was used, and more advanced methods such as $\tau$-leaping are expected to accelerate stochastic simulations. Also, a more thorough exploration of the parameter space would be pertinent, assessing for instance the influence of the structure size distributions.

\section{Edge-based compartmental model}
\label{apx:ebcm} 

\subsection{Presentation of the EBCM}

\begin{figure}
    \centering
    \includegraphics[width=\textwidth]{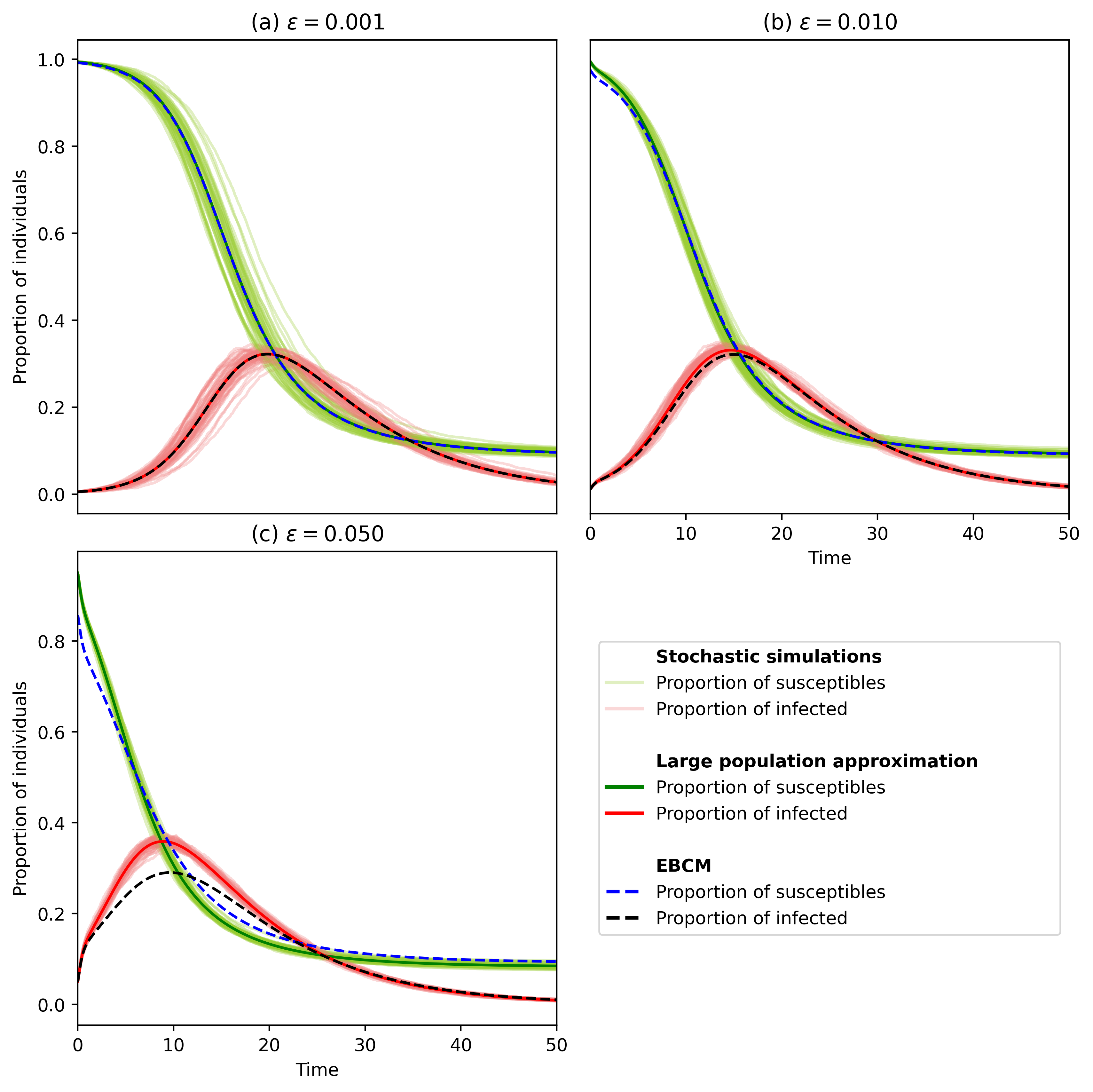}
    \caption{Comparison of the stochastic model with the large population approximation given by dynamical system \crefrange{eq:subeq-sysdyn-si}{eq:subeq-sysdyn-structures} and the corresponding EBCM. Household and workplace distributions are those of Figure \ref{fig:piX-insee}. Epidemic parameters are set to $(\beta_G, \lambda_H, \lambda_W, \gamma) = (0.125, 1.5, 0.00115, 0.125)$. Initial conditions correspond to $\varepsilon \in \{0.001, 0.01, 0.05\}$ as indicated for each panel. For each of these scenarios, Gillespie's algorithm is used to simulate $50$ trajectories of the stochastic model defined in Proposition \ref{def:zeta} in a population of $K = 10000$ individuals (faint lines). For Panel (a), only trajectories reaching a threshold proportion of $0.005$ infected are kept, and time is shifted so that time $0$ corresponds to the moment when this threshold is reached. 
    Finally, the deterministic solution $(s,i)$ of both dynamical system \crefrange{eq:subeq-sysdyn-si}{eq:subeq-sysdyn-structures} (thick lines) and the EBCM (dashed lines) are represented for each scenario. For Panel (a), the same time shifting procedure as for simulations is applied.}
    \label{fig:ebcm}
\end{figure}

Let us start by describing how to obtain the population structure of the local level of mixture described in Section \ref{sec:general-model} using a clique configuration model (CCM). In our case, each node belongs to exactly one clique within each layer (one household and one workplace, respectively). Let us briefly notice that whenever a node is picked uniformly at random, the probability of it belonging to a structure of type $X$ and size $n$ is given by $\hat{\pi}^X_n = n \pi^X_n/m_X$, for any $n \in \bbrackets{1}{\nmax}$ and $X \in \{H,W\}$. As a consequence, the layer corresponding to structures of type $X \in \{H,W\}$ is obtained by the following two steps. First, associate to each node a structure size distributed according to the size-biased law $\hat{\pi}^X$. This is done independently for each node. Second, for $k \in \bbrackets{1}{\nmax}$, form cliques of size $k$ by drawing uniformly without replacement $k$-tuples in the set of nodes of associated structure size $k$. This step stops when all nodes of associated clique size $k$ belong to a clique. 
This procedure is repeated independently for each layer, allowing to assemble households and workplaces.

Let us now turn to deriving the EBCM. Consider $s$ and $i$ the proportions of susceptible and infected individuals in the population, respectively. Let $\theta^X_n(t)$ for $X \in \{H,W\}$ and $n \in \bbrackets{1}{\nmax}$ be the chance of a susceptible belonging to a structure of type $X$ and size $n$ to escape infection within this structure, and $\theta^G(t)$ the chance of escaping infection through the mean-field level, up to time $t \geq 0$. The key idea is that a node is susceptible at time $t$ if and only if it has escaped infection up to time $t$, and the risks of infection within each layer are independent from one another. This makes use of properties of the CCM, which heuristically explain the decoupling of the risk of infection in the two local layers from one another. Further, the fact that in an infinite population, each individual structure has a negligible impact on the proportion of infected yields the intuition behind the decorrelation of the risks of infection at the local and global level. This leads to 
\begin{equation*}
    s = \theta^G  \prod_{X \in \{H,W\}} \left(\sum_{n=1}^{\nmax} \hat{\pi}^X_n \theta^X_n \right).
\end{equation*}
As we are considering an \emph{SIR} model, it follows that $i'(t) = -s'(t) + \gamma i(t)$, so that the difficulty resides in understanding the dynamics of $\theta^G$ and $\theta^X_n(t)$, for $X \in \{H,W\}$ and $n \in \bbrackets{1}{\nmax}$.

Define for $X \in \{H,W\}$ and $n \in \bbrackets{2}{\nmax}$:
\begin{equation*}
        m^X_n = \theta^G \hat{\pi}^X_n \theta^X_n \sum_{k = 1}^{n_{\max}}  \hat{\pi}^{\overline{X}}_{k} \theta^{\overline{X}}_k,
\end{equation*}
which corresponds to the proportion of individuals who are susceptible and belong to a structure of type $X$ and size $n$. Also, let $n^X_{(S,I,R)}$ be the proportion of susceptibles belonging to a structure of type $X$ containing exactly $S$ susceptibles, $I$ infected and $R$ recovered individuals. This allows us to introduce the following quantities, which participate in the rates at which a member of a structure of type $X$ and size $n$ is infected, either within the considered structure or outside of it, respectively: 
\begin{equation*}
    T^X_n = \lambda_X \sum_{\substack{(S,I,R) \in \N^3 \\ S+I+R = n}} SI\,n^X_{(S,I,R)} \quad \text{and} \quad
    \tau^X_n = \left(\beta_G i + \sum_{k=1}^{\nmax} T^{\overline{X}}_{k} \right) \frac{\hat{\pi}^X_n \theta^X_n}{\sum_{k=1}^{\nmax} \hat{\pi}^X_{k} \theta^X_{k}}.
\end{equation*}
One obtains the following dynamics: 
\begin{equation*}
    \frac{d}{dt} \theta^G =-\beta_G i \theta^G, \quad \text{and} \quad
    \forall X \in \{H,W\}, \forall n \in \bbrackets{2}{\nmax},\; \frac{d}{dt} \theta^X_n = -\frac{T^X_n}{m^X_n} \theta^X_n.
\end{equation*}
Further, for any $X \in \{H,W\}$, $n \in \bbrackets{2}{\nmax}$ and $(S,I,R) \in \N^3$ such that $S+I+R = n$ and either $S \geq 2$ or $SI \geq 1$:
\begin{equation*}
    \begin{aligned}
        \frac{d}{dt} n^X_{(S,I,R)} &= -\left(\lambda_X SI +\frac{\tau^X_n}{m^X_n} S + \gamma I\right) n^X_{(S,I,R)} \\
        &+ \gamma(I+1)n^X_{(S,I+1,R-1)} \setind{R \geq 1}\\
        &+ \left(\lambda_X(S+1)(I-1)  +\frac{\tau^X_n}{m^X_n}(S+1)\right) n^X_{(S+1,I-1,R)} \setind{I > 1}.
    \end{aligned}
\end{equation*}
Additionally, as in a structure of size one, no infection may occur within the structure itself, $\theta^H_1$ and $\theta^W_1$ are constant over time. Finally, it remains to define the initial conditions. Following Volz \emph{et al.} \cite{volzEffectsHeterogeneousClustered2011},  we consider the case $\varepsilon \ll 1$. Then  the only quantities which are not null at time zero are: for any $X \in \{H,W\}$, $n \in \bbrackets{1}{\nmax}$ and $I \in \bbrackets{1}{n-1}$, 
\begin{equation}
    \begin{aligned}
        i(0) &= \varepsilon, \\
        \theta^G(0) = \theta^X_n(0)& = 1 - \varepsilon, \\
        n^X_{(n,0,0)} &= \frac{1}{n} \hat{\pi}^X_n (1 - \varepsilon)^n, \\
        n^X_{(n-I,I,0)} &= \hat{\pi}^X_n (1 - \varepsilon)^{n-I} \varepsilon^I.
    \end{aligned}
\end{equation}

The proportions of susceptible and infected as predicted by both the EBCM and dynamical system \crefrange{eq:subeq-sysdyn-si}{eq:subeq-sysdyn-structures} are shown in Figure \ref{fig:ebcm}, for different values of $\varepsilon$. Let us first notice that in the case $\varepsilon=0.001$, corresponding to Panel (a) of Figure \ref{fig:ebcm}, the solutions $(s,i)$ of both the EBCM and dynamical system \crefrange{eq:subeq-sysdyn-si}{eq:subeq-sysdyn-structures} are in perfect accordance, emphasizing the fact that for very small values of $\varepsilon$, the EBCM seems to yield the correct asymptotic population dynamics. However, for larger values of $\varepsilon$, the EBCM struggles to reproduce these dynamics. The problem for capturing the epidemic dynamics for higher values of $\varepsilon$ lies in the fact that defining the proper initial condition for the EBCM is not straightforward, leading to initial conditions consisting in an approximation which is only sensible whenever $\varepsilon$ is very small.

\subsection{Computational performance}
\label{apx:ebcm-runtime}

In order to compare the computation times needed to solve either dynamical system \crefrange{eq:subeq-sysdyn-si}{eq:subeq-sysdyn-structures} with initial condition \eqref{eq:unif-init}, or the dynamical system associated to the EBCM which has been introduced above, we will proceed similarly as in \ref{apx:runtime}, making use of the ODE solver odeint from the scipy.integrate library in both cases. However, this time, only one parameter set will be used, corresponding to the parameters chosen for Panel (a) of Figure \ref{fig:convergence}. 
Further, the average normalised computation time is only computed once, instead of having three repeats as in \ref{apx:runtime}. Considering the relatively small fluctuations between repeats for all scenarios in Figure \ref{fig:runtime-ratios}, this is not expected to significantly affect the qualitative result. 

\begin{table}[tb]
    \caption{Numerical assessment of the computation time needed to solve either dynamical system \crefrange{eq:subeq-sysdyn-si}{eq:subeq-sysdyn-structures} or the EBCM introduced in \ref{apx:ebcm}. Model parameters: household and workplace size distribution from Figure \ref{fig:piX-insee}; $(\beta_G, \lambda_H, \lambda_W, \gamma) = (0.125, 1.5, 0.00115, 0.125)$; initial proportion of infected $\varepsilon = 0.005$; resolution of the numerical system over the time interval $[0,30]$.}
    \label{tab:runtime-ebcm}
    \centering
    \begin{tabular}{p{0.27\textwidth} p{0.1\textwidth}  p{0.14\textwidth} p{0.14\textwidth} p{0.14\textwidth}}
        \toprule
        & Runs & \multicolumn{3}{c}{Normalised runtimes} \\
        \cline{3-5}
        & & Average & Minimum & Maximum \\
        Dynamical system \crefrange{eq:subeq-sysdyn-si}{eq:subeq-sysdyn-structures} & 50 & 0.15 & 0.14 & 0.17 \\
        EBCM & 10 & 2076 & 1887 & 2254 \\
        \bottomrule
    \end{tabular}
\end{table}

The model parameters and the associated average runtimes are shown in Table \ref{tab:runtime-ebcm}. Due to the excessive computation needed to solve the EBCM, only 10 runs of this script were performed. However, considering that the average normalised runtime for solving the EBCM is several orders of magnitude higher than the average normalised runtime for solving dynamical system \crefrange{eq:subeq-sysdyn-si}{eq:subeq-sysdyn-structures}, this again is not expected to significantly alter the results. Finally, the computation times necessary for solving the EBCM are relatively homogeneous over all runs, indicating that the average computation time is not biased by an outlier.

\section{Proof of Proposition \ref{prop:sysdyn-sols}}
\label{apx:proofs}  

Let us start with the following lemma, which will be needed afterwards. 
\begin{lemma}
\label{lem:derivative-ineq}
    Consider a solution $y$ of dynamical system \crefrange{eq:subeq-sysdyn-si}{eq:subeq-sysdyn-structures} and let $\Delta(t) = m_X s(t) - \sum_{(S,I) \in \config} S n^X_{S,I}(t)$. Then 
    \begin{equation*}
        \frac{d}{dt} \Delta(t) = \gamma n^X_{1,1}(t) - \Big(\tau_G(t) + \frac{1}{s(t)}\tau_{\overline{X}}(t)\Big)\Delta(t).
    \end{equation*}
\end{lemma}

\begin{proof}[Proof of Lemma \ref{lem:derivative-ineq}.] Let $X \in \{H,W\}$. First, notice that
    \begin{equation*}
        \{(S+1,I-1): (S,I) \in \config, I \geq 1\} = \{(S,I) \in \config: S > 1\}.
    \end{equation*}
    As $S-1 = 0$ whenever $S=1$, we thus obtain that 
    \begin{equation}
    \label{eq:aux6}
        \begin{aligned}
            \sum_{(S,I) \in \config} & S^2 I n^X_{S,I} - \sum_{(S,I) \in \config} S(S+1)(I-1) n^X_{S+1,I-1}\setind{I \geq 1} \\
            & = \sum_{(S,I) \in \config} S^2 I n^X_{S,I} - \sum_{(S,I) \in \config} (S-1) S I n^X_{S,I}
             = \sum_{(S,I) \in \config} S I n^X_{S,I}.
        \end{aligned}
    \end{equation}
    Similarly, $\{(S,I+1): (S,I) \in \config, S + I < \nmax\} = \config \setminus \left\{\{(S,I) \in \config: I = 0\}\cup \{(1,1)\}\right\}$.
    As $SI = 0$ whenever $I=0$, it follows that
    \begin{equation}
    \label{eq:aux7}
        \sum_{(S,I) \in \config} \gamma SI n^X_{S,I} - \sum_{(S,I) \in \config} \gamma S(I+1) n^X_{S,I+1}\setind{S + I < \nmax} = \gamma n^X_{1,1}.
    \end{equation}
    The desired conclusion then results directly from Equations \crefrange{eq:subeq-sysdyn-si}{eq:subeq-sysdyn-structures}, regrouping the terms of the form of Equations \eqref{eq:aux6} and \eqref{eq:aux7} in order to simplify the expression.
\end{proof}    

We are now ready to focus on the desired result.

\begin{proof}[Proof of Proposition \ref{prop:sysdyn-sols}.]
    (i) By assumption, $y(0) \in V$. Let us start by checking that all components of $y$, as well as $\Delta$, stay non-negative over time.
    
    Let $t_0 \geq 0$ be such that $y(t_0) \in V$. If $i(t_0) = 0$, then $i'(t_0) = -s'(t_0) \geq 0$ by assumption, which ensures that $i$ will not become negative on a neighbourhood of $t_0$. Similar arguments hold for the lower bounds of $\Delta$ and $s$, using Lemma \ref{lem:derivative-ineq} and Inequality \eqref{eq:aux-ineq}, respectively.
    
    Let us now turn our attention to $n^X_{S,I}$ for $X \in \{H,W\}$ and $(S,I) \in \config$. Recall from Equation \eqref{eq:subeq-sysdyn-structures} that its derivative may be ill defined, due to the division by $s(t)$. However, inequality \eqref{eq:aux-ineq} ensures that the ratios $\tau_{\overline{X}} n^X_{S,I}/s$ are well defined at all time, for any $(S,I) \in \config$. As a consequence, we may now notice as previously that, if $n^X_{S,I}(t_0) = 0$, Equation \eqref{eq:subeq-sysdyn-structures} ensures that $\frac{d}{dt} n^X_{S,I}(t_0) \geq 0$. 
    
    The desired conclusion follows: whenever either of the quantities of interest reach zero, their derivatives are non-negative which ensures that they do not become negative shortly thereafter.
    
    Next,  let us have a look at the upper bounds. For $X \in \{H,W\}$, a brief computation yields $$\frac{d}{dt} \sum_{(S,I) \in \config} n^X_{S,I}(t) = -\gamma n^X_{1,1}(t) \leq 0.$$ This assures that starting from $y^* \in V$, the inequality $\sum_{(S,I) \in \config} n^X_{S,I}(t) \leq 1$ holds. For $X \in \{H,W\}$ and $(S,I) \in \config$, it follows that if $n^X_{S,I}(t_0) = 1$, then for any $(S',I') \in \config \setminus{(S,I)}$, $n^X_{S',I'}(t_0) = 0$. Thus 
    \begin{equation*}
        \frac{d}{dt}n^X_{S,I}(t_0) = -\left(\lambda_X SI + \tau_{\overline{X}}(t_0) \frac{S}{s(t_0)} + \tau_G(t_0) S + \gamma I\right) \leq 0,
    \end{equation*}
    from which one may deduce that $n^X_{S,I}$ remains less than or equal to one. 
    The remaining upper bounds on $s$, $i$ and $s+i$ may be obtained using similar arguments.
    
    We thus have established that if $y^* \in V$, then $y(t) \in V$ for all $t$ such that $y$ is well defined.
    \bigskip
    
    (ii) Consider any $T \geq 0$. Let $y = (s,i,n^X_{S,I}: X \in \{H,W\}, (S,I) \in \config)$ be a solution to the Cauchy problem \eqref{eq:cauchypb} with initial condition $y^* \in V$. Then it follows from Inequality \eqref{eq:aux-ineq} that 
    \begin{equation*}
        s'(t) \geq  -[(\lambda_H + \lambda_W)\nmax + \beta_G]s(t),
    \end{equation*}
    and as further $s \in \C^1(\R_+)$ according to the Cauchy problem, we obtain by comparison that $s(t) \geq s(0)\exp(-[(\lambda_H + \lambda_W)\nmax + \beta_G]t)$ for any $t \in [0,T]$. As a consequence,  on $[0,T]$, $s$ is bounded from below by 
    \begin{equation*}
        \epsilon_T = s(0) \exp(-[(\lambda_H + \lambda_W)\nmax + \beta_G]T).
    \end{equation*}
    
    In order to prove that there exists at most a unique solution $y$ for any initial condition $y^* \in V$, we will distinguish two cases. 
    
    First, if $s(0)=0$, then it follows that $s(t) = 0$ for any $t \geq 0$, and hence $n^X_{S,I}(t) = 0$ for any $t \geq 0$ and any $X \in \{H,W\}$, $(S,I) \in \config$. Subsequently, the equation for $i$ reduces to $i'(t) = -\gamma i(t)$ on $\R_+$, ensuring uniqueness of $y$ on $\R_+$.

    Second, if $s(0) > 0$, it follows from (i) and from our lower bound on $s$ over $[0,T]$ that $y(t) \in V_T = V \cap \{s \geq \epsilon_T\}$ for any $t\in [0,T]$. Our aim is to show that $f$ is Lipschitz continuous on $V_T$. Let $y = (s,i,n^X_{S,I}: X \in \{H,W\}, (S,I) \in \config)$ and $\hat{y} = (\hat{s},\hat{i},\hat{n}^X_{S,I}: X \in \{H,W\}, (S,I) \in \config)$ be two elements of $V$. First, consider $f_s$. Let $X \in \{H,W\}$, and define $c_X = \lambda_X \#\config (\nmax)^2/m_X$, then 
    \begin{equation*}
        \lvert \tau_X(y) - \tau_X(\hat{y}) \rvert \leq \frac{\lambda_X}{m_X} \sum_{(S,I) \in \config} SI\lvert n^X_{S,I} - \hat{n}^X_{S,I} \rvert \leq c_X \inftynorm{y - \hat{y}}.
    \end{equation*}
    It further holds that 
    \begin{equation*}
        \lvert \tau_G(y)s - \tau_G(\hat{y})\hat{s} \rvert \leq \beta_G(\lvert i \rvert \lvert s - \hat{s} \rvert + \lvert i - \hat{i} \rvert \lvert \hat{s} \rvert) \leq 2 \beta_G \inftynorm{y - \hat{y}}.
    \end{equation*}
    Thus, letting $c_s = c_H + c_W + 2\beta_G$, it follows that 
    \begin{equation*}
         \lvert f_s(y) - f_s(\hat{y}) \rvert \leq c_s \inftynorm{y - \hat{y}}.
    \end{equation*}
    Similarly, letting $c_i = c_s + \gamma$, 
    \begin{equation*}
        \lvert f_i(y) - f_i(\hat{y}) \rvert \leq \lvert f_s(y) - f_s(\hat{y}) \rvert + \gamma \lvert i - \hat{i} \rvert \leq c_i \inftynorm{y - \hat{y}}.    
    \end{equation*}
    It remains to focus on $f_{X,S,I}$. Proceeding as above, letting $c'_X = (2 \lambda_X \nmax + 4 \beta_G + \gamma)\nmax$, we obtain that 
    \begin{equation*}
    \begin{aligned}
        \lvert f_{X,S,I}(y) - f_{X,S,I}(\hat{y}) \rvert &\leq c'_X \inftynorm{y - \hat{y}} + \left\lvert \frac{\tau_{\overline{X}}(y)}{s}Sn^X_{S,I} -  \frac{\tau_{\overline{X}}(\hat{y})}{\hat{s}}S \hat{n}^X_{S,I}\right\rvert \\
        &+ \left\lvert \frac{\tau_{\overline{X}}(y)}{s}(S+1)n^X_{S+1,I-1} -  \frac{\tau_{\overline{X}}(\hat{y})}{\hat{s}}(S+1) \hat{n}^X_{S+1,I-1}\right\rvert \setind{I \geq 1}.
        \end{aligned}
    \end{equation*}
    Notice that as $y$ belongs to $V_T$, it follows that for $k_X = \lambda_X \nmax m_X$,
    \begin{equation*}
        (\tau_{\overline{X}}(y) S n^X_{S,I}, s) \in D_T = \{(x,y) : \epsilon_T \leq y \leq 1,\; 0 \leq x \leq k_X y^2\}.
    \end{equation*}
    Let $(x,y)$ and $(u,v)$ be two elements of $D_T$. It then holds that 
    \begin{equation*}
        \left\lvert \frac{x}{y} - \frac{u}{v} \right\rvert \leq \frac{1}{v}\left(\frac{x}{y} \lvert v - y \rvert + \lvert x - u \rvert\right) \leq \epsilon^{-1}_T (1 \vee k_X) (\lvert v - y \rvert + \lvert x - u \rvert).
    \end{equation*}
    Letting $k_{X,T} = \epsilon^{-1}_T (1 \vee k_X)(\lambda_X \nmax^2/m_X +1) \nmax$, we obtain for any $X \in \{H,W\}$ and $(S,I) \in \config$,
    \begin{equation*}
    \begin{aligned}
        \left\lvert \frac{\tau_{\overline{X}}(y)}{s}Sn^X_{S,I} -  \frac{\tau_{\overline{X}}(\hat{y})}{\hat{s}}S \hat{n}^X_{S,I} \right\rvert & \leq \epsilon^{-1}_T (1 \vee k_X) (\lvert \tau_{\overline{X}}(y) S n^X_{S,I} - \tau_{\overline{X}}(\hat{y}) S \hat{n}^X_{S,I} \rvert + \lvert s - \hat{s} \rvert) \\
        & \leq k_{X,T} \inftynorm{y - \hat{y}}.
    \end{aligned}
    \end{equation*}
    As a consequence, we conclude that 
    \begin{equation*}
        \lvert f_{X,S,I}(y) - f_{X,S,I}(\hat{y}) \rvert \leq (c'_X + k_{X,T}) \inftynorm{y - \hat{y}}.
    \end{equation*}
    This establishes the desired Lipschitz continuity of $f$ on $V_T$, with associated Lipschitz constant $c_T = \max(c_s, c_i, c'_H + k_{H,T}, c'_W + k_{W,T})$. 

    Suppose now that there are two solutions $y$ and $\hat{y}$ of Equations  \crefrange{eq:subeq-sysdyn-si}{eq:subeq-sysdyn-structures} such that $y(0) = \hat{y}(0)$. It then holds that
    \begin{equation*}
        \inftynorm{y(T) - \hat{y}(T)} \leq \int_0^T \inftynorm{ f(y(t)) - f(\hat{y}(t)) } dt \leq c_T \int_0^T \inftynorm{y(t) - \hat{y}(t) }.
    \end{equation*}
    Thus Gronwall's lemma ensures that $\inftynorm{y(t) - \hat{y}(t)} = 0$ for any $t \leq T$. The desired conclusion on uniqueness follows.
    \bigskip
    
    (iii) In order to establish (iii), it remains to notice that by Equation \eqref{eq:new-hyp-eta0}, the initial condition $y^*$ defined by Equation \eqref{eq:initialcdts} belongs to $V$. 
\end{proof}

\end{appendix}

%%%%%%%%%%%%%%%%%%%% ACKNWDGMT %%%%%%%%%%%%%%%%%%%%
\bigskip
\noindent {\bf Acknowledgments.} The author thanks the reviewers for their pertinent and detailed feedback on the
manuscript. The authors further thanks Vincent Bansaye for his continuous support during the elaboration of this work and detailed feedback on the manuscript, Viet Chi Tran for fruitful discussions on tightness criteria, as well as Frank Ball and Marie Doumic for pertinent discussions on related topics. Finally, the author is very grateful to the late Elisabeta Vergu for her guidance and stimulating remarks. She will be missed. 

%%%%%%%%%%%%%%%%%%%%% BIBLIO %%%%%%%%%%%%%%%%%%%%%%

%%%%%%%%%%%%%%%%%%%%%%% FIN %%%%%%%%%%%%%%%%%%%%%%%
\end{document}